\documentclass[a4paper,11pt]{article}
\usepackage{amsfonts, epsfig, amsmath, amssymb, color, amsthm, mathabx}
\usepackage{mathrsfs}
\usepackage{mathtools}
\usepackage{hyperref}
\usepackage[english]{babel}

\usepackage{xcolor}
\usepackage{bbm}

\makeatletter
\@addtoreset{equation}{section}
\makeatother
\newcommand{\be}{\begin{equation}}
\newcommand{\ee}{\end{equation}}
\newcommand{\bel}{\begin{equation}\label}

\newcommand{\ba}{\begin{equation}\begin{aligned}}
\newcommand{\ea}{\end{aligned}\end{equation}}

\newcommand{\mul}{\begin{multline}}
\newcommand{\emul}{\end{multline}}

\DeclareMathOperator{\1}{\mathbbm{1}}

\let\oldmarginpar\marginpar
\renewcommand{\marginpar}[1]{\oldmarginpar{\scriptsize\texttt{\color{red}{#1}}}}

\newcommand{\vf}{\varphi}
\newcommand{\ve}{\varepsilon}

\newcommand{\cF}{{\mathcal F}}
\newcommand{\cA}{{\mathcal A}}
\newcommand{\cC}{{\mathcal C}}

\newcommand{\wt}{\widetilde}

\newcommand{\mbN}{{\mathbb N}}

\newcommand{\mbR}{{\mathbb R}}

\newcommand{\Pb}{{\mathop{\rm P}}}

\newcommand{\limn}{\lim_{n\to\infty}}
\newcommand{\lime}{\lim_{\ve\to0+}}

\newcommand{\di}{\mathrm{d}}

\newcommand{\E}{{\mathop{\rm E}}}

\theoremstyle{plain}
\newtheorem{thm}{Theorem}[section]
\newtheorem{lem}[thm]{Lemma}

\newtheorem{corl}[thm]{Corollary}

\theoremstyle{definition}
\newtheorem{defn}[thm]{Definition}

\theoremstyle{remark}
\newtheorem{remk}[thm]{Remark}

\begin{document}
\title{Limit theorems for one-dimensional diffusions:  reflected, skew, sticky, skew-sticky, and snapping-out limits}

\author{
Andrey Pilipenko
\thanks{Section of Mathematics, University of Geneva, Switzerland}
\thanks{Institute of Mathematics of the National Academy of Sciences of Ukraine, Kyiv, Ukraine}
\and
Gilles Vilmart\footnotemark[1]
}

\date{}
\maketitle
 \vspace{-1em}
 
 \begin{abstract}
We establish general conditions for the weak convergence of regular one-dimensional diffusion processes when their scale functions, speed measures, state spaces and boundary conditions may substantially vary. 
 Our approach is based on the classical Itô--McKean representation of regular diffusions via time changes and scale transformations of a single Brownian motion. This reduces many stochastic convergence problems to deterministic questions concerning the convergence of integrals, monotone functions, compositions, and generalized inverses. 
 
 As a main application, we investigate various homogenization regimes for diffusions in media with closely spaced semipermeable membranes, where the resulting limiting processes include not only reflected and skew Brownian motions, but also skew-sticky Brownian motions and snapping-out diffusions. This yields a representation of a snapping-out diffusion as a function of a single Brownian motion, without introducing an auxiliary sequence of independent exponential random variables.
\end{abstract}

\noindent\textbf{Keywords:}
one-dimensional diffusions; weak convergence; scale functions and speed measures;
semipermeable membranes; skew and sticky Brownian motion; snapping-out diffusions. 

\noindent\textbf{2020 Mathematics Subject Classification:}
Primary: 60F17, 60J60;
Secondary: 60J50, 60J55, 60J65.

\section{Introduction}

Limit theorems for Markov processes represent a classical and thoroughly studied domain of stochastic analysis. When the coefficients of   stochastic differential equations (SDEs) are sufficiently regular—for instance, smooth and bounded—the continuous dependence of solutions on these coefficients is well-established. 

Standard intuition may fail for diffusion processes in media containing semi-permeable membranes or sharp interfaces.
To illustrate the subtleties that arise, let us consider two motivating examples. 
Suppose that, for each $n \geq 1$, the process $X_n$ 
 is a one-dimensional diffusion in a medium with several semi-permeable membranes and can be formally described as a solution to the following SDE: 
\begin{equation}\label{eq:SDE_intro_membranes}
\mathrm{d} X_n(t) = \mathrm{d} W(t) + \sum_{k} \gamma_{n,k} \, \mathrm{d} L^{X_n}_{\varepsilon_{n,k}}(t),
\end{equation}
where $W$ is a standard Brownian motion, $\varepsilon_{n,k}$ are the positions of the membranes, $\gamma_{n,k} \in (-1, 1)$ are the corresponding permeability parameters, and $L^{X_n}_y(t)$ denotes   the symmetric local time of $X_n$ at  $y$.

Assume that the membranes approach the origin as $n \to \infty$, that  is,   $\lim_{n\to\infty} \varepsilon_{n,k} = 0$ for all $k$. A natural  question arises: what is the weak limit of the sequence $X_n$? Furthermore, does this limit depend on the ratios of the distances between the membranes, their relative spatial ordering, or the structure of the permeability coefficients?
 
In a heuristic sense, the local time at a given point can be interpreted as the integral of the Dirac  delta function evaluated along the trajectory of the process. Consequently, if a cluster of membranes collapses to a single point (the origin), it is natural to expect the limiting process to be a skew Brownian motion satisfying the SDE
\begin{equation}\label{eq:SDE_intro_skewBM}
\mathrm{d} X_0(t) = \mathrm{d} W(t) + \gamma \, \mathrm{d} L^{X_0}_0(t),
\end{equation}
where $\gamma$ is the effective skewness parameter.
However, a naive linear intuition regarding the summation of individual drift intensities is incorrect. It is well known that \eqref{eq:SDE_intro_skewBM} has no solution if $|\gamma|> 1$; see \cite{Harrison+Shepp}.  As demonstrated in \S \ref{expl: two membranes skew BM}, the actual relation defining the effective parameter $\gamma$ is nonlinear, but is independent of the spatial configuration of the layer.

Consider another classical example originating from the seminal work of Portenko~\cite{portenko1976generalized, Portenko90}. Instead of membranes, suppose that the singular drift  $\alpha \delta_0$ is regularized  by a sequence of regular drifts $a_n(x)=n \, a(nx),$ where $a\in L_1(\mbR)$ and $\alpha=\int_\mbR a(x) \di x$.
It was established by Portenko  in a more general setting 
(see  also \S \ref{expl:Portenko}),  that, as $ n \to \infty$, the corresponding  diffusions satisfying  $ \di X_n(t) = a_n(X_n(t)) \di t +\di W(t)$ converge   weakly to a skew Brownian motion, but with the effective parameter given by $\gamma = \tanh(\alpha)$.
This non-linear transformation clearly illustrates   difficulties that may arise while considering singular diffusions.

Stochastic methods for studying questions of existence, uniqueness, and convergence of solutions to one-dimensional SDEs involving local time have been developed in great generality; see, e.g., \cite{LeGall83, EngelbertSchmidt1985, makhno2016onedim-homogen, hairer2010periodic}. The connection between SDEs involving local time and parabolic equations
with transmission conditions is studied in detail in
\cite{etoreMartinez2018}.
 For regular diffusions on open intervals, necessary and sufficient conditions for convergence were established by  Freidlin and Wentzell~\cite{freidlin-wentzell_1994Necessary_1dim} in terms of   scale functions and speed measures. Notably, Freidlin and Wentzell allowed the state spaces of the pre-limiting and limiting processes to be different.  The approach in \cite{freidlin-wentzell_1994Necessary_1dim}   relies on the analysis of the resolvents of the processes, for which explicit formulas can be derived via solutions of certain Sturm--Liouville-type problems. It was extended in  \cite{kim2009continuous}  to diffusions on a fixed half-line or a fixed bounded interval, incorporating boundary conditions into the convergence analysis. More recently, Criens~\cite{criens2023relation} studied weak convergence of regular diffusions on natural scale on a fixed interval in terms of convergence of their speed measures, including general boundary behavior.
 In contrast to the setting of \cite{kim2009continuous, criens2023relation}, our state spaces themselves need not be fixed; for example, the pre-limit diffusions may live on the whole real line, whereas the limiting diffusion lives on a half-line. Thus, a reflecting or absorbing boundary of the limiting process need not correspond to any boundary of the approximating processes.

In this paper, we use the classical Itô--McKean representation of one-dimensional diffusions as time-changed Brownian motions. This representation was successfully employed by Stone in   \cite{stone1963limit}   and later by Brooks and Chacon \cite{brooks1982weak} in their study of weak convergence of diffusions, scale functions, speed measures and time changes. 
   Namely, a one-dimensional diffusion $X$ can be represented as $X(t) = s^{-1}(B(A^{-1}(t)))$, where $s$ is the scale function of $X$ and $A(t)$ is a continuous additive functional of a Brownian motion $B$ expressed as an integral of the local time of $B$ with respect to the speed measure of the diffusion $X$. By embedding the entire sequence of diffusions into a single probability space via the representation $X_n(t) = s_n^{-1}(B(A_n^{-1}(t)))$, driven by the same underlying Brownian motion $B$, we are able to establish almost sure convergence of $(X_n)$ to $X$ provided that their scale functions $(s_n)$ and speed measures $(m_n)$ converge in an appropriate sense to the scale function and speed measure of $X$. 
  Consequently, for each fixed   $\omega$, the stochastic problem is largely reduced to a series of deterministic analytical problems  concerning the convergence of (generalized) inverse functions $(A_n^{-1})$, integrals of the Brownian local time against varying speed measures, and compositions of converging mappings. This pathwise framework relies only on the fundamental sample-path properties of Brownian local time (such as joint continuity and   support properties) and avoids complex analytic machinery, thereby substantially simplifying the proofs.  
 The same framework can also be applied in singular regimes where the state space and the boundary behavior change in the limit. Such extensions, however, require additional arguments.
For example, in the presence of reflecting boundaries, the additive functionals may have flat intervals, so that their generalized inverses have jumps. Consequently, convergence of the additive functionals does not by itself imply convergence after the corresponding time changes.
 
A separate part of the paper is devoted to   snapping-out diffusions. These processes   are natural generalizations of the snapping-out Brownian motion (SNOB), first introduced and studied by analytic methods in \cite{bobrowski2012pde, lejay2016snapping}    and   by stochastic methods in \cite{mandrekar2016brownian}; see also the recent paper \cite{erhard2025general}. A snapping-out diffusion has a hard membrane, located, for instance, at 0. While the process is on the positive or negative side of the membrane, it evolves as a diffusion \(X_+\) or \(X_-\), respectively, with reflection   at 0 into the corresponding half-line  (possibly with sticky reflection), until its local time at \(0\) reaches an exponentially distributed random level, at which point it switches to the other side.

We prove that snapping-out diffusions arise as limits of regular diffusions on the real line under simple conditions formulated in terms of scale functions and speed measures. Our results cover  particular cases considered in \cite{lejay2016snapping, mandrekar2016brownian}. We also obtain a representation of a snapping-out diffusion in terms of a single Brownian motion, without introducing an auxiliary sequence of independent exponentially distributed random variables.  This representation is particularly convenient for proving limit theorems. For example, it allows us to obtain skew Brownian motion as a limit of snapping-out Brownian motions with increasing snapping intensities.

Our methods may also be extended to diffusion processes on graphs by representing them as space-time transformations of Walsh processes; see \cite{anagnostakis2025general}. Related analytic approaches to homogenization problems for diffusions on graphs can be found in \cite{ bobrowski2026portenko, bobrowski2025snapping, bobrowski_ratajczyk_pilipenko2026}; see also the recent paper \cite{koralov2026criterionconvergencediffusionprocesses}, which combines analytic methods with martingale-problem techniques.

 In multidimensional settings, the homogenization of diffusions in periodic or layered media is a nontrivial problem that has been addressed in specific cases using various methods; see, e.g.,  \cite{bensoussan2011asymptotic, chechkin2007homogenization, jikov2012homogenization, hairer2008homogenization,  pavliotis2008multiscale, pardoux1999homogenization}. 
    Naturally, pathwise one-dimensional methods cannot be directly extended to higher dimensions.  Nevertheless, the explicit formulas and structural mechanisms established in our work may provide useful intuition for the effective boundary and interface behavior that can arise in higher-dimensional models; see   \cite{ aryasova2025hard, aryasova2024homogenization, aryasova2025limit, bressloff2023renewal, bobrowski2020modeling, bobrowski2025existence, schummBressloff2023}. Related asymptotic analyses of coupling conditions across thin fractures
can be found in \cite{gander2021modeling}.

This paper is organized as follows. In Section \ref{sec:Intro}, we recall  (a) the construction of a one-dimensional regular diffusion via space-time transformation of a   Brownian motion, (b) the characterization of its infinitesimal generator, stochastic interpretation of parabolic PDEs with  boundary conditions, and (c)  several classic examples, such as reflected and skew Brownian motions, sticky diffusions, among others. 
Section \ref{sec:Main results} is devoted to the formulation and proofs of our main convergence theorems. Subsection \ref{sec:open}  is introductory in nature and includes  warm-up calculations. In particular, we consider models where the limiting   diffusion has natural boundary  or  models where we investigate diffusions upon the exit from an interval. The results are close to \cite{stone1963limit, brooks1982weak} and show the main approach and difficulties that may arise.  Results where the limiting diffusion has absorbing or reflecting boundary are studied in \S \ref{sec:absorbing} and \ref{sec:reflecting}. 
We  introduce  a snapping   out diffusion 
in \S \ref{sec:SNOD}. In   Theorem \ref{thm:limit_SNOdiffosion}, we obtain the snapping-out diffusion as a limit of regular diffusions on the real line. A representation of a snapping-out diffusion in terms of a single Brownian motion is given in Theorem \ref{thm:SNO_representation}.  

In Section \ref{sec:expl of application}, we discuss various  applications of our results to homogenization problems. Here we show that numerous limiting results, some of which were proved by more involved methods, follow  as corollaries of our unified approach. We illustrate how sticky, reflecting or snapping-out behavior  may appear  in  limits     of diffusions with  drifts  approximating a Dirac delta, as well as  in   limits of diffusions with several interacting interfaces separated by  small distances. The examples also show qualitative transitions between different limiting regimes depending on the relative scaling of the membrane distances and permeability parameters.

Finally, auxiliary deterministic results concerning the convergence of monotone functions and generalized inverse functions, as well as the vague convergence of measures, are collected in the Appendix.

\section{Construction of One-Dimensional Diffusions}\label{sec:Intro}
\subsection{Space-time  transformation of a Brownian motion}

Let $(X, \Pb_x, x\in I)$ be a continuous one-dimensional strong Markov process taking values in an interval $I\subset \mbR$, which may be bounded or unbounded, and may or may not contain its endpoints. We will refer to $X$ as a diffusion on $I$. We assume throughout this paper that $X$ is an honest process, i.e.,
\[
\Pb_x\big(X(t)\in I \text{ for all } t\geq 0\big) = 1, \quad x\in I,
\] 
where $\Pb_x$ denotes the law of the process $X$ started from $x\in I$.

Introduce the stopping times $\sigma_y:=\inf\{t > 0 : X(t)=y\}$, $y\in I$, and assume that $X$ is regular, meaning that
\begin{equation}
    \label{eq:def_regular}
    \Pb_x(\sigma_y<\infty)>0 \quad \text{for all } x\in \mathrm{int}(I), \, y\in I.
\end{equation}
For example, if $I=[\alpha, \beta)$, then any point in $[\alpha, \beta)$ is accessible in finite time with positive probability from any initial point $x\in(\alpha,\beta)$. However, we do not assume that $\Pb_\alpha(\sigma_y<\infty)>0$ for $y\in(\alpha,\beta)$; the process $X$ may be absorbed  at $\alpha$.

It is well known (see, e.g., \cite{rogers-williams-vol2, Revuz+Yor}) that there exists a strictly increasing continuous function $s:I\to\mbR$, called the scale function, such that for all $y_1<x<y_2$ in $I$,
\ba\label{eq:scale}
\Pb_x(\sigma_{y_1}<\sigma_{y_2}) = 1-\Pb_x(\sigma_{y_1}>\sigma_{y_2}) = \frac{s(y_2)-s(x)}{s(y_2)-s(y_1)}.  
\ea

It is clear that the process $\wt X(t)=s(X(t))$, $t\geq0$, is also a regular diffusion on $s(I)$, and $\wt s(x)=x$ is its scale function. Diffusions whose scale functions are affine, i.e., $s(x)=ax+b$ with $a>0,$ are called \textit{diffusions in   natural scale}.  

The next fundamental result states that any regular diffusion is a  space-time  transformation of a Brownian motion (see \cite[pp.~277--279]{rogers-williams-vol2} and \cite[Theorem~20.9, p.~379]{Kallenberg}). Recall that the symmetric local time of a Brownian motion $B$ is defined by
\begin{equation}
    \label{eq:int_LT}
    l_B(t,y)=\lime\frac{1}{2\ve}\int_0^t\1_{\{|B(z)-y|<\ve\}} \di z, \quad   \text{a.s.}
\end{equation}
Moreover, there exists a modification of $l_B(t,y)$ that is continuous in $(t,y)$, and we use this modification throughout the paper. It is well known that $l_B$ is the density of the occupation measure of the Brownian motion, satisfying the occupation times formula
\begin{equation}
\label{eq:loctimeint}
    \int_0^t f(B(z)) \di z = \int_\mbR f(y) \, l_B(t,y) \di y, \quad t\geq 0, \ \text{a.s.}    
\end{equation}
for any bounded measurable function $f$.

\begin{thm}\label{thm:representation_diff_natural}
Let $\wt X$ be a regular diffusion in   natural scale on $\wt I,$ started from $\wt x \in \wt I$. Then there exist a measure $\wt m$ on the Borel $\sigma$-algebra of $\wt I$ satisfying
\begin{equation}
    \label{eq:speed_m_natural}
    0 < \wt m([\wt y_1, \wt y_2]) < \infty \quad \text{for all } \wt y_1, \wt y_2 \in \mathrm{int}(\wt I) \text{ with } \wt y_1 < \wt y_2,
\end{equation}
and a Brownian motion $B$ started from $\wt x$ and defined on some extension of the probability space $(\Omega,\cF,\Pb_{\wt x})$, such that $\wt X$ has the representation
\begin{equation}
    \label{eq:reg_diff_natural}
    \wt X(t) = B(A^{-1}(t)), \quad t \geq 0, \ \text{a.s.,}
\end{equation}
where
\begin{equation}
    \label{eq:reg_diff_CAF_natural}
    A(t) := \int_{\wt I} l_B(t,\wt y) \, \wt m (\di \wt y) \quad \text{and} \quad A^{-1}(t) := \inf\{s \geq 0 : A(s) > t\}.
\end{equation}
\end{thm}
 
\begin{remk}
   In their seminal book \cite{ItoMcKean65}, Itô and McKean use a definition of Brownian local time that includes a factor 1/2. This difference is compensated for by a corresponding factor of two in the speed measure. Normalization conventions vary across standard textbooks. Throughout this paper, we follow the convention used in Theorem~\ref{thm:representation_diff_natural} and adjust the formulas accordingly.
\end{remk}
\begin{remk}
The formulation of Theorem~\ref{thm:representation_diff_natural} is adapted from \cite{rogers-williams-vol2}.  In \cite{rogers-williams-vol2}, as in many other monographs, the process is assumed to start from an arbitrary fixed point $\wt x$. However, the assumption of a deterministic initial point is not essential for the structural representation. The same representation holds for an arbitrary initial distribution. More precisely, for any regular diffusion
in  natural scale $\wt X$ on $\wt I$ and any initial distribution $\wt \mu$, there exists a Brownian motion $B$ defined on an extension of the probability space $(\Omega,\cF,\Pb_{\wt\mu})$ with initial distribution $\wt \mu$ such that \eqref{eq:reg_diff_natural} and \eqref{eq:reg_diff_CAF_natural} hold.
\end{remk}

If $X$ is a regular diffusion on $I$ with scale function $s$, setting $\wt X = s(X)$ yields the representation $X = s^{-1} \big( B(A^{-1}) \big)$ via Theorem~\ref{thm:representation_diff_natural}. To obtain a more intuitive representation, it is preferable to express the integral in \eqref{eq:reg_diff_CAF_natural} over the original state space $I$ rather than $\wt I$. To this end, we define a measure $m$ on the Borel subsets of $I$ by the relation $m(A) = \wt m(s(A))$. Note that $s$ and $s^{-1}$ are measurable maps by continuity; hence, $s(A)=(s^{-1})^{-1}(A)$ is a measurable set; measures $\wt m = m \circ s^{-1}$ and $m = \wt m \circ (s^{-1})^{-1}$ are push-forward measures. Applying the change of variables formula, we obtain the following result.

\begin{thm}\label{thm:representation_diff}
Let $X$ be a regular diffusion on $I$ with  scale function $s$, started from $x \in I$. Then there exist a measure $m$ on the Borel $\sigma$-algebra of $I$ satisfying
\begin{equation}
    \label{eq:speed_m}
    0 < m([y_1,y_2]) < \infty \quad \text{for all } y_1, y_2 \in \mathrm{int}(I) \text{ with } y_1 < y_2,
\end{equation}
and a Brownian motion $B$ started from $s(x)$ and defined on some extension of the probability space $(\Omega,\cF,\Pb_x)$, such that $X$ has the representation
\begin{equation}
    \label{eq:reg_diff}
    X(t) = s^{-1}\big(B(A^{-1}(t))\big), \quad t \geq 0, \ \text{a.s.,}
\end{equation}
where
\begin{equation}
    \label{eq:reg_diff_CAF}
    A(t) := \int_{s(I)} l_B(t,y) \, m\circ s^{-1}(\di y) = \int_{I} l_B(t,s(y)) \, m (\di y),
\end{equation}
and $A^{-1}(t) := \inf\{s \geq 0 : A(s) > t\}$.
\end{thm}

The measure $m$ from Theorem~\ref{thm:representation_diff} is called the \textit{speed measure} of $X$.

It is well known (see, e.g., \cite{rogers-williams-vol2, Revuz+Yor}) that the speed measure of a regular diffusion satisfies the relation 
\ba\label{eq:speed}
\E_x \int_0^{\sigma_{y_1}\wedge\sigma_{y_2}} f(X(t))\,\di t = 2\int_{(y_1,y_2)} G(x,y) f(y) \, m(\di y),  
\ea
for any non-negative measurable function $f$, where the Green function $G$ is  symmetric, and is given, for $ y_1\leq y\leq x\leq  y_2,$ by the formula (see \cite[pp.~303--304]{Revuz+Yor})  
\[
G(x,y) = \frac{(s(y)-s(y_1))(s(y_2)-s(x))}{s(y_2)-s(y_1)}, \quad y_1 \leq y \leq x \leq y_2.
\]

Equation~\eqref{eq:scale} determines the function $s$ uniquely up to a nondegenerate affine transformation. If $ s, \hat s$ are  scale functions and $ m, \hat m$ are the corresponding speed measures for the same diffusion $X$, then there exist constants $a>0$ and $b\in \mbR$ such that
\[
\hat s(x) = a s(x) + b \quad \text{and} \quad \hat m(\di x) =  a^{-1} m(\di x).
\]
Furthermore, once the scale function $s$ is fixed, the restriction of $m$ to $\mathrm{int}(I)$ is uniquely determined by \eqref{eq:speed}.

Before we continue to discuss examples and applications, it is necessary to make a few important remarks regarding the boundary behavior of diffusions.

The additive functional $A(t)$ may be infinite for some $t > 0$. A necessary condition is that $m$   assigns infinite mass to every neighborhood of one of the endpoints of $I$ and that $s$ is finite at that endpoint. Let $t_0 := \inf\{s \geq 0 : A(s) = +\infty\} < \infty$. It can be shown that in this case $\lim_{t\to\infty} X(t)$ exists and is equal to $\alpha$ or $\beta$. Moreover, if this point is accessible in finite time (i.e., is included in $I$), then $X$ is absorbed there upon its first hitting.  For a rigorous classification and characterization of boundaries, we refer the reader to \cite{ItoMcKean65, rogers-williams-vol2, Ito_StochProc, knight}; see Remark \ref{remk:classification exit}  below.    

The following properties of $A$ are well known. We include a proof to make the paper self-contained.
\begin{lem}
    \label{lem:properties of A}
    Let $A(t) :=  \int_{s(I)} l_B(t,y) \, \wt m (\di y)$. Set  $\zeta:=\inf\{t\geq 0 : A(t)=+\infty\}\in[0,\infty]$. Then, almost surely, 
    $A$ is continuous and nondecreasing on $[0,\zeta) $. Moreover,  
    $A$ is strictly increasing on every interval $[a,b]\subset [0,\zeta) $   on which
    $B(t)\in \mathrm{int}(s(I)), t\in[a,b].$
   \end{lem}
\begin{proof}
    The monotonicity of $A$ follows from the fact that $t\to l_B(t,y)$ is nondecresing for any $y$. 
    The continuity of $A$ on $[0,\zeta)$ follows from the continuity and monotonicity of the map $t\to l_B(t,y)$, together with the monotone and dominated convergence theorems. 

    It is well known that $\Pb(l_B(t_1,B(t_0))>l_B(t_0,B(t_0)) )=1$ for any fixed $t_1>t_0\geq 0$.
    Since $l_B$ is jointly continuous,
    \[
    \Pb(\exists \ve, c>0\ :\ l_B(t_1,y)-l_B(t_0,y)\geq c,\quad |y-B(t_0)|\leq \ve )=1.
    \]
    So, for almost every $\omega$ such that $B(t_0)\in \mathrm{int}(s(I)),$ we have the inequality
    $$A(t_1)-A(t_0)\geq c \wt m(\{y : \quad |y-B(t_0)|\leq \ve_0\})>0$$
    for every sufficiently small $\ve_0>0$ such that $[B(t_0)-\ve_0,B(t_0)+\ve_0]\subset  s(I)$. Here we used the fact that $\wt m$ assigns positive mass to every nonempty open subinterval of $\mathrm{int}(s(I))$. Taking the intersection over all rational $t_0<t_1$, we obtain an event of probability 1 on which $A(t_1)-A(t_0)>0$ holds simultaneously for all rational $t_0<t_1$ such that $ t_0 \in [0,\zeta).$
    Let $[a,b]\subset [0,\zeta) $  be an interval on which
    $B(t)\in \mathrm{int}(s(I)), t\in[a,b].$ Choosing rational numbers  $a<t_0<t_1<b$ we obtain $A(b)-A(a)\geq A(t_1)-A(t_0)>0$. This proves the second assertion of the lemma.
\end{proof}

A remarkable fact is that the converse of Theorem~\ref{thm:representation_diff} also holds, meaning that formula \eqref{eq:reg_diff} defines a continuous  strong Markov process. The only delicate aspect concerns the behavior of $X$ at the endpoints of $I$. We first formulate this inverse result for a diffusion $\wt X$ in  natural scale;  the general case follows as a straightforward corollary. In the general setting, some caution is needed if $I$ is unbounded and the process can reach  $+\infty$ or $-\infty$   in finite time (which is impossible when a diffusion is on a natural scale).

\begin{thm}\label{thm:reg_diff_inv_natural}
Let $\hat I \subset [-\infty,\infty]$ be an interval with endpoints $-\infty\leq \wt \alpha < \wt \beta\leq\infty$, and let $\wt m$ be a measure on $\hat I$ satisfying \eqref{eq:speed_m_natural}  {for all } $\wt y_1, \wt y_2 \in \mathrm{int}(\hat  I) $. 
Let $\wt\gamma\in(\wt\alpha,\wt\beta)$ be arbitrary. Define the state space $\wt I \subset \hat I$, with $\hat I \setminus \wt I \subset \{\wt \alpha, \wt \beta\}$, according to the following criteria:
\begin{itemize}
    \item the point $\wt \alpha$ is included in $\wt I$ if and only if $\wt \alpha \in \mbR$ and 
    \[
    \int_{(\wt \alpha, \wt \gamma)} (x - \wt \alpha) \, \wt m(\di x) < \infty,
    \]
      (in this case, the process $\wt X$ defined below hits $\wt \alpha$ from $\mathrm{int}(\wt I)$ in finite time with positive probability);
    
    \item the point $\wt \beta$ is included in $\wt I$ if and only if $\wt \beta \in \mbR$ and 
    \[
    \int_{(\wt \gamma, \wt \beta)} (\wt \beta - x) \, \wt m(\di x) < \infty
    \]
    (in this case, the process $\wt X$ hits $\wt \beta$ from $\mathrm{int}(\wt I)$ in finite time with positive probability).
\end{itemize} 
Let $B$ be a Brownian motion started in $\wt I$. Then the process $\wt X$ defined by \eqref{eq:reg_diff_natural} and \eqref{eq:reg_diff_CAF_natural} is a regular diffusion on $\wt I$.
\end{thm}

For the proof, we refer to \cite[pp.~277--278]{rogers-williams-vol2} or \cite{ItoMcKean65}.

\begin{remk} Convergence or divergence of the integrals is independent of the choice of the reference point $\wt \gamma \in (\wt \alpha, \wt \beta)$.
Furthermore, if $\wt \alpha = -\infty$ or $\wt \beta = +\infty$, the integrals 
\[
\int_{(\wt \alpha, \wt \gamma)} (x - \wt \alpha) \, \wt m(\di x) \quad \text{and} \quad \int_{(\wt \gamma, \wt \beta)} (\wt \beta - x) \, \wt m(\di x)
\]
are infinite, respectively. Thus, the explicit requirements $\wt \alpha \in \mbR$ and $\wt \beta \in \mbR$ in the criteria for inclusion in $\wt I$ can be safely omitted, as the boundary points $\pm \infty$ can never be included in the state space.  \end{remk}

\begin{remk}
It is well known that if $K$ is a continuous additive functional of a strong Markov process $Y$, then the time-changed process $Y \circ K^{-1}$ also possesses the strong Markov property. Consequently, the process $\wt X$ defined by \eqref{eq:reg_diff_natural} is strong Markov. The sample-path continuity of $\wt X$ is immediate when $\wt I$ does not contain endpoints (in which case $A$ is almost surely continuous, strictly increasing, and unbounded) and can be rigorously established by elementary   arguments  from analysis when endpoints are included (see \cite[\S 5.3]{ItoMcKean65}). Therefore, the fact that $\wt X$ is a regular diffusion is standard; the nontrivial content of Theorem~\ref{thm:reg_diff_inv_natural} lies  in the precise characterization of which boundary points belong to the state space.
\end{remk}

Let $\wt X$ be a diffusion in   natural scale with speed measure $\wt m$ on an interval $\wt I \subset \mbR$, and let $s: I \to \wt I$ be a continuous  strictly increasing  bijection. Then $X(t) := s^{-1}(\wt X(t))$ is clearly a regular diffusion on $I$. Consequently, any process possessing the representation \eqref{eq:reg_diff}--\eqref{eq:reg_diff_CAF}, where $m = \wt m \circ s$, is a regular diffusion on $I$ with scale function $s$ and speed measure $m$.

In general, it is not entirely accurate to assert that for any strictly increasing continuous function $s: I \to \mbR$ and any measure $m$ on $I$ satisfying \eqref{eq:speed_m}, formulas \eqref{eq:reg_diff} and \eqref{eq:reg_diff_CAF} automatically define a regular diffusion on $I$ because the boundary points require separate treatment.  A subtle problem   arises if $I$ is unbounded and the process can reach, for instance, $+\infty$ in finite time with positive probability. This scenario can occur only if the limit $s(+\infty) := \lim_{x \to +\infty} s(x)$ is finite, which implies that $s(+\infty)$ must be included in the state space of $\wt X$. To address this particular situation, in the following theorem we consider processes taking values in the extended real line $[-\infty, \infty]$, equipped with the usual topology. In principle, the process can exit $+\infty$ after the hitting. The following result is an immediate consequence of Theorem  \ref{thm:reg_diff_inv_natural}.

\begin{thm}\label{thm:reg_diff_inv}
Let $\hat I = [\alpha, \beta] \subset [-\infty, \infty]$ be an interval, let $s: \hat I \to [-\infty, \infty]$ be a continuous, strictly increasing function, and let $m$ be a measure on $\hat I$ satisfying \eqref{eq:speed_m} for all $y_1, y_2 \in \mathrm{int}(\hat I)$. Define the state space $  I \subset \hat I$, with $\hat I \setminus   I \subset \{\alpha, \beta\}$, according to the following criteria:
\begin{itemize}
    \item the point $\alpha$ is included in $  I$ if and only if $s(\alpha) \in \mbR$ and 
    \[
    \int_{(\alpha, \gamma)} (s(x) - s(\alpha)) \, m(\di x) < \infty,
    \]
    where $\gamma \in (\alpha, \beta)$ is an arbitrary point (in this case, the process $X$  defined below hits $\alpha$ from $\mathrm{int}(I)$ in finite time with positive probability);
    
    \item the point $\beta$ is included in $  I$ if and only if $s(\beta) \in \mbR$ and 
    \[
    \int_{(\gamma, \beta)} (s(\beta) - s(x)) \, m(\di x) < \infty
    \]
    (in this case, the process $X$ hits $\beta$ from $\mathrm{int}(I)$ in finite time with positive probability).
\end{itemize} 
Let $B$ be a Brownian motion started in $s(  I)$. Then the process $X$ defined by \eqref{eq:reg_diff} and \eqref{eq:reg_diff_CAF} is a regular diffusion on $ I$.
\end{thm}

\begin{remk}\label{remk:classification exit} 
  The following characterization of the behavior of $X$ at $\alpha$ can be found in \cite{ItoMcKean65, Ito_StochProc, rogers-williams-vol2, Karlin_Taylor1981second}. The characterization at the right endpoint $\beta \in I$ is analogous.
\begin{itemize}
    \item If $s(\alpha) \in \mbR$ and $m([\alpha, \gamma]) < \infty$, then the process $X$ reaches $\alpha$ from $\mathrm{int}(I)$ with positive probability and enters $\mathrm{int}(I)$, when started from $\alpha$, with probability $1$. In this case, the point $\alpha$ is called a \textit{regular boundary}.

    \item If $s(\alpha) \in \mbR$, $m((\alpha, \gamma]) < \infty$, and $m(\{\alpha\}) = \infty$, then the process $X$ reaches $\alpha$ from $\mathrm{int}(I)$ with positive probability and absorbed there upon hitting it. In this case $\alpha$ is an \textit{absorbing point}.
    
    \item If $s(\alpha) \in \mbR$, $\int_{(\alpha, \gamma)} (s(x) - s(\alpha)) \, m(\di x) < \infty$, and $m((\alpha, \gamma]) = \infty$, then the point $\alpha$ is an \textit{exit boundary} (and thus an absorbing point); that is, $X$ reaches $\alpha$ from $\mathrm{int}(I)$ with positive probability but cannot enter $\mathrm{int}(I)$ from $\alpha$.

    \item If $s(\alpha) \in \mbR$ and $\int_{(\alpha, \gamma)} (s(x) - s(\alpha)) \, m(\di x) = +\infty$, then $X$ cannot hit $\alpha$ in finite time (hence $\alpha \notin I$), but $\Pb_x(\lim_{t \to \infty} X(t) = \alpha) > 0$ for all $x \in \mathrm{int}(I)$, and $\lim_{x \to \alpha+} \Pb_x(\lim_{t \to \infty} X(t) = \alpha) = 1$.

    \item If $s(\alpha+) = -\infty$, then $X$ does not hit $\alpha$ in finite time, $\alpha\notin  I$, and $\Pb_x(\lim_{t \to \infty} X(t) = \alpha) = 0$ for all $x \in \mathrm{int}(I)$.
\end{itemize}
\end{remk}

\begin{remk}\label{remk:local_time}
The local time of the diffusion $X$ is defined as 
\begin{equation}
    \label{eq:int_loc_time_X0}
    L_X(t,y) := l_B(A^{-1}(t),s(y)), \quad t \geq 0,
\end{equation}
for all $y \in I$ that are not traps, and in particular for all $y \in \mathrm{int}(I)$; see \cite[\S 5.4]{ItoMcKean65}.

Almost surely, the local time $L_X(t,y)$ is jointly continuous in  $(t, y) \in [0,\infty) \times \mathrm{int}(I)$. Indeed, the functions $(t,x) \mapsto l_B(t,x)$, $x \mapsto s(x)$, and $t \mapsto A(t)$ are almost surely continuous. Moreover, $A$ is strictly increasing whenever the Brownian motion $B$ belongs to $\mathrm{int}(s(I))$; see Lemma \ref{lem:properties of A}. Consequently, the inverse $A^{-1}$ is continuous at those time instants when $X$ is in $\mathrm{int}(I)$. At a jump of $A^{-1}$, the corresponding Brownian path segment lies outside $\mathrm{int}(s(I))$, and hence the Brownian local time at interior points does not change during this segment. Together with joint continuity of $l_B$, this yields  the joint continuity of   $L_X$ on  $[0,\infty)\times \mathrm{int}(I)$. 

It follows from \eqref{eq:loctimeint} and standard real analysis arguments that
\begin{equation}\label{eq:int_loc_time_X}
    \Pb\left( \int_0^t f(X(z)) \, \di z = \int_I f(y) L_X(t,y) \, m (\di y) \text{ for all } t \geq 0 \right) = 1
\end{equation}
for any bounded measurable function $f$ supported on the set of non-trap points; see \cite[\S 5.4]{ItoMcKean65} and \cite[Thm~49.1, Ch.~V.7]{rogers-williams-vol2}. 

The continuity of $L_X$ and relation \eqref{eq:int_loc_time_X} imply that 
\begin{equation}\label{eq:loc time = limit}
\begin{aligned}
L_X(t,y)
&= \lim_{\ve \to 0+} \frac{\int_0^t \1_{\{|X(z)-y| \leq \ve\}} \, \di z}{m([y-\ve,y+\ve])} \\
&= \lim_{\ve \to 0+} \frac{\int_0^t \1_{\{X(z) \in [y, y+\ve]\}} \, \di z}{m([y, y+\ve])} \\
&= \lim_{\ve \to 0+} \frac{\int_0^t \1_{\{X(z) \in [y-\ve,y]\}} \, \di z}{m([y-\ve,y])}
\end{aligned}
\end{equation}
for every $t \geq 0$ and $y \in \mathrm{int}(I)$. Hence, $L_X$ is indeed a local time in the standard sense; namely, it is a continuous additive functional of $X$, and $t \mapsto L_X(t,y)$ increases only when $X$ visits $y$. If $y \in I$ is a non-trap  left or right endpoint of $I$, then $L_X(t,y)$ can be expressed as $\lim_{\ve \to 0+} \frac{\int_0^t \1_{\{X(z) \in [y,y+\ve)\}} \, \di z}{m([y, y+\ve))}$ or $\lim_{\ve \to 0+} \frac{\int_0^t \1_{\{X(z) \in [y-\ve,y]\}} \, \di z}{m([y-\ve,y])}$, respectively (see \cite[\S 5.4]{ItoMcKean65}). Although the path continuity of $t \mapsto L_X(t,y)$ at the endpoints was not explicitly formulated there, it can be verified directly from the steps of their proof. To make the paper self-contained we include the arguments for the case when  the left endpoint $\alpha$ is included in  $I$ and is not a trap.  Since $A$ is continuous and nondecreasing, the function $A^{-1}$ is a càdlàg increasing function. Hence   $L_X$ is continuous at  $(t,\alpha)$ whenever $t$   is not a jump of $A^{-1}$. It remains  to verify  that on a set of complete probability we have
$$
\lim_{n\to\infty} l_B(A^{-1}(t_n),s(x_n)) = l_B(A^{-1}(t),s(\alpha))
$$
for any $t>0$ that is a point of jump of $A^{-1}$  and any sequences $(t_n)\subset [0,\infty), (x_n)\subset I $ such that $\limn t_n=t, \quad \limn  x_n=\alpha.$ By the right-continuity of $A^{-1}$, it suffices to consider only $t_n<t.$

Since $t$ is a point of jump of $A^{-1}$, there is a nondegenerate maximal interval $[z_-, z_+]\subset (0,\infty)$ such that $A(z)=t, z\in(z_-,z_+).$  Then $B(z)\notin s(I), z\in(z_-,z_+)$, $B$ visits $\mathrm{int}(s(I))$ in any neighborhood of $z_\pm,$ and $A^{-1}(t-)=z_-, \quad A^{-1}(t)=z_+.$  The local time of a Brownian motion  $l_B$ is jointly continuous. Thus 
$\lim_{n\to\infty} l_B(A^{-1}(t_n),s(x_n)) = l_B(z_-,s(\alpha))$. Since $B(z)\neq s(\alpha) , z\in(z_-,z_+)$ we have equality of local times $l_B(z_-,s(\alpha))=l_B(z_+,s(\alpha))$, 
which completes the proof of continuity of $L_X.$

\end{remk}

\begin{remk}
  
Assume that $\alpha \in I$ is the left endpoint of $I$ and is not a trap. Then we can obtain an analogue of the representations \eqref{eq:reg_diff} and \eqref{eq:reg_diff_CAF}, where the   Brownian motion $B$ is replaced by a Brownian motion $B_+$ reflected at the point $s(\alpha)$:
\begin{equation}
    \label{eq:reg_diff_refl}
    X(t) = s^{-1}\big(B_+(A_+^{-1}(t))\big), \quad t \geq 0. 
\end{equation}
Here
\begin{equation}
    \label{eq:reg_diff_CAF_refl}
    A_+(t) := \int_{s(I)} l_{B_+}(t,y) \, m\circ s^{-1}(\di y) = \int_{I} l_{B_+}(t,s(y)) \, m (\di y),
\end{equation}
and $A_+^{-1}(t) := \inf\{s \geq 0 : A_+(s) > t\}$. The local time $l_{B_+}(t,u)$ is defined by the formula $\lim_{\ve \to 0+} \frac{1}{\ve} \int_0^t \1_{\{B_+(z) \in [u, u +\ve)\}} \, \di z$. Formulas \eqref{eq:int_loc_time_X0} and \eqref{eq:int_loc_time_X} remain valid if we replace $B$ and $A$ with $B_+$ and $A_+$. For the proof, we refer to \cite[\S 5.5]{ItoMcKean65}, which establishes a deep result concerning the chain rule and the subordination of diffusions. Note that the symmetric local time coincides with both the right and left local times for every $y \in \mathrm{int}(I)$. The setting involving the right endpoint of $I$, or both endpoints simultaneously, is entirely analogous.
\end{remk}

\subsection{A functional approach: generators and  boundary conditions}

Let $X$ be a regular diffusion on $I$, and let $\mathcal A$ denote its generator. Recall that, for $f\in\mathrm{Dom}(\mathcal A)$, the function
\[
u(t,x):=\E_x f(X(t))
\]
satisfies the Kolmogorov backward equation
\[
u_t=\mathcal A u,\qquad t>0,
\]
with initial condition $u(0,x)=f(x)$.

Thus, in order to describe the corresponding PDE, one needs to identify both the action of the generator and its domain. The former determines the differential operator in the interior of the state space, while the latter determines the boundary and interface conditions. 

In terms of the scale function and the speed measure, the generator has the classical form\footnote{Sometimes the generator $\cA$ is written without the factor of $1/2$, as seen in the monograph by It\^o and McKean \cite{ItoMcKean65}. This convention plays no role as long as the corresponding adjustments are made to the speed measure $m$ and the definition of the local time.}
\[
\mathcal{A}f = \frac{1}{2} D_m D_{s} f.
\]
The notion of a generator is somewhat delicate, since it is formulated in different ways in the literature.   We  explain this classical formula, with particular emphasis on the endpoints of $I$ and on the atoms of $m$. The rigorous definition of \(\mathcal A\) will not be needed in our proofs. We include this subsection to clarify the connection between diffusions and the corresponding PDEs with boundary conditions.

 There are several standard approaches to defining the generator, which, under appropriate conditions, lead to equivalent characterizations of the law of \(X\).  Among the standard approaches are: (a)   fixing  a suitable Banach space and constructing Markov semigroups, (b) defining a generator via a  martingale problem, (c) using Dynkin's characteristic operator, among   others. For example, Itô and McKean  started with resolvents $ R_\lambda f(x):=\E_x\int_0^\infty e^{-\lambda s} f(X(s)) \di s, \lambda>0$ defined on the set of bounded  continuous functions. Then they defined $\cA $ as the (unbounded) operator such that $(\lambda -\cA)^{-1}=R_\lambda$. Subsequently, they associated $\cA$ with Dynkin's operator and with ``weak'' infinitesimal operator. 

In all cases, it is well known that $P_t f(x) := \E_x f(X(t)), t\geq 0$ maps bounded continuous functions to bounded continuous functions; see \cite[Proposition 50.1]{rogers-williams-vol2}. However,   $X$ may not be a Feller--Dynkin process, i.e., $(P_t)$ may  not always map $C_0(I)$ into itself. This  happens if and only if  at least one of the boundary points is an entrance point; see \cite{criens2023feller}. Introduce the space $\bar{C}(I)$ of continuous functions on $I$ that have finite limits at the endpoints of $I$ (which is automatic if  the endpoints  belong to $I$), equipped with the norm $\|f\| := \sup_{x \in I} |f(x)|$.  It can be shown  that  $(P_t)_{ t\geq 0}$, restricted to   $\bar{C}(I)$, is  a strongly continuous semigroup.

The  generator of the semigroup, $\mathcal{A}: \mathrm{Dom}(\mathcal{A}) \subset \bar{C}(I) \to \bar{C}(I)$,  is defined by the standard formula $\cA f =\lim_{t\to0+} \frac{P_t f-f}{t}, $ where $\mathrm{Dom}(\mathcal{A})$ consists of $f\in \bar{C}(I)$  such that the limit exists in the supremum norm.

\begin{remk}\label{remk: relation Kolmogorov and weak convergence}
The representation of the solution to the Kolmogorov backward equation
$$
u(t,x)=\E_x f(X(t))
$$
also relates weak convergence of Markov processes to convergence of the corresponding solutions of the Kolmogorov backward equations.
Indeed, suppose that
\[
u_n(t,x)=\E_x f_n(X_n(t)), \qquad n\geq0,
\]
where $x$ belongs to the state spaces $E_n$ of the processes $X_n$. Assume also that all $E_n$ are contained in a
common metric space $E$,  all  $f_n$, $n\geq0$, are bounded
continuous functions defined  on the space $E$, and  that 
 $\limn \sup_{x\in E}|f_n(x)-f_0(x)|=0.$ 
 
 Let $x_n\in E_n, n\geq 0$, with $\limn x_n=x_0$, and suppose that, for a fixed
$t\geq0$, we have convergence in distribution $X_n(t)
\Rightarrow
X_0(t)$ when  $X_n(0)=x_n, n\geq 0.$
Then   $\limn u_n(t,x_n)= u_0(t,x_0).$
 \end{remk}

It can be shown (see \cite{ItoMcKean65, Ito_StochProc, Revuz+Yor}) that if $f \in \mathrm{Dom}(\mathcal{A})$, then for any $x \in \mathrm{int}(I)$ the following derivatives exist:
\begin{gather*}
D_s^+ f(x) = \lim_{\ve \to 0+} \frac{f(x+\ve) - f(x)}{s(x+\ve) - s(x)}, \\[2mm]
D_s^- f(x) = \lim_{\ve \to 0+} \frac{f(x-\ve) - f(x)}{s(x-\ve) - s(x)}, \\[2mm]
D_m D_s^+ f(x) = \lim_{\ve', \ve'' \to 0+} \frac{D_s^+ f(x+\ve'') - D_s^+ f(x-\ve')}{m([x-\ve', x+\ve''])}, \\[2mm]
D_m D_s^- f(x) = \lim_{\ve', \ve'' \to 0+} \frac{D_s^- f(x+\ve'') - D_s^- f(x-\ve')}{m([x-\ve', x+\ve''])},
\end{gather*}
and satisfy
\[
\mathcal{A}f(x) = \frac{1}{2} D_m D_{s}^\pm f(x), \quad x \in \mathrm{int}(I).
\]
Moreover, if $m(\{x\}) = 0$ (see \cite[Chapter~VII, Theorem~3.12]{Revuz+Yor}), then we have equality 
$$D_s^+ f(x) = D_s^- f(x) =: D_s f(x),$$
and the fundamental relation
\[
D_s f(x_2) - D_s f(x_1) = 2 \int_{x_1}^{x_2} \mathcal{A}f(y) \, m(\di y)
\]
holds for any $x_1, x_2 \in \mathrm{int}(I)$ provided that $m(\{x_1\}) = m(\{x_2\}) = 0$. On the other hand, if $m(\{x\}) > 0$, then the one-sided derivatives satisfy the jump condition
\begin{equation}
    \label{eq:boundary condition with atom two-sided}
    D_s^+ f(x) - D_s^- f(x) = 2 \, m(\{x\}) \mathcal{A}f(x).
\end{equation}

The domain of the generator \cite{ItoMcKean65} equals the set of functions $f \in \bar{C}(I)$ such that:
\begin{enumerate}
    \item[\rm (a)] $D_m D_s^+ f(x)$ and  $D_m D_s^- f(x)$ are well-defined and coincide  for all $x \in \mathrm{int}(I)$,
    \item[\rm (b)] $\frac{1}{2} D_m D_s^\pm f$  belongs to   $ \bar{C}(I)$, 
    \item[\rm (c)] $f$ satisfies certain boundary conditions at those endpoints of   $I$ that belong to $I$.
\end{enumerate}

We present the boundary conditions only for the left endpoint $\alpha$. The right endpoint is treated similarly. If $\alpha \in I$ and $m([\alpha, \alpha+\ve]) < \infty$ for small $\ve > 0$, then 
\begin{equation}
    \label{eq:boundary_condition_general}
    2 \, m(\{\alpha\}) \mathcal{A}f(\alpha) = D_s^+ f(\alpha).
\end{equation}
In all other cases, the process $X$ cannot exit from $\alpha$ into $\mathrm{int}(I)$, and $\mathcal{A}f(\alpha) = 0$.

\subsection{Examples: Brownian motion with semipermeable and sticky interfaces}
In this section, we   present several well-known examples of diffusions, their scale functions, speed measures, local times, and various stochastic representations. 
\subsubsection{Brownian motion}  
The identity function $s(x) = x$ and the Lebesgue measure $m(\di x) = \di x$ are a scale function and a speed measure of a Brownian motion on $I = \mbR$. Note that $s_c(x) = c x$ and $m_c(\di x) = c^{-1} \di x$ are also a scale function and a speed measure of a Brownian motion for any $c > 0$. The domain of $\cA$ consists of $u\in C^2(\mbR)$ such that $u, u''\in \bar C(\mbR).$ Moreover, $\cA u =\frac{1}{2}u'' $ for  such   $u.$

\subsubsection{Reflected Brownian motion}
  Let $W$ be a   Brownian motion with $W(0) \geq 0$ a.s. The process $X(t) = |W(t)|$, $t \geq 0$, is called a reflected Brownian motion. This is an example  of a regular diffusion on $I = [0,\infty)$  with a scale function $s(x) = x$ for $x \geq 0$ and speed measure $m(\di x) = \1_{\{x \geq 0\}} \di x$. Note that $s_c(x) = c x$ and $m_c(\di x) = \1_{\{x \geq 0\}} c^{-1} \di x$ are also scale functions and speed measures of a reflected Brownian motion for any $c > 0$. 
The domain of $\cA$ consists of $u\in C^2([0,\infty))$ such that $u, u''\in \bar C([0,\infty)) $ that satisfy boundary condition  \eqref{eq:boundary_condition_general},   which in this case is the Neumann
boundary condition $u'(0+)=0.$  Moreover, $\cA u =\frac{1}{2}u'' $ for  such   $u.$

It follows from Theorem \ref{thm:representation_diff} and the occupation times formula \eqref{eq:loctimeint} that there exists a Brownian motion $B$ such that $X(t) = B(A_+^{-1}(t))$, $t \geq 0$, where $A_+(t) := \int_0^t \1_{\{B(z) \geq 0\}} \, \di z$.  

Let
\[
l_W(t,y) := \lim_{\ve \to 0+} \frac{1}{2\ve} \int_0^t \1_{\{|W(z)-y| \leq \ve\}} \, \di z
\]
and
\[
l_{|W|}(t,y) := \lim_{\ve \to 0+} \frac{1}{2\ve} \int_0^t \1_{\{||W(z)| - y |\leq  \ve\}} \, \di z
\]
be the symmetric local times of $W$ and $|W|$, respectively. It is easy to verify that almost surely
\[
l_{|W|}(t,y) = \begin{cases}
    l_W(t,y) + l_W(t,-y), & t \geq 0, \, y > 0; \\
    l_W(t,0), & t \geq 0, \, y = 0.
\end{cases}
\]
On the other hand, relation \eqref{eq:int_loc_time_X0} yields $L_{|W|}(t,y) = l_B(A_+^{-1}(t), y)$ for $t > 0$ and $y \geq 0$, which is continuous in $y\in [0,\infty)$. Note that $l_{|W|}(t,y) = L_{|W|}(t,y)$ for $y > 0$, but $L_{|W|}(t,0) = 2 l_{|W|}(t,0)$. This corresponds to Remark~\ref{remk:local_time} since
\[
L_{|W|}(t,0) := \lim_{\ve \to 0+} \frac{1}{\ve} \int_0^t \1_{\{|W(z)| \in [0, \ve]\}} \, \di z.
\]

\subsubsection{Space transformation of a regular diffusion}\label{expl:transformationSM}
Let $X$ be a regular diffusion on $I$ with scale function $s$ and speed measure $m$, and let $f: I \to J$ be a continuous and increasing bijection. It is clear that $Y(t) := f(X(t))$ is a regular diffusion on $J$. It follows from \eqref{eq:scale},   \eqref{eq:reg_diff}, and \eqref{eq:reg_diff_CAF}  that its scale function and speed measure are given by $s_Y = s \circ f^{-1}$ and $m_Y := m \circ f^{-1}$, respectively. Here, $s \circ f^{-1}$ denotes the composition of $s$ with the inverse function $f^{-1}$, and $m \circ f^{-1}$ is the image (push-forward) of the measure $m$ under the map $f$; that is, $m \circ f^{-1}(A) = m(\{x \in I : f(x) \in A\})$.

In particular, if $m(\di x) = \rho(x) \di x$ possesses a density $\rho$, and $f$ is continuously differentiable with $f'>0$, then the speed measure of $Y$ satisfies $m_Y(\di y) = \frac{\rho(f^{-1}(y))}{f'(f^{-1}(y))} \, \di y$.

\subsubsection{Scale function and speed measure of a solution to a stochastic differential equation}\label{subsec:SDEScale}
Consider a stochastic differential equation (SDE)
\begin{equation*}
\di X(t) = a(X(t)) \, \di t + b(X(t)) \, \di W(t), 
\end{equation*}
where $a$ and $b$ are continuous functions with $b(x) \neq 0$ for all $x \in \mbR$. Then the SDE has a  weak solution, unique in law, up to the explosive time. If, for any starting point, the solution does not explode almost surely (for instance, if $a$ and $b$ are bounded), then this solution is automatically a regular diffusion on $I = \mbR$. Any function $s$ and measure $m$ that satisfy the relations
\begin{equation*}
s'(x) = \exp\left\{-2\Phi(x)\right\}, 
\quad
m(\di x) = m'(x) \, \di x,
\end{equation*}
where 
\begin{equation*}
\Phi'(x) = \frac{a(x)}{b^2(x)}, 
\quad
m'(x) = \frac{\exp\left\{ 2\Phi(x)\right\}}{b^2(x)}, \quad x\in\mbR.
\end{equation*}
are a scale function   and a  speed measure of $X;$ see \cite[page 297]{rogers-williams-vol2}.

\subsubsection{Skew Brownian motion} \label{expl:skewBM}
Recall that a skew Brownian motion with parameter $\gamma \in [-1,1]$ is defined as the unique solution to the SDE
\begin{equation}\label{eq:skew}
\di X(t) = \di W(t) + \gamma \, \di L_0^X(t),
\end{equation}
where $L_0^X(t) = \lim_{\ve \to 0+} \frac{1}{2\ve} \int_0^t \1_{\{|X(z)| < \ve\}} \, \di z$ is the symmetric local time of $X$ at the origin (see \cite{Harrison+Shepp}). If  $\gamma \notin [-1,1]$, then there is neither a strong nor a weak solution to \eqref{eq:skew}.   If $\gamma = 1$ (or $\gamma = -1$), then $X$ becomes a Brownian motion reflected upward (or downward, respectively) at $0$. If $|\gamma| < 1$, then $X$ is a regular diffusion on $I = \mbR$, and its canonical scale function and speed measure are given by \cite[\S 4.2]{Lejay:skew}:
\begin{equation}\label{eq:skew-scale-speed}
s(x) = \begin{cases} 
    p^{-1} x, & x \geq 0, \\
    q^{-1} x, & x < 0, 
\end{cases} \quad 
m(\di x) = \begin{cases}
    p \, \di x, & x \geq 0, \\
    q \, \di x, & x < 0,
\end{cases}
\end{equation}
where $p = \frac{1+\gamma}{2}$ and $q = \frac{1-\gamma}{2}$, i.e., $\gamma = p - q$ and $p = 1 - q \in (0,1)$. Sometimes, $p$ and $q$ are associated with `probabilities' that an excursion of the skew Brownian motion `selects' positive and negative direction upon leaving 0.

It follows from our preceding discussions that a continuous increasing function $\hat{s}$ and a measure $\hat{m}$ are the scale function and the speed measure of the skew Brownian motion if and only if they fulfill the following requirements:
\begin{itemize}
    \item $\hat{s}$ is affine on $(-\infty,0)$ and on $(0,\infty)$; 
    \item $q \hat{s}'(0-) = p \hat{s}'(0+)$;
    \item $\hat{m}(\di x) = (\hat{s}'(0-))^{-1} \1_{\{x < 0\}} \di x + (\hat{s}'(0+))^{-1} \1_{\{x \geq 0\}} \di x$.
\end{itemize} 

Let $X $ be constructed by formula \eqref{eq:reg_diff}, where $s$ and $m$ are given in \eqref{eq:skew-scale-speed}. 
Let us calculate the left, the right, and the symmetric local time of $X$ at 0 in terms of the local time of the Brownian motion from \eqref{eq:reg_diff}. Observe that  $s^{-1}(x)=\begin{cases} 
    p  x, & x \geq 0, \\
    q  x, & x < 0.
\end{cases}$
 
It follows from  \eqref{eq:int_loc_time_X} and the change of variables formula that  
\ba
L^{\text{{right}}}_{X_{skew}}(t,0)
&:= \lime \ve^{-1} \int_0^t \1_{ X_{skew}(z)\in[0,\ve] }\di z \\
&= \lime \ve^{-1}  \int_\mbR   \1_{ y\in[0,\ve]  } L_{X_{skew}}(t,y) m_{skew}(\di y) \\
&= \lime    \ve^{-1}\int_\mbR   \1_{ y\in[0,\ve] }  l_{B}(A^{-1}(t),s_{skew}(y)) m_{skew}(\di y) \\
&= \lime \ve^{-1}  \int_\mbR  \1_{ y\in[0,\ve] }  l_{B}(A^{-1}(t),y/p)  p \di y \\
&= p \, l_{B}(A^{-1}(t),0)
\ea
because $l_B$ is jointly continuous. Similarly
\[
\begin{aligned}
L^{\text{{left}}}_{X_{skew}}(t,0)
&:= \lime \ve^{-1} \int_0^t \1_{ X_{skew}(z)\in[-\ve,0] }\di z \\
&= q\, l_{B}(A^{-1}(t),0) ,
\end{aligned}
\]
and consequently the symmetric local time of $X_{skew}$ at 0 equals
\begin{equation}
    \label{eq:local time skewBM}
\begin{aligned}
L^{\text{{symmetric}}}_{X_{skew}}(t,0)
&=\lime (2\ve)^{-1} \int_0^t \1_{ |X_{skew}(z)|\leq \ve }\di z \\
&=\frac12 l_{B}(A^{-1}(t),0).
\end{aligned}
\end{equation}
\begin{remk}
    If $p=q=\frac{1}{2}$, then the skew Brownian motion is a  
    Brownian motion $W$. At first sight, the factor  $1/2$ in \eqref{eq:local time skewBM} may appear surprising.  If we use $s$ and $m$ from \eqref{eq:skew-scale-speed}, then $A(t)=\int_{\mbR} l_B(t,y) \,  \di y/4 =t/4, \quad A^{-1}(t)=4t, $  $s^{-1}(x)=x/2$. Consequently,  $W(t)=X_{\text{skew}}(t) = \frac{1}{2}B(4t)$. Here, if $X_{\mathrm{skew}}(0)=x$, then $B(0)=2x$, so that $\frac12 B(4t), t\geq 0,$ is a Brownian motion started from $x$. In particular, when  $x=0,$ it has the same law as $B$
\end{remk}

The domain of the generator $\mathrm{Dom}(\mathcal{A})$ consists of functions $f \in \bar{C}(\mbR)$ such that the second derivative $f''(x)$ exists for all $x \neq 0$ and can be extended by continuity to a function belonging to $\bar{C}(\mbR)$. Moreover, since $m(\{0\}) = 0$, the continuity condition for the scale derivative $D_s^+ f(0) = D_s^- f(0)$ yields the transmission (boundary) condition at the origin: $p f'(0+) = q f'(0-)$. 

Thus, the domain $\mathrm{Dom}(\mathcal{A})$ consists of $f \in \bar{C}(\mbR)$ such that 
\begin{enumerate}
    \item $f''(x), x\neq 0$ exists, $f''(0+)=f''(0-)$ and if we define $f''(0):=\lim_{x\to0}f''(x),$ then $f''\in \bar{C}(\mbR)$;
    \item $p f'(0+) = q f'(0-)$.
\end{enumerate}
 For $f\in \mathrm{Dom}(\mathcal{A})$ we have $\mathcal{A}f = \frac{1}{2} f''$.

The backward Kolmogorov equation for $u(t,x) = \E_x f(X(t))$ is given by
\[
u_t(t,x) = \frac{1}{2} u_{xx}(t,x), \quad x \neq 0, \ t > 0,
\]
where $f\in \mathrm{Dom}(\mathcal{A})$ and the interface condition holds at $x=0$:
\begin{equation}
    \begin{aligned}
\label{eq:skewBM-interface conditions}
u(t,0) &=u(t,0+)  = u (t,0-), \quad t > 0,\\
    p u_x(t,0+) &= q u_x(t,0-), \quad t > 0, \\
    u_{xx}(t,0+) &= u_{xx}(t,0-), \quad t > 0.
\end{aligned}
\end{equation}
Here, the first and third relations express continuity across the interface, while the second one is the nontrivial transmission condition.

\subsubsection{ Brownian motion with many semipermeable membranes}\label{expl:many membranes}
Consider the SDE with local times
\begin{equation}\label{eq:skew-2}
\di X(t) = \di W(t) + \sum_{i=1}^n \gamma_i \, \di L_{a_i}^X(t),
\end{equation}
where $a_1 < \dots < a_n$, $\gamma_i \in (-1,1)$, and $L_{a_i}^X(t) = \lim_{\ve \to 0+} \frac{1}{2\ve} \int_0^t \1_{\{|X(z)-a_i| < \ve\}} \, \di z$ is the symmetric local time of $X$ at $a_i$ for $1 \leq i \leq n$. 

The properties of this process on each interval containing only one of the points $a_i$ are the same as for a skew Brownian motion. Hence, any continuous strictly increasing function $s$ and a measure $m$ that satisfy the following properties form the  scale function and the speed measure pair for $X$:
\begin{itemize}
    \item $s$ is affine on each of the intervals $(a_i, a_{i+1})$ for $0 \leq i \leq n$, where $a_0 := -\infty$ and $a_{n+1} := +\infty$;
    \item $q_i {s}'(a_i-) = p_i {s}'(a_i+)$, where $p_i = 1 - q_i = \frac{1+\gamma_i}{2}$ for $1 \leq i \leq n$;
    \item ${m}(\di x) = \Big( \sum_{i=0}^{n} {s}'(a_i+)^{-1} \1_{\{x \in (a_i, a_{i+1})\}} \Big) \, \di x$.
\end{itemize} 
Note that ${s}'(a_{i+1}-) = {s}'(a_i+)$ for $1 \leq i \leq n-1$ because $s$ is linear on the intermediate intervals.

The backward Kolmogorov equation takes the form  
\[
u_t(t,x) = \frac{1}{2} u_{xx}(t,x), \quad x \notin \{a_1,\dots, a_n\}, \ t > 0,
\]
with the following interface conditions for  $1 \leq i \leq n$:
\begin{equation}
\label{eq:multi-skewBM-interface conditions}
    \begin{aligned}
    u(t, a_i) &= u(t, a_i+)=u(t,a_i-), \\
    p_i u_x(t, a_i+) &= q_i u_{x}(t,a_i-), \\
    u_{xx}(t, a_i+) &= u_{xx}(t,a_i-).
\end{aligned}
\end{equation}
 As above, only the second relation is a transmission condition; the other two express continuity across the interface.

\subsubsection{Skew-sticky Brownian motion}\label{expl:skew-sticky BM}
    Let $X_{\text{skew}} = X_{\text{skew}(\gamma)}$ be a  skew Brownian motion with parameter $\gamma = p - q \in (-1,1)$, whose scale function and speed measure are given by 
    \[
    \begin{aligned}
    s_{\text{skew}}(x)
    &= \frac{x}{q}\mathbb{I}_{\{x < 0\}} + \frac{x}{p}\mathbb{I}_{\{x \geq 0\}}, \\
    \mathrm{d} m_{\text{skew}}(x)
    &= (q\mathbb{I}_{\{x < 0\}} + p\mathbb{I}_{\{x \geq 0\}})\mathrm{d} x,
    \end{aligned}
    \]
    where $p = \frac{1+\gamma}{2}$ and $q = \frac{1-\gamma}{2}$. 
    
    For any $\lambda > 0$, we define a continuous strictly increasing additive functional $A_{\lambda}(t) := t + \lambda L_{X_{\text{skew}}}(t,0)$, where $L_{X_{\text{skew}}}(t,0)= L^{\text{{symmetric}}}_{X_{\text{skew}}}(t,0)$ is the symmetric local time of $X_{\text{skew}} = X_{\text{skew}(\gamma)}$ at the origin. The process 
    \[
    X^\lambda(t) := X^{\text{sticky}(\lambda)}_{\text{skew}(\gamma)}(t) := X_{\text{skew}}(A_{\lambda}^{-1}(t)), \quad t \geq 0,
    \]
    is a regular diffusion on $\mathbb{R}$ called a \emph{skew-sticky Brownian motion}. This time-change makes the point $0$ ``sticky'', allowing the process to spend a positive Lebesgue measure of time there.
For SDE characterizations of sticky Brownian motion and sticky reflected Brownian motion, see \cite{engelbertPeskir2014}.
 
    Let us determine the scale function and speed measure of $X^\lambda$.
    
    We can write $X^\lambda(t) = s_{\text{skew}}^{-1}\big(B \circ A^{-1} \circ A_{\lambda}^{-1}(t)\big) = s_{\text{skew}}^{-1}\big(B \circ (A_{\lambda} \circ A)^{-1} (t)\big)$, where $A(t) = \int_\mathbb{R} l_B(t, s_{\text{skew}}(y)) \, \mathrm{d} m_{\text{skew}}(y)$. 
  Apply \eqref{eq:local time skewBM}:
    \[
    \begin{split}
        (A_{\lambda} \circ A)(t) &= A(t) + \lambda L_{X_{\text{skew}}}(A(t),0) = A(t) + \frac{\lambda}{2} l_{B}\big(A^{-1}(A(t)), 0\big) \\
        &= \int_\mathbb{R} l_B(t, s_{\text{skew}}(y)) \left( \mathrm{d} m_{\text{skew}}(y) + \frac{\lambda}{2} \delta_0(\mathrm{d} y) \right).
    \end{split}
    \]
    The scale function $s$  of the skew-sticky Brownian motion matches $s_{\text{skew}}$, while  its speed measure $\mathrm{d}m$ acquires an atom at the origin:
    \begin{equation}
    \label{eq:skewstickyscale}
    \begin{aligned}
    s(x)
    &= \frac{x}{q} \mathbb{I}_{\{x < 0\}} + \frac{x}{p} \mathbb{I}_{\{x \geq 0\}}, \\
    \mathrm{d} m(x)
    &= \left( q \mathbb{I}_{\{x < 0\}} + p \mathbb{I}_{\{x \geq 0\}} \right) \mathrm{d} x
       + \frac{\lambda}{2} \delta_0(\mathrm{d} x).
    \end{aligned}
    \end{equation}

    Transmission condition \eqref{eq:boundary condition with atom two-sided} at $0$ takes the form:
    \begin{equation}
    \label{eq:sticky_boundary_cond}
    p u'(0+) - q u'(0-) = 2\frac{\lambda}{2} \frac{u''(0)}{2}= \frac{\lambda}{2}  u''(0).
 \end{equation}
    In particular, the transmission condition for the sticky Brownian motion (i.e., the case $p=q=1/2$) is $u'(0+) -   u'(0-) =    \lambda   u''(0).$ Note also that $\cA u(x)=\frac12 u'' (x), x\neq 0$ for all $u\in \text{Dom}(\cA).$

\section{Limit theorems} \label{sec:Main results}
Let $X_n$ be a regular diffusion on $I_n$ with scale function $s_n$ and speed measure $m_n$, $n \geq 0$. We assume that, in a sense to be specified below, we have convergence $\lim_{n \to \infty} s_n = s_0$, $\lim_{n \to \infty} m_n = m_0$, $\lim_{n \to \infty} I_n = I_0$, and $\lim_{n \to \infty} X_n(0) = X_0(0)$. Under suitable conditions, this yields convergence in distribution of the corresponding diffusions.  To prove these results, we construct copies of $X_n$, $n \geq 0$, defined on a common probability space with the same Brownian motion via representations similar to \eqref{eq:reg_diff}, \eqref{eq:reg_diff_CAF};  see  \eqref{eq:reg_diff_general}, \eqref{eq:reg_diff_CAF_general}  below. The desired convergence in distribution will then follow from almost sure convergence of these coupled copies, which can be established by deterministic real-analysis arguments. Throughout the paper, we consider convergence of processes in the space   $C([0,\infty)),$ equipped with topology of locally uniform convergence; see, for example, \cite{Ethier+Kurtz:1986}.

 \subsection{Open intervals}\label{sec:open}

In this section, we first  consider the case where $I_0 = (\alpha_0, \beta_0)$, so that both  endpoints are inaccessible. We also study  convergence of the processes $X_n$ stopped upon exiting  a fixed compact interval $[a,b] \subset (\alpha_0, \beta_0)$. We will not impose any specific conditions on the inclusion or exclusion of endpoints in $I_n$ for $n \geq 1$.

Let $B_0$ be a  Brownian motion started from $0$ with local time $l_{B_0}(t,y)$, $t \geq 0$, $y \in \mbR$, and let $\xi$ be a random variable independent of $B_0$. Then $B(t) := s(\xi) + B_0(t)$ is a Brownian motion with initial value  $s(\xi)$  and the process $X$  defined below  is a regular diffusion with initial value $\xi$:
\begin{equation}\label{eq:reg_diff_general}
X(t) = s^{-1}\big(s(\xi) + B_0(A^{-1}(t))\big), \quad t \geq 0, 
\end{equation}
where 
\begin{equation*}
A(t) := \int_{I} l_B(t, s(y)) \, m(\di y).
\end{equation*}
Observe that $l_B(t, \tilde{y}) = l_{B_0}(t, \tilde{y} - s(\xi))$ for $t \geq 0$, $\tilde{y} \in \mbR$ a.s. Thus, $A(t)$ can be rewritten as
\begin{equation}\label{eq:reg_diff_CAF_general}
A(t) := \int_{I} l_{B_0}(t, s(y) - s(\xi)) \, m(\di y).
\end{equation}

Let $X_n$ be a regular diffusion on $I_n$ with scale function $s_n$ and speed measure $m_n$, $n \geq 0$. We denote the endpoints of $I_n$ by $-\infty \leq \alpha_n < \beta_n \leq \infty$. Assume that:
\begin{itemize}
    \item[\textbf{A}$_1$.] $\limsup_{n \to \infty} \alpha_n \leq \alpha_0$ and $\liminf_{n \to \infty} \beta_n \geq \beta_0$;
    \item[\textbf{A}$_2$.] for any $x \in (\alpha_0, \beta_0)$, we have the pointwise convergence $\lim_{n \to \infty} s_n(x) = s_0(x)$;
    \item[\textbf{A}$_3$.] $\lim_{n \to \infty} \int_{\alpha_0}^{\beta_0} f \, \di m_n = \int_{\alpha_0}^{\beta_0} f \, \di m_0$ for any continuous function $f$ compactly supported in $(\alpha_0, \beta_0)$. That is, $m_n$ converges vaguely to $m_0$ as $n\to\infty$ on $I_0$; see Appendix for definition and properties of   vague convergence.
\end{itemize}

To illustrate the underlying idea,  first we  present a simple result where $I_n = \mbR$ for all $n \geq 0$. Generally, technical difficulties arise when the processes can hit the boundaries of $I_n$. This case  will be analyzed carefully later in this and subsequent sections.

\begin{thm} \label{thm:open-infinite-intervals}
Assume that $I_n = \mbR$ for all $n \geq 0$, conditions {\textbf{A}$_2$--\textbf{A}$_3$} are satisfied, and $s_n(\pm\infty) = \pm \infty$, $n \geq 0$. Let $B_0$ be a standard Brownian motion and let $\{\xi_n, n \geq 0\}$ be random variables independent of $B_0$ such that $\lim_{n \to \infty} \xi_n = \xi_0$ a.s. Suppose that $\{X_n, n \geq 0\}$ have the representations
\begin{equation}\label{eq:representation X_n}
X_n(t) = s_n^{-1}\big(s_n(\xi_n) + B_0(A_n^{-1}(t))\big), \quad t \geq 0, 
\end{equation}
where 
\begin{equation}\label{eq:representation A_n}
A_n(t) := \int_{I_n} l_{B_0}(t, s_n(y) - s_n(\xi_n)) \, m_n(\di y).
\end{equation}
Then $X_n\to X_0$  almost sure, locally uniformly in time; that is,
\[
\forall T > 0, \quad \lim_{n \to \infty} \max_{z \in [0,T]} |X_n(z) - X_0(z)| = 0.
\]
\end{thm}   

\begin{proof}
We work on an event $\Omega_0$ of probability $1$ on which all considered processes and random variables have standard properties; for example, $\lim_{n \to \infty} \xi_n = \xi_0$, $l_{B_0}(t, y)$ is jointly continuous in $(t, y)$, etc. For the remainder of the proof, we fix $\omega\in \Omega_0$.

Recall that if a Brownian motion has not visited a given point, its local time at that point is zero. Since $s_0(\pm\infty) = \pm\infty$, for every $T>0$ the supports of functions  $y \mapsto l_{B_0}(t, s_0(y) - s_0(\xi_0)), t\in[0,T],$ are contained in a common  compact set. Lemmas \ref{lem: conv monot cont} and \ref{lem:conv compositions} yield that $s_n(\xi_n)  \to  s_0(\xi_0)$ as $n \to \infty$. It follows from assumption \textbf{A}$_2$ and monotonicity of $s_n$ that the supports of the functions $y \mapsto l_{B_0}(t, s_n(y) - s_n(\xi_n))$ for $t \in [0,T]$ are contained within a common compact set for all sufficiently large $n$. 
Thus $l_{B_0}(t, s_n(y) - s_n(\xi_n)) \to l_{B_0}(t, s_0(y) - s_0(\xi_0))$ as $n \to \infty$ uniformly in $(t,y) \in [0,T]\times\mbR$.    Lemmas \ref{lem: conv monot cont} and \ref{lem:conv of integrals} imply the locally uniform convergence $A_n \to A_0$ as $n \to \infty$. 

Since $A_0$ is strictly increasing, see Lemma \ref{lem:properties of A}, we obtain the locally uniform convergence of their inverses, $A_n^{-1} \to A_0^{-1}$ (see Corollary \ref{corl:5.1}).  Applying  Corollary \ref{corl:5.1} once again, we get the locally uniform convergence $s_n^{-1} \to s_0^{-1}$.    Lemma \ref{lem:conv compositions} completes the proof.
\end{proof}

Let $\sigma_{x}(X) := \inf\{t \geq 0 : X(t) = x\}$ denote the first hitting time of $x$ by $X$, with the convention that   $\sigma_{x}(X) := \infty$ if the corresponding set is empty. 

In the next theorem,  endpoints $\alpha_0, \beta_0$ may belong  to $I_0.$ The theorem shows that the convergence of scale functions and speed measures implies   convergence of the processes on every time interval lying strictly before the first hitting time of the endpoints of $I_0.$ 

\begin{thm}\label{thm:conv_up_exit}
Suppose that assumptions {\textbf{A}$_1$--\textbf{A}$_3$} are satisfied. Let $B_0$ be a standard Brownian motion, and let $\{\xi_n, n \geq 0\}$ be random variables independent of $B_0$ such that $\lim_{n \to \infty} \xi_n = \xi_0$ a.s. Assume that the processes $\{X_n, n \geq 0\}$ have the representations \eqref{eq:representation X_n}, \eqref{eq:representation A_n}. Then for any $T > 0$, we have the almost sure convergence:
\[
\lim_{n \to \infty} \1_{\{T < \sigma_{\alpha_0}(X_0) \wedge \sigma_{\beta_0}(X_0)\}} \max_{z \in [0,T]} |X_n(z) - X_0(z)| = 0.
\]
\end{thm}

\begin{proof}
As in the proof of the previous theorem, we confine our discussion to   $\omega\in \Omega_0$ of probability $1$ where all necessary standard sample path properties hold. If $T \geq  \sigma_{\alpha_0}(X_0) \wedge \sigma_{\beta_0}(X_0)$, there is nothing to prove. Hence, assume that $T < \sigma_{\alpha_0}(X_0) \wedge \sigma_{\beta_0}(X_0).$ For notational brevity, without loss of generality we will assume that $s_0(\xi_0)=0.$ Then 
$A_0(t) = \int_{I_0} l_{B_0}(t, s_0(y) ) \, m_0(\di y) $ and $\limn s_n(\xi_n)=0.$

 It follows from the assumption of the theorem that there are $a , b $,  $ s_0(\alpha_0)<a < b <s_0(\beta_0) $ such that   $T < \sigma_{a }\big( B_0(A_0^{-1})\big) \wedge \sigma_{b }\big( B_0(A_0^{-1})\big)$. So $B_0(z)\in (a,b)$ for any $ z \in[0,A_0^{-1}(T )] $, i.e., for any $z$ such that  $A_0(z)\in[0,T ].$ 
 Corollary \ref{corl:integrals} implies that
\[
\lim_{n \to \infty} \sup_{z \in [0,z_0]} |A_n(z) - A_0(z)| = 0,
\]
for any $z_0$ such that $z_0 < \sigma_{a}\big( B_0 \big) \wedge \sigma_{b}\big( B_0 \big).$

Note that $A_0$ is strictly increasing on $[0,z_0]$ for any
$z_0 < \sigma_a(B_0)\wedge\sigma_b(B_0)$, see Lemma
\ref{lem:properties of A}. For such $z_0$, choose $z_1$ such that
\[
z_0<z_1<\sigma_a(B_0)\wedge\sigma_b(B_0).
\]
Since $A_0$ is strictly increasing on $[0,z_1]$, Corollary
\ref{corl:5.1} yields
\[
A_n^{-1}\to A_0^{-1}
\]
uniformly on $[0,A_0(z_0)]$.

Choose $z_0$ such that $A_0^{-1}(T)<z_0< \sigma_{a }\big( B_0 \big) \wedge \sigma_{b }\big( B_0 \big).$ Then $T < A_0(z_0)$. So  $A_n^{-1}\to A_0^{-1}$ uniformly on $[0,T]$.  

Recall that $B_0(A_n^{-1}(t))\in (a,b)$ for $ t \in[0, T] $. Therefore 
$s_n(\xi_n)+B_0(A_n^{-1}(t ))\in (a,b), t\in [0,T]$ if $n$ is sufficiently large (here we used agreement that $\limn s_n(\xi_n)=0.$).

It follows from Corollary \ref{corl:5.1} that $\limn s_n^{-1}=s_0^{-1}$ uniformly on $[a,b].$
 Lemma \ref{lem:conv compositions}  yields the convergence of the compositions:
\begin{equation}\label{eq:conv_proc_uni}
\lim_{n \to \infty} \max_{t \in [0,T]} \big| s_n^{-1}\big(s_n(\xi_n) + B_0(A_n^{-1}(t))\big) - s_0^{-1}\big( B_0(A_0^{-1}(t))\big) \big| = 0.
\end{equation} 
\end{proof}
\begin{corl}
Assume that $I_0$ is an open interval. Then for any $T > 0$, we have the almost sure convergence:
\[
\lim_{n \to \infty} \sup_{t \in [0,T]} |X_n(t) - X_0(t)| = 0.
\]
\end{corl}

\begin{proof}
 Since $I_0$ is open, the endpoints are inaccessible, which implies $\sigma_{\alpha_0}(X_0) = \sigma_{\beta_0}(X_0) = +\infty$ a.s. 
The assertion therefore follows directly from Theorem  \ref{thm:conv_up_exit}.
\end{proof}
Let $\tau_{(a,b)}(X) := \inf\{t \geq 0 : X(t) \notin (a,b)\}$, with the usual convention that   $\inf\emptyset:= \infty$. 
\begin{corl}\label{corl:conv}
Let $X_n$ be as in Theorem \ref{thm:conv_up_exit}. Assume that the sequences of real numbers $\{a_n, n \geq 0\}$ and $\{b_n, n \geq 0\}$ satisfy    $\lim_{n \to \infty} a_n = a_0$, and $\lim_{n \to \infty} b_n = b_0$, where $\alpha_0 < a_0 < b_0 < \beta_0$. Then,  almost surely,
\begin{equation}\label{eq:conv_stopping_times}
\lim_{n \to \infty}   \tau_{(a_n,b_n)} (X_n)= \tau_{(a_0,b_0)} (X_0),
\end{equation} 
and the    stopped processes converge uniformly:
\[
\lim_{n \to \infty} \sup_{t \geq 0} |X_n(t \wedge\tau_{(a_n,b_n)} (X_n)) - X_0(t \wedge \tau_{(a_0,b_0)} (X_0))| = 0.
\]
\end{corl}

\begin{proof}
Property \eqref{eq:conv_stopping_times} follows directly from Lemma \ref{lem:hitting_limit} due to the sample path properties of Brownian motion, the strict increase of $s_0$ on $\mathrm{int}(I_0)$, and the strict increase of $A_0^{-1}$ before the process $s_0(\xi_0) + B_0(A_0^{-1}(\cdot))$ exits the interval $(s_0(\alpha_0), s_0(\beta_0))$, see Lemma \ref{lem:properties of A}. The second assertion of the corollary follows immediately from the first one and Theorem \ref{thm:conv_up_exit}.
\end{proof}

\begin{corl}\label{corl:conv_weak}
Let the scale functions and speed measures of the regular diffusions $(X_n)$ satisfy conditions {\textbf{A}$_1$--\textbf{A}$_3$}, and let $X_n(0) \Rightarrow X_0(0)$ as $n \to \infty$. Assume that the sequences of numbers $\{a_n, n \geq 0\}$ and $\{b_n, n \geq 0\}$ satisfy the assumptions of Corollary \ref{corl:conv}. Then we have the convergence in distribution:
\[
\left(\tau_{(a_n,b_n)} (X_n), \, X_n(\cdot \wedge \tau_{(a_n,b_n)} (X_n))\right)
\Rightarrow 
\left(\tau_{(a_0,b_0)} (X_0), \, X_0(\cdot \wedge \tau_{(a_0,b_0)} (X_0))\right)
\]
as $n \to \infty$ in $\mbR \times C([0,\infty))$. In particular, if $I_0$ is an open interval, then $X_n \Rightarrow X_0$ as $n \to \infty$ in $C([0,\infty))$.
\end{corl}
\begin{proof}
    For the proof, we use Skorokhod's representation theorem and construct copies  $\xi_n$ of $X_n(0), n\geq 0,$ defined on a common probability space such that $\xi_n\to\xi_0,\quad n\to\infty$ almost surely. Then we take a Brownian motion $B_0$ independent of $(\xi_n)$ and  defined on an extension of the new probability space and construct copies of the diffusions $(X_n)$ as in Theorem \ref{thm:conv_up_exit}. For these copies we have a.s. convergence, which implies convergence in distributions of the original processes.
\end{proof}

\subsection{Absorbing boundary}\label{sec:absorbing}

We know from the results of Section \ref{sec:open} that if the scale functions $(s_n)$ and speed measures $(m_n)$ of regular diffusions converge in some sense to the scale function $s_0$ and speed measure $m_0$ of a regular diffusion $X_0$ on $I_0$, then we have  convergence of the diffusions up to the   exit time  from $\mathrm{int}(I_0)$. In this section, we study the behavior of the processes after the moment of hitting the boundary of $I_0$. Specifically, we consider the case where $I_0 = [0,\infty)$ and $X_0$ is stopped upon hitting $0$. According to the classification of boundaries detailed in Remark \ref{remk:classification exit}, this phenomenon of an absorbing boundary can occur only in the following two cases: 

\begin{itemize}
    \item[\textbf{Case 1}] ($0$ is an absorbing point): $m_0((0,1]) < \infty$ and $m_0(\{0\}) = \infty$;
    \item[\textbf{Case 2}] ($0$ is an exit boundary): $m_0((0,1]) = \infty$ and $\int_{(0, 1]} (s_0(x)-s_0(0)) \, m_0(\di x) < \infty$.
\end{itemize}

Assume that $\omega$ is such that $\sigma_0(X_0) := \inf\{t \geq 0 : X_0(t) = 0\}<\infty$. Note that  we must have $A_0(\sigma_{-s_0(\xi_0)}(B_0)-) < \infty$ and $A_0(t) = +\infty$ for $t > \sigma_{-s_0(\xi_0)}(B_0)$ for a.a. of such $\omega$. In Case 1, this property is immediate from the definition, whereas in Case 2, it follows from Remark \ref{remk:classification exit}.

Let $X_n$ be a regular diffusion on $I_n$ with scale function $s_n$ and speed measure $m_n$. We denote the left endpoint of $I_n$ by $\alpha_n$ and the right endpoint by $\beta_n$. These endpoints may either be included in $I_n$ or excluded from it. 

 We  adopt the convention that $s_n(x) := -\infty$ if $x < \alpha_n$, and $s_n^{-1}(y) = \alpha_n$ if $y \leq \inf_{x \in I_n} s_n(x)$. Furthermore, we extend the measures  $m_n$ and their push-forwards $m_n \circ s_n^{-1}$ to the entire real line  by setting $m_n(\mbR \setminus I_n):=0,\quad (m_n \circ s_n^{-1})(\mbR \setminus s_n(I_n)):= 0$.  Since scale functions are defined up to an additive constant, we will assume that  $s_0(0)=0$.

\begin{thm}\label{thm:conv_absorbing}
Assume that the following conditions hold:  
\begin{enumerate}
    \item[\textbf{A}$_1'$.] $\lim_{n \to \infty} \beta_n = +\infty$ and $\lim_{n \to \infty} \alpha_n =: a \in[-\infty, 0]$.
    \item[\textbf{A}$_2'$.] For every $x > a$, the limit $\lim_{n \to \infty} s_n(x) =: s(x) \in [-\infty, \infty]$ exists. Moreover,  if $a < 0$, then $s(x) \in [-\infty, 0)$ for all $x \in (a, 0)$.
    \item[\textbf{A}$_3'$.] The measures $m_n$ converge vaguely to $m_0$ as $n\to\infty$ on $(0,\infty)$.
    \item[\textbf{B}$_1$.] $s_n(X_n(0)) \Rightarrow s_0(X_0(0))$ as $n \to \infty$.
    \item[\textbf{B}$_2$.] $s(x) = s_0(x)$ for all $x > 0$.
    \item[\textbf{B}$_3$.] For every $\delta > 0$, $\lim_{n \to \infty} (m_n \circ s_n^{-1})([-\delta, \delta]) = \infty$.
    
\end{enumerate}

Then the stochastic processes converge in distribution: $X_n \Rightarrow X_0$ as $n \to \infty$ in $C([0, \infty))$.
\end{thm}

Before proving the theorem, let us make a few remarks concerning alternative formulations of its assumptions and the necessity of some of them.

\begin{remk}
Assumption \textbf{B}$_3$ is equivalent to $\lim_{n \to \infty} m_n ([s_n^{-1}(-\delta), s_n^{-1}(\delta)]) = +\infty$ for any $\delta > 0$. Under assumptions $\mathbf A_3'$ and  $\mathbf B_2$, condition  $\mathbf B_3$  is also equivalent to $\lim_{n \to \infty} m_n ([s_n^{-1}(-\delta), \delta]) = +\infty$ for every $\delta > 0$. 
\end{remk}
\begin{remk}
    Fix  $x_0 > 0$ such that $m_0(\{x_0\})=0$, and introduce the distribution-like functions:
\begin{gather}\label{eq:distr_functions}
F_n(x) := \begin{cases}
    m_n((x_0, x]), & x \geq x_0, \\
    -m_n((x, x_0]), & x < x_0,
\end{cases} \notag \\
\text{and} \quad \wt{F}_n(x) := \begin{cases}
    (m_n \circ s_n^{-1})((s_n(x_0), x]) = m_n((x_0, s_n^{-1}(x)]), & x \geq s_n(x_0), \\
    -(m_n \circ s_n^{-1})((x, s_n(x_0)]) = -m_n((s_n^{-1}(x), x_0]), & x < s_n(x_0).
\end{cases} 
\end{gather}
Then condition \textbf{A}$_3'$ means that $\lim_{n \to \infty} F_n(x) = F_0(x)$ for any $x > 0$ that is a point of continuity of $F_0$. Under assumption \textbf{A}$_3'$, condition \textbf{B}$_3$ can be written as 
\[
\forall \delta > 0, \quad \lim_{n \to \infty} \wt{F}_n(-\delta) = \lim_{n \to \infty} F_n(s_n^{-1}(-\delta)) = -\infty.
\]
\end{remk}

\begin{remk}
Condition \textbf{B}$_3$ cannot be replaced by 
\[
\forall \delta > 0, \quad \lim_{n \to \infty} m_n([-\delta, \delta]) = \infty
\]
or
\[
\forall \delta > 0, \quad \lim_{n \to \infty} F_n(-\delta) = -\infty.
\]
Indeed, let $I_n = [-\frac{1}{n}, \infty)$, $m_n = \infty \cdot \delta_{-1/n} + \lambda|_{I_n}$, and $s_n(x) = x \1_{\{x \geq 0\}} + n x \1_{\{x \in [-\frac{1}{n}, 0)\}}$ for $n \in \mbN$. Let $I_0 = [0, \infty)$, $m_0 = \infty \cdot \delta_{0} + \lambda|_{I_0}$, and $s_0(x) = x \1_{\{x \geq 0\}}$, where $\lambda$ denotes the Lebesgue measure.
Then $s_n^{-1}(x)= x\1_{x\geq 0}+ \frac{x}{n}\1_{x\in  [-1,0)}, \quad x\geq -1; 
\quad \wt m_n (\di x):= m_n\circ s_n^{-1} (\di x) = 
\infty \cdot \delta_{- 1} +\lambda|_{[0,\infty)}+ n^{-1} \lambda|_{[-1,0)}.$

Assume that $X_n(0) = 1$ for all $n \geq 0$. Then $X_0$ is a  Brownian motion started from $1$ and absorbed at $0$. In particular,
\[
\Pb\big(\exists t_1, t_2 \in [0,1], \ t_1 < t_2 : \ 0 = X_0(t_1) < X_0(t_2) = 1\big)=0.
\]
On the other hand, by the occupation time formula we have  for all $n \geq 1$:
\[
A_n(t)=\int_{s_n(I_n)} l_B(t,y) \, \wt m_n (\di y) = \int_0^t (\1_{B(z)\geq 0}+n^{-1}\1_{B(z)\in (-1,0)}) \di z \leq t,\quad t< \sigma_{-1}(B).
\]
So, $A_n^{-1}(t)\geq t$ and  the processes $X_n$ satisfy the property:
\begin{align*}
\inf_{n \in \mbN} \Pb\big(\exists t_1, t_2 \in [0,1], \ t_1 < t_2 : \ 0 = X_n(t_1) < X_n(t_2) = 1\big) &\geq \\
\Pb_1\big(\exists t_1, t_2 \in [0,1], \ t_1 < t_2 : \ 0=B(t_1)  < B(t_2) = 1, \ B(t) &> -1, \ t \in [0,1]\big) > 0.
\end{align*}
Hence, $X_n$ does not converge to $X_0$ in distribution. 

The counterexample can be modified so that $I_n = [0, \infty)$. Indeed, the shifted diffusions $\frac{1}{n} + X_n(t)$, $t \geq 0$, have the state space $[0, \infty)$ but still fail to converge to $X_0$.
\end{remk}

\begin{remk}
\label{remk:convX and s(X)}If $X_0(0) > 0$ a.s., then assumption \textbf{B}$_1$ is equivalent to $X_n(0) \Rightarrow X_0(0)$ as $n \to \infty$. Indeed, by Skorokhod's representation theorem, without loss of generality we may assume that initial conditions converge almost surely. Since $s_n(x) \to s_0(x)$ locally uniformly for $x > 0$, the convergence $\lim_{n \to \infty} X_n(0) = X_0(0)$ a.s. directly implies  $\lim_{n \to \infty} s_n(X_n(0)) = s_0(X_0(0))$ a.s. On the other hand, Corollary \ref{corl:5.1} implies the locally uniform convergence of the inverse functions $s_n^{-1}(y) \to s_0^{-1}(y)$ for $y > 0$. Thus, assumption \textbf{B}$_1$ yields $\lim_{n \to \infty} X_n(0) = \lim_{n \to \infty} s_n^{-1}\big(s_n(X_n(0))\big) = s_0^{-1}\big(s_0(X_0(0))\big) = X_0(0)$ a.s.

However, if $X_0(0) = 0$ a.s., the equivalence no longer holds, and the statement of the theorem can fail if we merely assume that $X_n(0) \Rightarrow X_0(0)$.  

Indeed, let $I_0 = [0, \infty)$, $m_0 = \infty \cdot \delta_{0} + \lambda|_{I_0}$, and $s_0(x) = x \1_{\{x \geq 0\}}$. Then $X_0(t) = 0$ for all $t \geq 0$. Set $I_n = \mbR$, $m_n = n \cdot \delta_{0} + \lambda$,  
\[
s_n(x) = x \1_{\{x \geq 0\}} + n x \1_{\{x \in [-\frac{1}{n}, 0)\}} + \left(-1 + \frac{x + \frac{1}{n}}{1 - \frac{1}{n}}\right) \1_{\{x < -\frac{1}{n}\}},
\]
and   choose the initial states to be $X_n(0) = -\frac{1}{n}$ for $n \geq 2$. Then $X_n(0) \to X_0(0) = 0$ deterministically, but $\limn s_n(X_n(0))=-1\neq 0=s_0(X_0(0))$.
Observe that $A_n(t)\leq t$ whenever $B(s)\in (-2,0), s\in [0,t].$  Then   
\[
\Pb(\min_{0\leq t\leq 1} X_n(t)\leq -1) =
\Pb_{-1}(\min_{0\leq t\leq 1} s_n^{-1}(B(A_n^{-1}(t)))\leq -1)=\]
\[
\Pb_{-1}(\min_{0\leq t\leq 1}   B(A_n^{-1}(t)) \leq -2) 
\geq \Pb_{-1}\big(\sigma_{-2}(B) < \sigma_0(B),\quad \sigma_{-2}(B)\leq 1\big) >0,
\]
which implies that $X_n$ cannot converge to $X_0 \equiv 0$ in distribution.
\end{remk}

\begin{proof}[Proof of Theorem \ref{thm:conv_absorbing}]
By Skorokhod's representation theorem, there exist random variables $\xi_n \overset{d}{=} X_n(0)$, $n \geq 0$, defined on a common probability space such that $\lim_{n \to \infty} s_n(\xi_n) = s_0(\xi_0)$ almost surely. We can construct copies of the processes $X_n$, $n \geq 0$ (denoted by the same symbols), with initial conditions $\xi_n$ using a representation analogous to that in Theorem \ref{thm:conv_up_exit} with the same Brownian motion $B_0$. Then, to establish convergence in distribution, it suffices to show that for any fixed $T > 0$:
\begin{equation}\label{eq: uniform convergence of processes}
\lim_{n \to \infty} \max_{t \in [0,T]} |X_n(t) - X_0(t)| = 0 \quad \text{a.s.}
\end{equation}

Since $(\xi_n)$ are independent of $B_0$,   we may condition on $(\xi_n)$ and fix a realization  $x_n := \xi_n$, $n \geq 0$ such that $s_n(x_n)\to s_0(x_0)$. So, without loss of generality we assume that $(x_n)$ are  non-random. Throughout the proof, we confine our analysis to sample points $\omega$ belonging to a fixed set of probability $1$ where the Brownian motion $B_0$ exhibits its standard sample path properties (e.g., the joint continuity of its local time). All subsequent convergences are considered for such $\omega$ only. Furthermore, we implicitly assume that all parameters such as $T$, $\ve$, and $\delta$ are restricted to rational numbers, which avoids any uncountability issues when intersecting sets of full measure.  

Denote $y_n := s_n(x_n) - s_0(x_0) \to 0$ as $n \to \infty$, and let $\wt{m}_n := m_n \circ s_n^{-1}$.

For notational brevity, it is convenient to assume that instead of  the representation \eqref{eq:reg_diff_general}, the processes $X_n$ are  constructed using the same Brownian motion $B_{s_0(x_0)}$ started from the point $s_0(x_0)$:
\begin{equation}\label{eq:repr20}
X_n(t) = s_n^{-1}\big(y_n + B_{s_0(x_0)}(A_n^{-1}(t))\big), \quad t \geq 0, 
\end{equation}
where 
\begin{equation}\label{eq:repr21}
A_n(t) = \int_{s_n(I_n)} l(t, y - y_n) \, \wt{m}_n(\di y),
\end{equation}
and $l(t, x)$ denotes the local time of $B_{s_0(x_0)}$ in $x$ up to  time $t$ .

It follows   from Theorem \ref{thm:conv_up_exit} that for any $T < \sigma_0(X_0)$, we have the uniform convergence:
\begin{equation}
    \label{eq:750}
    \lim_{n \to \infty} \max_{t \in [0, T]} |X_n(t) - X_0(t)| = 0 .  
\end{equation}
Hence, to complete the proof of the theorem, it remains to analyze the case where $\sigma_0(X_0) < T$.

Recall  that $s_0(0) = 0$.
\begin{lem}\label{lem:scale at 0}
For any $\ve > 0$, there exist $\delta > 0$ and $n_0 \in \mbN$ such that for all $n \geq n_0$:
\begin{equation}\label{eq:ineq_scale0}
-\ve \leq \inf_{x \in s_n(I_n) \cap [-\delta,\delta]} s_n^{-1}(x) \leq \sup_{x \in s_n(I_n) \cap [-\delta,\delta]} s_n^{-1}(x) \leq \ve.
\end{equation}
\end{lem}

\begin{proof}
Since $\lim_{n \to \infty} s_n(x) = s_0(x)$ for all $x > 0$ and $s_0(0) = 0$, there is $\delta>0$ such that  $s_n(\ve)>\delta$ for sufficiently large $n.$ Hence, $s_n^{-1}(x)\leq \ve $ for $x\leq \delta.$

The first inequality in \eqref{eq:ineq_scale0} is trivial if $a = 0$. If $a < 0$, then by the assumptions, we have $\lim_{n \to \infty} s_n(y) = s(y) < 0$ for any $y \in (a, 0)$.   Set  $y := \frac{a}{2}\vee \frac{-\ve}{2}.$ Decreasing  $\delta>0$  if necessary, we may assume that $s(y)<-\delta$. Then  $s_n(y) < -\delta$ for sufficiently large $n$ and, consequently,  $s_n^{-1}(-\delta) > y > -\ve$. The monotonicity of $s_n^{-1}$  yields the proof.
\end{proof}

For a given $\ve > 0$, let $\delta > 0$ be chosen according to Lemma \ref{lem:scale at 0}. We assume that $\omega$ exhibits typical sample path properties and that $\sigma_0(B_{s_0(x_0)}) < \infty$ for this $\omega$. If $x_0=0$, set $z_0:=0.$ Otherwise, if $x_0>0$, we choose a time $z_0 \in [0, \sigma_0(B_{s_0(x_0)}))$ such that 
\[
0 \leq \sup_{z \in [z_0, \sigma_0(B_{s_0(x_0)})]} B_{s_0(x_0)}(z) < \frac{\delta}{2}.
\]
In either case, we choose  $\delta_1 \in (0, \delta)$ such that, after the instant $z_0$, the Brownian motion $B_{s_0(x_0)}$ hits $-\delta_1$ before hitting $\delta/2$. This last requirement means that $\sigma_{-\delta_1}(B_{s_0(x_0)})$ occurs strictly before the first moment after $z_0$ when $B_{s_0(x_0)}$ hits $\delta/2$.

Set $t_n := A_n(z_0)$. It follows from our assumptions that $\lim_{n \to \infty} t_n = t_0 := A_0(z_0)$ and, for all sufficiently large $n$, we have the inequality $|y_n + B_{s_0(x_0)}(A_n^{-1}(t_n))| < \delta/2$. Furthermore, after the moment $t_n$, the process $y_n + B_{s_0(x_0)}(A_n^{-1}(t))$ either hits $-\delta_1/2$ before hitting  $\delta$, or   hits neither of these points.

Denote   $\tau_n:=\inf\{ t\geq t_n : X_n(t)\notin  ( {s_n^{-1}(-\delta_1/2)} , \ve)\}.$ In all cases, Lemma \ref{lem:scale at 0} implies that 
\[
\sup_{t \geq t_0} |X_0(t)| \leq \ve \quad \text{and} \quad \limsup_{n \to \infty} \sup_{t \in [t_n, \tau_n ]} |X_n(t)| \leq \ve.
\]
If $x_0>0$, combining this with the uniform convergence  for any $T < \sigma_0(X_0)$ from \eqref{eq:750}, we obtain
\[
\limsup_{n \to \infty} \sup_{t \in [0, \tau_n]} |X_n(t) - X_0(t)| \leq 2\ve.
\]
If $x_0=0,$ the same estimate follows directly from the preceding bound, since \(t_n=t_0=0\) and \(X_0\equiv0\).

Therefore, to complete the proof of the theorem, it remains to show that 
\[
\lim_{n \to \infty} \tau_n= 
+\infty \quad \text{a.s.}
\]
For this purpose, it is sufficient to check that $\lim_{n \to \infty} A_n(\sigma_{-\delta_1/4}(B_{s_0(x_0)})) = +\infty$, which is equivalent to
\[
\lim_{n \to \infty} \int_{s_n(I_n)} l\big(\sigma_{-\delta_1/4}(B_{s_0(x_0)}), y - y_n\big) \, \wt{m}_n(\di y) = +\infty.
\]
 
Since $s_0(x_0) \geq 0$, the local time $l\big(\sigma_{-\delta_1/4}(B_{s_0(x_0)}), 0\big) > 0$ is strictly positive for a typical $\omega$. It follows from the joint continuity of the local time that there exist a (random) threshold $\wt{\delta} \in (0, \delta_1/4)$ and a constant $c > 0$ such that $l\big(\sigma_{-\delta_1/4}(B_{s_0(x_0)}), y\big) > c$ whenever $|y| \leq \wt{\delta}$. Therefore, for all sufficiently large $n$, we have $l\big(\sigma_{-\delta_1/4}(B_{s_0(x_0)}), y - y_n\big) \geq c$ for $|y| \leq \wt{\delta}/2$, which by condition \textbf{B}$_3$ implies:
\[
\begin{aligned}
\lim_{n \to \infty}
\int_{s_n(I_n)} l\big(\sigma_{-\delta_1/4}(B_{s_0(x_0)}), y - y_n\big) \, \wt{m}_n(\di y)
&\geq \lim_{n \to \infty} c \, \wt{m}_n\big(s_n(I_n) \\
&\qquad \cap [-\wt{\delta}/2, \wt{\delta}/2]\big) = +\infty.
\end{aligned}
\]
The theorem is proved.

\end{proof}

\subsection{Reflecting endpoint}\label{sec:reflecting}

In this section, we consider the case where $X_0$ is a regular diffusion on $I_0 = [0, \infty)$ with scale function $s_0$ and speed measure $m_0$. We assume that $X_0$ can hit $0$ with positive probability and can subsequently enter $(0, \infty)$. According to the classification of boundaries, this behavior is possible if and only if $m_0([0, 1]) < \infty$ and $s_0(0) > -\infty$.

Let $X_n$ be a regular diffusion on $I_n$ with scale function $s_n$ and speed measure $m_n$. We denote the left endpoint of $I_n$ by $\alpha_n$ and the right endpoint by $\beta_n$. These endpoints may either be included in $I_n$ or excluded from it. Recall the notation $y_n := s_n(X_n(0)) - s_0(X_0(0))$ and $\wt{m}_n := m_n \circ s_n^{-1}$.

Similarly to the previous section, we extend the measures $\wt{m}_n$ to the entire real line by setting  $\wt{m}_n(\mbR \setminus s_n(I_n)) := 0$. Without loss of generality, we assume that $s_0(0) = 0$ and adopt the convention that $s_n(x) := -\infty$ for $x < \alpha_n$, $n \geq 0$, and in particular, $s_0(x)=-\infty$ for $ x<0$.

\begin{thm}\label{thm:regular point}   
Suppose that the following assumptions are satisfied: 
\begin{enumerate}
    \item[\textbf{A}$'_1$.] $\lim_{n \to \infty} \beta_n = +\infty$ and $\lim_{n \to \infty} \alpha_n =: a \in[-\infty, 0]$;
    \item[\textbf{A}$''_2$.] for any $x \neq 0$, $\lim_{n \to \infty} s_n(x) = s_0(x)$;
    \item[\textbf{A}$''_3$.] the measures $\wt{m}_n$ converge vaguely to $\wt{m}_0$ as $n \to \infty$ on $\mbR$;  
    \item[\textbf{B}$_1$.] $s_n(X_n(0)) \Rightarrow s_0(X_0(0))$ as $n \to \infty$.
\end{enumerate}
Then the stochastic processes converge in distribution: $X_n \Rightarrow X_0$ as $n \to \infty$ in $C([0, \infty))$.
\end{thm}

\begin{remk} Let $F_n, \wt F_n$ be from \eqref{eq:distr_functions}.
    Condition \textbf{A}$''_{3}$  is equivalent to the   $\limn F_n(x)=F_0(x)$ for any $x>0$ such that $x$ is a point of continuity of $F_0 $
and   $  \limn \wt F_n (y)=  \limn F_n\circ s_n^{-1}(y)=F_0(0-), y<0.$

   Note that  vague convergence $m_n\to m_0$  on $\mbR$ together with \textbf{A}$''_2$ would imply   \textbf{A}$''_{3}$   because $\limn s_n (x)=-\infty, x<0 $ and therefore $\limn s^{-1}_n (x)=0.$ However, the convergence $m_n\to m_0$  is not equivalent to \textbf{A}$''_{3}$.  Indeed, take $s_n(x)=x\1_{x>0}+n^2x\1_{x<0}$ and $m_n=\lambda + n\delta_{-\frac1n}.$ Then \textbf{A}$''_{3}$ is true but $  \limn  m_n=  
m_0 $ is not. The same example with $X_n(0)=-\frac1n $ also shows us that assumption \textbf{B}$_1$ can't be replaced with   $X_n(0)\Rightarrow X_0(0)$ if $X_0(0)=0 $ with positive probability. On the other hand, see Remark \ref{remk:convX and s(X)}, if $X_0(0)>0$ a.s., convergence $X_n(0)\Rightarrow X_0(0)$ is equivalent to \textbf{B}$_1$. 
\end{remk}
\begin{remk}
If the conditions of the theorem are satisfied and $\lim_{n \to \infty} \alpha_n = a < 0$, then the time-changed processes $B_{s_0(x_0)} \circ A_n^{-1}$ need not converge to $B_{s_0(x_0)} \circ A_0^{-1}$ in distribution on the space of continuous functions. 
Indeed, let $s_0(x)=x, x\geq 0,\quad m_0(\di x)=\1_{x\geq 0}\di x,\quad x_0=0;$ that is,  $X_0$ is a reflected Brownian motion. Assume that $I_n=\mbR, \quad s_n(x)=x\1_{x\geq 0}+ nx\1_{x<0}, \quad m_n(\di x)=  \di x, \quad x_n=0, \quad n\in \mbN.$ Then $\wt m_n(\di x)= (n^{-1}\1_{x<0}+\1_{x\geq 0})\di x$ and by the occupation time formula $A_n(t)=\int_{\mbR} l_{B_0}(t,y) \wt m_n(\di y) =  \int_0^t (\1_{B_0(s)\geq 0}+n^{-1} \1_{B_0(s)<0} ) \di s\leq t. $   Note that $A_n^{-1}:[0,\infty)\to[0,\infty)$ is a continuous bijection. Hence, the path of $B_0\circ A_n^{-1}$ attains the same values as the path of \(B_0\), although at different times. Moreover, since $A_n(t)\leq t$, we have $ A_n^{-1}(t)\geq t$, and therefore,  $\inf_{t\in [0,1]} B_0\circ A_n^{-1}(t)\leq \inf_{t\in [0,1]} B_0(t).$  Hence, 
\[
\inf_{n\in\mbN}\Pb(\exists t\in [0,1] :  B_{s_0(x_0)} \circ A_n^{-1}(t)=-1) \geq 
\Pb(\exists z\in [0,1] :  B_{0}(z) =-1 )>0.
\]
This proves that the statement $B_{s_0(x_0)} \circ A_n^{-1}\Rightarrow B_{s_0(x_0)} \circ A_0^{-1}$ in $C([0,\infty))$ is incorrect.

\end{remk}

\begin{proof}[Proof of Theorem \ref{thm:regular point}]
    We may assume that the processes $X_n$ admit the representation \eqref{eq:repr20}--\eqref{eq:repr21} and that $\lim_{n \to \infty} y_n = 0$ almost surely. Recall also that our analysis is confined to a fixed set of probability $1$ where all sample paths exhibit standard properties, so we shall omit   ``a.s."   throughout the proof. Without loss of generality, we will also assume that $s_0$ is unbounded from above. Indeed, if it is bounded, we can consider the modified scale functions $s_{R,n}(x) := s_n(x) \1_{\{x < R\}} + (x - R + s_n(R)) \1_{\{x \geq R\}}, x\in I_n,$ for all $n\geq 0$ and subsequently let $R \to +\infty$ and observe that the processes $X_n$ and $X_{R,n}$ coincide up to their first exit time from the set $I_n \cap \{x \leq R\}$. 
The assumption that $s_0$ is unbounded from above ensures that $s_0^{-1}(x)$ is finite for all $x \in \mbR$. This simplifies some technical details. Moreover, unboundedness of $s_0$ implies that $A_0$ is unbounded, and hence $A_0^{-1}(t)$ is finite for all $t \in [0, \infty)$.

The local time $l(z, y)$ is continuous with respect to $(z,y)$ and is equal to 0 if the Brownian motion has not visited $y$ before time $z$. Thus,  Corollary \ref{corl:integrals}  yields that we have the locally uniform convergence $A_n \to A_0$. Note, however, that the limiting function $A_0$ is not strictly increasing; it is constant during the excursions of the underlying Brownian motion into the negative half-line.  Consequently, its generalized inverse $A_0^{-1}$ is discontinuous, which makes the convergence of the processes $X_n$ nontrivial.  Nevertheless, Lemma  \ref{lem: conv monot cont} and  Lemma \ref{lem: inverses}(6) imply convergence  $\limn A_n^{-1}(t) = A_0^{-1}(t) $ for all $t$ from a   set  dense in $[0,\infty).$

Let $T > 0$ be fixed. Corollary \ref{corl:5.1} implies the locally uniform convergence of the inverse scale functions:
\begin{equation}\label{eq:s-1=uniform conv}
\lim_{n \to \infty} \max_{x \in [-R, R]} |s_n^{-1}(x) - s_0^{-1}(x)| = 0
\end{equation}
for any $R > 0$.

Therefore, 
$$\lim_{n \to \infty} X_n(t) = \limn s_n^{-1}\big(y_n + B_{s_0(x_0)}(A_n^{-1}(t))\big)= s_0^{-1}\big( B_{s_0(x_0)}(A_0^{-1}(t))\big)=X_0(t),$$ for all $t$ in  a dense subset of $[0, T]$. Thus, to complete the proof of the theorem, it suffices to show that the sequence of functions $(X_n)_{n \geq 1}$ is uniformly equicontinuous on $[0, T]$ for the fixed  $\omega$.

  Choose $z_0 > 0$ such that $A_0(z_0) > T$. Our assumptions imply that $\lim_{n \to \infty} A_n(z) = A_0(z)$ for any $z > 0$; hence, $A_n(z_0)$ is finite for all sufficiently large $n$. Without loss of generality, we restrict our consideration to the case where $A_n(z_0)$ is finite for all $n \in \mbN$, as the general case can be treated similarly with minor technical modifications.

Let $\ve > 0$ be fixed. Since $s_0^{-1}(x) = 0$ for all $x \leq 0$, the locally uniform convergence \eqref{eq:s-1=uniform conv} implies that for any $R > 0$, there exists a constant $\varkappa = \varkappa(\ve, R) > 0$ such that $|s_n^{-1}(x)| \leq \ve$ for all $x \in [-R, 2\varkappa]$ and 
\[
\max_{x \in [-R, R]} |s_n^{-1}(x) - s_0^{-1}(x)| \leq \ve
\] 
for all sufficiently large $n$. 

For all sufficiently large $n$, $A_n(z_0)>T$, and hence $A_n^{-1}(t)\leq z_0, t\in[0,T].$  Let $R:=\max_{z\in[0,z_0]}|B(z)|+1.$ Then for sufficiently large $n $ we have the inequality
$|y_n+B(A_n^{-1}(t))|\leq R, t\in [0,T]$  and   $|X_n(t)|\leq  \ve$ whenever  $y_n+B(A_n^{-1}(t))\leq 2\varkappa, t\in[0,T]$. 

Set
\begin{align}
    a_0 &:= 0, \quad b_k := \inf\{t \geq a_k : B(A_0^{-1}(t)) \leq \varkappa/3\}, \\
    a_{k+1} &:= \inf\{t > b_k : B(A_0^{-1}(t)) \geq 2\varkappa/3\}, \quad k \geq 0.
\end{align}

Note that:
\begin{itemize}
    \item $B(A_0^{-1}(t)) \leq 2\varkappa/3$ if $t \in (b_k, a_{k+1})$ and $B(z) \leq 2\varkappa/3$ if $z \in (A_0^{-1}(b_k), A_0^{-1}(a_{k+1}))$,
    \item $B(A_0^{-1}(t)) \geq \varkappa/3$ if $t \in (a_k, b_{k})$ and $B(z) \geq \varkappa/3$ if $z \in (A_0^{-1}(a_k), A_0^{-1}(b_{k}))$,
    \item there are only  finitely many  indices $k$ such that $a_k \leq T$,
    \item $A_0$ is strictly increasing and continuous on small neighborhoods of any $[A_0^{-1}(a_k), A_0^{-1}(b_k)]$, $k \geq 1$.
\end{itemize}

Corollary \ref{corl:5.1} implies that we have the uniform convergence $A_n^{-1} \to A_0^{-1}$ as $n \to \infty$ on each interval $[a_k, b_k]$, $k \geq 1$ (and similarly on $[a_0, b_0]$ if $b_0 \neq 0$). Hence, we can find a $\delta > 0$ such that $|X_n(t) - X_n(s)| < \ve$ for sufficiently large $n$ and any $s, t \in [0, T]$ satisfying $|t - s| < \delta$, provided that $s$ and $t$ belong to the same interval $[a_k, b_k]$.

Since there are only finitely many indices $k$ such that  $a_k\leq T $, the continuity of $B$ and the convergence of $A_n^{-1}$ at the endpoints $b_k$ and $a_{k+1}$ imply that, for all sufficiently large $n$,  $A_n^{-1}(t)\in[A_n^{-1}(b_k), A_n^{-1}(a_{k+1})], t\in [b_k, a_{k+1}]$, for every such $k$. Thus 
 $ B(A_n^{-1}(t))\le\varkappa,  t\in[b_k,a_{k+1}].$

Assume that $\delta > 0$ is chosen as above, $0 \leq s \leq t \leq T$, $t - s < \delta$, with $s \in [a_k, b_k]$ and $t \in [a_l, b_l]$ for some $k < l$. Then, by splitting the trajectory, we obtain:
\[
|X_n(t) - X_n(s)| \leq |X_n(t) - X_n(a_l)| + |X_n(a_l)| + |X_n(s) - X_n(b_k)| + |X_n(b_k)| \leq 4\ve.
\]

If $s \in [b_k, a_{k+1}]$, $t \in [b_l, a_{l+1}]$,  and $|y_n|<\varkappa,$ then $y_n+(B(A_n^{-1}(s))\vee B(A_n^{-1}(t)))\leq 2\varkappa$, yielding:
\[
|X_n(t) - X_n(s)| \leq |X_n(t)| + |X_n(s)| \leq 2\ve.
\]
All other combinations of intervals can be treated similarly, which implies $|X_n(t) - X_n(s)| \leq 4\ve$ for all $s, t \in [0, T]$ with $|t - s| \leq \delta$, where $|y_n|<\varkappa$. Since all $(X_n)$ are continuous we may again decrease $\delta>0$ if necessary and obtain the inequality for all $n\geq 1.$

Since $\ve > 0$ was arbitrary, 
this proves the uniform equicontinuity of the sequence $(X_n)_{n \geq 1}$ on $[0, T]$, and therefore completes the proof.

\end{proof}

\subsection{snapping-out diffusions}\label{sec:SNOD}

In the previous subsections, we proved the convergence in distribution of diffusions under the assumption that their scale functions and speed measures converge to those of the limiting diffusion, supplemented by natural additional assumptions regarding the behavior of scales and speed measures in neighborhoods of the endpoints of $I_0$. In this section, we investigate the case where the limits of the scales functions and speed measures exist in a certain sense on $(-\infty, 0)$ and $(0, \infty)$, but the limiting scale function possesses a jump at $0$. Under natural assumptions, we will prove that the limit process is a \textit{snapping-out diffusion}, which generalizes the snapping-out Brownian motion introduced and studied from different perspectives in \cite{bobrowski2012pde, lejay2016snapping, mandrekar2016brownian}.

Informally, a snapping-out diffusion $X$ with switching at $0$ operates as follows. Let $X_-$ and $X_+$ be regular diffusions on $(-\infty, 0]$ and $[0, \infty)$, respectively. To distinguish the two copies of the boundary point \(0\), we will sometimes denote the endpoint of $(-\infty, 0]$ by $0-$ and the endpoint of $[0, \infty)$ by $0+$. We assume that $0\pm$ are regular boundary points of $X_\pm$, meaning that $X_\pm$ can hit $0\pm$ with positive probability and can exit from $0\pm$ into the interior of the corresponding interval.  

We construct a strong Markov process $X$ on the disjoint union $(-\infty, 0-] \sqcup [0+, \infty)$ that behaves as $X_+$ if $X(0) \in [0+, \infty)$ until the instant when its local time at $0+$ reaches an independent exponentially distributed random variable $\xi_+ \sim \text{Exp}(\lambda_+)$. If that   moment is finite, the process is reborn at $0-$ and subsequently behaves as $X_-$, independently of the past, until its local time at $0-$ reaches an independent exponentially distributed random variable $\xi_- \sim \text{Exp}(\lambda_-)$. Then it is reborn at $0+$, and the construction is repeated.

Below, we provide a formal construction of this process and clarify the specific choice of local time at $0$ under consideration.

The distribution $\Pb_x$ of $X$ given $X(0) = x \in [0+, \infty)$ is formally constructed via the concatenation of independent paths (the case where $X(0) = x \in (-\infty, 0-]$ is treated analogously). Namely, let $(X^n_\pm)_{n \geq 1}$ be independent copies of $X_\pm$ starting from $0\pm$, respectively, and let $X_+^0$ be a copy of $X_+$ independent of $(X^n_\pm)_{n \geq 1}$ with $X_+^0(0) = x$. We assume that all these diffusions admit the representation \eqref{eq:reg_diff} with independent  Brownian motions $B^n_\pm$, and that the local times $L_\pm^n(t, 0\pm)$ of $X_\pm^n$ at $0\pm$ are determined via \eqref{eq:int_loc_time_X0}. Furthermore, let $(\xi^n_+)_{n \geq 0}$ and $(\xi^n_-)_{n \geq 1}$ be independent exponentially distributed random variables with parameters $\lambda_\pm$, respectively, which are also independent of all the aforementioned diffusions and Brownian motions. Set $\sigma_\pm^n := \inf\{t \geq 0 : L^n_\pm(t, 0\pm) \geq \xi_\pm^n\}$.

For $n \geq 1$, let $\mathcal{T}_n^{-} := \sum_{i=0}^{n-1}\sigma_+^i + \sum_{i=1}^{n-1}\sigma_-^i$ and $\mathcal{T}_n^{+} := \sum_{i=0}^{n-1}\sigma_+^i + \sum_{i=1}^{n}\sigma_-^i$, with the convention that an empty sum equals $0$.

Define the process $X(t)$ by concatenating  parts of trajectories  of $(X_\pm^n)$ as follows:
\begin{equation}
    X(t) := \begin{cases}
        X^0_+(t), & t \in [0, \sigma^0_+), \\
        X^{n}_-(t - \mathcal{T}_n^{-}), & t \in [\mathcal{T}_n^{-}, \, \mathcal{T}_n^{+}), \quad n\geq 1,\\
        X^{n}_+(t - \mathcal{T}_n^{+}), & t \in [\mathcal{T}_n^{+}, \, \mathcal{T}_{n+1}^{-}), \quad n\geq 1.
    \end{cases}
\end{equation}
 If some \(\sigma_\pm^n=\infty\), the concatenation stops at that stage and no further switching occurs.

If initial point $x=X(0)\in(-\infty, 0-]$, then $(X(t))$ is defined similarly.  It follows from the general results on path concatenation (see, e.g., \cite{ikedaconc}) that $(X, \Pb_x)$ defines a strong Markov process on $(-\infty, 0-] \sqcup [0+, \infty)$.

\begin{defn}
The process $(X(t))_{t \geq 0}$ constructed above is called the \textit{snapping-out diffusion} with the interface located at $0$, scale functions $s_\pm$, speed measures $m_\pm$, and transmission intensities $\lambda_\pm$. In particular, if $s_\pm(x)=x,\, \pm x\geq 0, \quad m_\pm(\di x)=\di x,\, \pm x\geq 0$, then the snapping-out diffusion is called the \textit{snapping-out Brownian motion}.
\end{defn}

\begin{remk}
    \label{remk:SNOD-different representation}
    Note that  $s_\pm, m_\pm$ for  $X_\pm$ from the definition of the snapping-out diffusion, and consequently their local times at 0, are not determined uniquely. Indeed, let $c_\pm>0$ be constants. Then $\wt s_\pm(x):=c_\pm s_\pm(x), \wt m_\pm(\di x):=c_\pm^{-1} m_\pm(\di x)$ are scales and speed measures of the same diffusions. By \eqref{eq:loc time = limit}, the corresponding local times satisfy
   \[
 \wt L_\pm(t,0\pm)=c_\pm L_\pm(t,0\pm).
  \]
   Hence, a  {snapping-out diffusion} with   scale functions $s_\pm$, speed measures $m_\pm$, and transmission intensities $\lambda_\pm$ have the same distribution as  {snapping-out diffusion} with   scale functions $\wt s_\pm$, speed measures $\wt m_\pm$, and transmission intensities $\wt \lambda_\pm:= c^{-1}_\pm \lambda_\pm.$ 
    
    Choosing $c_-/c_+=\lambda_-/\lambda_+$, we may always assume that the transmission intensities are equal. Moreover, if we take $c_\pm=\lambda_\pm$, then  $\wt \lambda_\pm=1.$
\end{remk}
\begin{remk}
By abuse of terminology, the canonical projection of the process $X$ onto $\mbR$ (identifying $0-$ and $0+$ with $0$) will also be referred to as a snapping-out diffusion. 
\end{remk}

\begin{remk}\label{remk:boundary SNOD} 
The boundary conditions for the generator  of the process $X$ at the interfaces $0\pm$ are given by the system:
\begin{equation}\label{eq:SNOD-boundary conditions}
    \begin{cases}
        \frac{\mathrm{d} u(0+)}{\mathrm{d} s_+} = 2 m_+(\{0+\}) \cA_+ u(0+) + 2\lambda_+ \big(u(0+) - u(0-)\big), \\
        \frac{\mathrm{d} u(0-)}{\mathrm{d} s_-} = -2 m_-(\{0-\}) \cA_- u(0-) + 2\lambda_- \big(u(0+) - u(0-)\big).
    \end{cases}
\end{equation}
Indeed, consider for example the point $0+.$  First, we  recall classical  calculations of Dynkin's characteristic operator (which coincides with the generator) for an  elastic diffusion at $0+$, that is, for the process $X_+$ that is killed when the local time   at $0+$ reaches an independent exponentially distributed random variable $\xi_+$ with parameter 
 $\lambda_+$; see \cite[\S 4.5, formula (5b)]{ItoMcKean65}. We then explain the slight difference between elastic and snapping-out diffusions.

Denote $\sigma_\ve:=\inf\{s\geq0:X_+(s)\geq\ve\}$ 
and $\sigma_\infty:=\inf\{s\geq0:L_{X_+}(s,0+)\geq\xi_+\},$ where $\sigma_\infty$ is the killing time.
 
Dynkin's characteristic operator equals
\[
\cA_+^{elastic}u(0+)
=
\lim_{\ve\to0+}
\frac{
\E_{0+}\!\left[
u(X_+(\sigma_\ve))
\1_{\{\sigma_\ve<\sigma_\infty\}}
\right]-u(0+)}
{\E_{0+}(\sigma_\ve\wedge\sigma_\infty)}.
\]

After elementary transformations, one obtains, see \cite[\S 4.5, formula (10)]{ItoMcKean65}, that
   \begin{align}
 \cA_+^{elastic} u(0+) 
\frac{\E_{0+} (\sigma_\ve\wedge\sigma_\infty) \Pb_{0+}(\sigma_\ve<\sigma_\infty)^{-1}}{(s_+(\ve)-s_+(0+)) }= \\
\frac{u(\ve)-u(0+)}{s_+(\ve)-s_+(0+)} -  
\frac{  {\Pb_{0+}(\sigma_\ve<\sigma_\infty)^{-1}-1} }{s_+(\ve)-s_+(0+)} u(0+) +o(1),\ \ve \to0+.
\label{eq:elastic_BC}
\end{align} 
 
It is known, see \cite[\S 4.5, formula (11a), (11b)]{ItoMcKean65}, that 
\[
\begin{aligned}
\frac{  {\Pb_{0+}(\sigma_\ve<\sigma_\infty)^{-1}-1} }{s_+(\ve)-s_+(0+)}
&=2\lambda_++o(1)\quad \text{and}\\
\frac{\E_{0+} (\sigma_\ve\wedge\sigma_\infty) \Pb_{0+}(\sigma_\ve<\sigma_\infty)^{-1}}{s_+(\ve)-s_+(0+) }
&=2m_+(\{0+\})+o(1),
\end{aligned}
\]
which yields immediately the boundary condition for the elastic diffusion:
 \[
\frac{\mathrm{d} u(0+)}{\mathrm{d} s_+} = 2 m_+(\{0+\}) \cA^{elastic}_+ u(0+) + 2\lambda_+   u(0+).
\]
Here   the factor 2 appears because the definition of the generator differs by a factor of 2 in this paper and in \cite{ItoMcKean65}.  

 For the snapping-out diffusion\footnote{We are grateful to Adam Bobrowski who suggested this argument.}, killing at $\sigma_\infty$ is replaced by an instantaneous rebirth at $0-$. Therefore, in the calculation above, the contribution  $0$ corresponding to the cemetery state is replaced by  $u(0-)$  and we obtain one extra term $\frac{  {\Pb_{0+}(\sigma_\ve<\sigma_\infty)^{-1}-1} }{s_+(\ve)-s_+(0+)} u(0-)$ in the right-hand side of the analogue of \eqref{eq:elastic_BC}. Consequently,
$$ \frac{\mathrm d u(0+)}{\mathrm d s_+} = 2m_+(\{0+\})\cA_+u(0+) + 2\lambda_+\bigl(u(0+)-u(0-)\bigr). $$

The boundary condition at $0-$ is obtained analogously.
\end{remk}

The main result of this section is formulated in the following theorem. Similarly to the previous sections use the  convention concerning the  extension of   $s_\pm$ and $m_\pm$   to the whole real  $\mbR$.

\begin{thm}\label{thm:limit_SNOdiffosion}
Let $(X_n)_{n \geq 1}$ be a sequence of regular diffusions on $I_n = \mbR$ with scale functions $s_n$ and speed measures $m_n$ for $n \geq 1$. Assume that there exist strictly increasing, continuous functions $s_+ \colon [0, \infty) \to \mbR$ and $s_- \colon (-\infty, 0] \to \mbR$ and locally finite measures $m_\pm$ on $[0, \infty)$ and $(-\infty, 0]$, respectively, such that:
\begin{enumerate}
    \item $s_\pm$ and $m_\pm$ are the scale functions and speed measures of regular diffusions on $[0, \infty)$ and $(-\infty, 0]$, respectively.
    \item $\lim_{n \to \infty} s_n(x) = s_-(x)$ for $x < 0$, $\lim_{n \to \infty} s_n(x) = s_+(x)$ for $x > 0$, and $s_-(0) < s_+(0)$.
    \item for any continuous, compactly supported functions $f_\pm$ on $\mbR$ with $\mathrm{supp}(f_-) \subset (-\infty, s_+(0))$ and $\mathrm{supp}(f_+) \subset (s_-(0), \infty)$, we have   convergence $\lim_{n \to \infty} \int_{\mbR} f_\pm \, \mathrm{d}\wt{m}_n = \int_{\mbR} f_\pm \, \mathrm{d}\wt{m}_\pm$, where $\wt{m}_n := m_n \circ s^{-1}_n$ and $\wt{m}_\pm := m_\pm \circ s^{-1}_\pm$. Equivalently, we have the vague convergence $\limn  \wt  m_n = \wt m_-$ on $(-\infty, s_+(0))$ and $\limn  \wt  m_n = \wt m_+$ on $( s_-(0), \infty)$.
    \item $X_n(0) \Rightarrow x$ as $n \to \infty$, where $x \neq 0$.
\end{enumerate}
Then the sequence of processes $(X_n)_{n \geq 1}$ converges in distribution in $C([0,\infty))$ to the snapping-out diffusion starting from $x$ with the interface  at $0$, scale functions $s_\pm$, speed measures $m_\pm$, and symmetric transmission intensities 
\[
\lambda_- = \lambda_+ = \frac{1}{2\big(s_+(0) - s_-(0)\big)}.
\]
\end{thm}
\begin{proof}
Since $\wt m_+((-\infty, s_+(0))) = \wt m_-(( s_-(0),\infty))=0,$ both limiting measures assign zero mass to the interval $ (s_-(0), s_+(0)).$ By a diagonal argument, we can therefore choose sequences   $(x_{n,-})$ and $(x_{n,+})$ such that 
\[
\begin{aligned}
x_{n,-} &< x_{n,+}, \quad \lim_{n \to \infty} x_{n,\pm} = 0, \\
\lim_{n \to \infty} s_n(x_{n,-}) &= s_-(0), \quad
\lim_{n \to \infty} s_n(x_{n,+}) = s_+(0),
\end{aligned}
\]
and $\lim_{n \to \infty} m_n([x_{n,-}, x_{n,+}]) = \lim_{n \to \infty} \wt m_n([s_n(x_{n,-}), s_n(x_{n,+})]) = 0$. 

Denote by $X_{n,-}$ the regular diffusion on $(-\infty, x_{n,+}]$ and by $X_{n,+}$ the regular diffusion on $[x_{n,-}, \infty)$, whose scales and speed measures are the restrictions of $s_n$ and $m_n$ to the corresponding intervals, respectively. Let $X_-$ and $X_+$ be the regular diffusions on $(-\infty, 0]$ and $[0, \infty)$, respectively, with scale functions $s_\pm$ and speed measures $m_\pm$. 

Theorem \ref{thm:regular point} implies that if $X_{n,-}(0) = y_n$, where $\lim_{n \to \infty} y_n = y \leq 0$ and $\lim_{n \to \infty} s_n(y_n) = s_-(y)$, then $(X_{n,-})_{n \geq 1}$ converges in distribution to $X_-$ with $X_-(0) = y$. Analogously, if $\lim_{n \to \infty} y_n = y \geq 0$ and $\lim_{n \to \infty} s_n(y_n) = s_+(y)$, then $(X_{n,+})_{n \geq 1}$ converges in distribution to $X_+$ with $X_+(0) = y$. 

To analyze the switching behavior, we define the hitting times:
\begin{equation}
\tau_{n,-} := \inf\{t \geq 0 : X_{n,-}(t) = x_{n,+}\}, \quad \tau_{n,+} := \inf\{t \geq 0 : X_{n,+}(t) = x_{n,-}\},
\end{equation}
and
$$\sigma_\pm := \inf\{t \geq 0 : L_\pm(t, 0\pm ) \geq \xi\},$$ 
where $L_\pm(t, 0\pm)$ is the local time at $0\pm$ of the diffusion $X_\pm$, and $\xi \sim \text{Exp}\big((2(s_+(0) - s_-(0)))^{-1}\big)$ is an exponentially distributed random variable independent of $X_\pm$.

Since $X_{n,-}$ and $X_{n,+}$ behave exactly like the original process $X_n$ until hitting $x_{n,+}$ and $x_{n,-}$, respectively, it is  enough to identify the limit of each excursion between successive switching times; the full result then follows by iterating this argument using the strong Markov property. Thus, it remains to prove that  $X_{n,-}(\cdot \wedge \tau_{n,-}) \Rightarrow X_-(\cdot \wedge \sigma_-)$ and $X_{n,+}(\cdot \wedge \tau_{n,+}) \Rightarrow X_+(\cdot \wedge \sigma_+)$ as $n \to \infty$, provided that $\lim_{n \to \infty} s_n(X_{n,-}(0)) = s_-(X_-(0))$ and $\lim_{n \to \infty} s_n(X_{n,+}(0)) = s_+(X_+(0))$. 

We will prove only $X_{n,+}(\cdot \wedge\tau_{n,+})\Rightarrow X_+(\cdot\wedge\sigma_+)$. Without loss of generality, we may assume that $x_{n,+}=0, s_n(x_{n,+})=0=s_+(0).$ We will also assume that $X_{n,+}(0)=X_+(0)=0$, since the behavior of $X_{n,+}$ stopped at the first   hitting time of  0 is described by Corollary \ref{corl:conv}.

We construct all $X_+$ and $X_{n,+}$ for $n \geq 1$ using the same standard Brownian motion $B$:
\[
X_+(t) = s_+^{-1}\big(B(A_+^{-1}(t))\big), \quad X_{n,+}(t) = s_n^{-1}\big(B(A_{n,+}^{-1}(t))\big), \quad t \geq 0,
\]
where 
\begin{equation}
    A_+(t) = \int_{[0,\infty)} l_B(t, s_+(y)) \, m_+(\mathrm{d} y), \quad A_{n,+}(t) = \int_{[x_{n,-},\infty)} l_B(t, s_n(y)) \, m_n(\mathrm{d} y).
    \label{eq:A_+ A_n+}
\end{equation}

Similarly to the convention in the previous proofs, we consider only those $\omega \in \Omega$ for which the sample path of the Brownian motion possesses standard properties. We will omit ``a.s.'' in most places.

It follows from the assumptions of our theorem and the proof of Theorem \ref{thm:regular point} that 
\begin{equation}
    \label{eq:1109}
    \lim_{n \to \infty} X_{n,+} = X_+, \quad \lim_{n \to \infty} A_{n,+} = A_+, 
\end{equation}
 locally uniformly.

Note that 
\begin{align*}
\tau_{n,+}
&= \inf\{t \geq 0 : X_{n, +}(t) = x_{n, -}\} \\
&= \inf\{t \geq 0 : B(A_{n,+}^{-1}(t)) = s_n(x_{n, -})\} \\
&= \inf\{A_{n,+}(z) : B(z) = s_n(x_{n,-})\} \\
&= A_{n,+}\left(\inf\{z\geq 0 : B(z) = s_n(x_{n,-})\}\right).
\end{align*}
The third equality follows from the following property of the Brownian motion. If $z_0=\inf\{z : B(z) = s_n(x_{n,-})\}$ is the first moment of hitting the point $s_n(x_{n,-}),$ then a.s.  any right neighborhood   of $z_0$ contains points $z$ such that $B(z)>s_n(x_{n,-}).$ Therefore, $z_0$ is a point of increase of $ A_{n,+}$ and by the definition of the generalized inverse function we have equality $ A_{n,+}^{-1}( A_{n,+}(z_0))=z_0$. More general, if $z$ is such that $B(z)>s_n(x_{n,-})$, then  $z$ is a point of increase of $ A_{n,+}$, and hence  $ A_{n,+}^{-1}( A_{n,+}(z))=z$. In contrast, this identity need not hold when for $z$ such that $B(z)<s_n(x_{n,-})$, since $A_{n,+}$ is then locally constant.
It follows from \eqref{eq:1109}, the continuity of $A_+$, Lemma \ref{lem:hitting_limit}, and standard  sample-path properties  of Brownian motion that 
\[
\lim_{n \to \infty} \tau_{n,+} =  
A_{+}\left(\inf\{z\geq 0 : B(z) = s_-(0)\}\right)=: \tau_+.
\]
This and \eqref{eq:1109} imply the locally uniform convergence 
\begin{equation*}
\begin{aligned}
\lim_{n \to \infty} X_{n,+}(\cdot \wedge \tau_{n,+})
&= X_+(\cdot \wedge \tau_+) \\
&= s_+^{-1}\big(B(A_+^{-1}(\cdot \wedge \tau_+))\big).
\end{aligned}
\end{equation*}
\begin{remk}
    Note that $\tau_+\neq  \inf\{t\geq 0 : B(A_{+}^{-1}(t)) = s_-(0)\}$ because $B(A_{+}^{-1}(t))$ is always non-negative.
\end{remk}
Denote   $\tau^B(s_-(0)) := \inf\{z : B(z) = s_-(0)\}$. So $\tau_+= A_+(\tau^B(s_-(0))).$
\begin{lem}
We have equality
\[
\begin{aligned}
\Big(s_+^{-1}\big(B(A_+^{-1}( t \wedge \tau_+))\big)\Big)_{t\geq 0}
&= \Big(X_+(t \wedge \tau_+) \Big)_{t\geq 0} \\
&= \Big(X_+(t)\1_{t\leq \tau_+} \Big)_{t\geq 0} \\
&=\Big(X_+(t) \, \mathbb{I}_{\{A_+^{-1}( t) \leq  \tau^B(s_-(0))\}}\Big)_{t\geq 0}.
\end{aligned}
\]
\end{lem}
\begin{proof}
     The first equality is just the definition. The second one follows from the fact that $X_+(\tau_+)=0.$ Note that, generally, $A_+^{-1}(A_+(t))\neq t$ because $A_+$ is not  a strictly increasing function and some   care is needed in proving  the last equality. 
     Let $(z_-, z_+)$ be the maximal interval containing $\tau^B(s_-(0))$ and  such that $B(z)<0, z\in (z_-, z_+). $ Then  $B (z_-)=B(z_+)=0, $  $A_+(z)=A_+(\tau^B(s_-(0))), z\in (z_-, z_+), \quad A_+(z)<A_+(\tau^B(s_-(0))), z<z_-,$ and $A_+(z)>A_+(\tau^B(s_-(0))), z>z_+$. Therefore 
     \[
\begin{aligned}
A_+^{-1}(t)&<z_-< \tau^B(s_-(0)) \text{ for }  t<\tau_+, \\
A_+^{-1}(\tau_+)&=z_+, \text{ and} \\
A_+^{-1}(t)&>z_+ \text{ for }  t>\tau_+.
\end{aligned}
     \]
The only point where the indicators may not match is  $\tau_+.$ In any case, 
$$X_+(\tau_+)=s_+^{-1}(B(A_+^{-1}(\tau_+)))=s_+^{-1}(B(z_+))= s_+^{-1}(0)=0.$$ 
The lemma is proved.
     \end{proof} 

So, it remains only to verify that
\[
\begin{aligned}
&\Big( s_+^{-1}\big(B(A_+^{-1}(t))\big)
\mathbb{I}_{\{A_+^{-1}(t) \leq \tau^B(s_-(0))\}} \Big)_{t \geq 0} \\
&\qquad \overset{d}{=} \big(X_+(t \wedge \sigma_+)\big)_{t \geq 0}.
\end{aligned}
\]
Recall that the local time at $0+$ may increase only when the process visits $0+$. Hence,  $X_+(\sigma_+)=0$ and
\[
\begin{aligned}
X_+(t \wedge \sigma_+)
&= X_+(t) \mathbb{I}_{\{t < \sigma_+\}} \\
&= X_+(t) \mathbb{I}_{\{t \leq \sigma_+\}} \\
&= X_+(t) \mathbb{I}_{\{L_+(t,0+) < \xi\}} \\
&= X_+(t) \mathbb{I}_{\{L_+(t,0+) \leq  \xi\}},
\end{aligned}
\]
where $\xi \sim \text{Exp}\big((2(s_+(0) - s_-(0)))^{-1}\big)$ is independent of $B$. 
The last equality is true  since \(\xi\) is independent of \(X_+\) and has a continuous distribution. So, almost surely it  does not belong to the at most countable set of  levels on which \(L_+(\cdot,0+)\) is constant on a nondegenerate interval.  

Therefore, the theorem will be proven if we show that the following identity in distribution holds:
\[
\begin{aligned}
&\Big( s_+^{-1}\big(B(A_+^{-1}(t))\big)
\mathbb{I}_{\{A_+^{-1}(t) < \tau^B(s_-(0))\}} \Big)_{t \geq 0} \\
&\qquad \overset{d}{=} \Big( s_+^{-1}\big(B(A_+^{-1}(t))\big)
\mathbb{I}_{\{L_+(t,0+) < \xi\}} \Big)_{t \geq 0}
\end{aligned}
\]
or
\[
\begin{aligned}
&\Big(  B(A_+^{-1}(t))
\mathbb{I}_{\{A_+^{-1}(t) \leq  \tau^B(s_-(0))\}} \Big)_{t \geq 0} \\
&\qquad \overset{d}{=} \Big( B(A_+^{-1}(t))
\mathbb{I}_{\{L_+(t,0+) \leq \xi\}} \Big)_{t \geq 0}
\end{aligned}
\]

First, we  prove this result in the case where $m_+(\mathrm{d} x) = \mathbb{I}_{\{x \geq 0\}} \, \mathrm{d} x$ and $s_+(x)=x, x\geq 0$, that is if $B_+ $ is a reflected Brownian motion.

Recall the key elements of Itô's excursion theory and the construction of a standard Brownian motion as a concatenation of excursions, cf. \cite{Blumenthal2012excursions}, where a more general construction is discussed in full detail. 

Let $\cC$ denote the space of excursions, i.e., the set of continuous functions $\omega \colon [0,\infty) \to \mbR$ such that $\omega(0)=0$ and there exists a lifetime $\sigma = \sigma(\omega) \in (0,\infty)$ such that $\omega(t) \neq 0$ for $t \in (0,\sigma)$ and $\omega(t)=0$ for $t \geq \sigma$. Let $\cC_+$ and $\cC_-$ be the subsets of $\cC$ consisting of positive and negative excursions, respectively. 

Let $N = \sum_k \delta_{(z_k, \omega_k)}$ be a point measure on $[0,\infty) \times \cC$. Set 
\[
\varphi(z) := \sum_{z_k \leq z} \sigma(\omega_k), \quad \varphi^{-1}(t) := \inf\{z \geq 0 : \varphi(z) > t\},
\]
and define the process $B(t) := \omega_k\big(t - \varphi(z_k-)\big)$ for $t \in [\varphi(z_k-), \varphi(z_k))$, and $B(t) := 0$ otherwise. If all these objects $\varphi, \varphi^{-1}$, and $B$ are well-defined, we say that $B$ is obtained via Itô's construction from $N$.

It is well known that there exists a $\sigma$-finite measure $\nu$ on $\cC$ and a Poisson point measure $N = \sum_k \delta_{(z_k, \omega_k)}$ on $[0,\infty) \times \cC$ with intensity $\mathrm{d}z \times \nu(\mathrm{d}\omega)$ such that the process $B$ obtained via Itô's construction from $N$ is a standard Brownian motion and, additionally, the process $(\varphi^{-1}(t))_{t\geq0}$,  coincides with the symmetric local time of $B$ at $0$, defined as 
\[
l_B(t) = \lim_{\varepsilon \to 0+} \frac{1}{2\varepsilon} \int_0^t \mathbb{I}_{\{|B(z)| \leq \varepsilon\}} \, \mathrm{d} z.
\]
That is, we have the identity 
\begin{equation}
    \label{eq:local time excursion}
\begin{aligned}
l_B(t)
&= \inf\biggl\{ z  \geq 0 : \sum_{z_k \leq z} \sigma(\omega_k) > t \biggr\} \\
&= \inf\biggl\{ z  \geq 0 : \sum_{z_k \leq z} \sigma(\omega_k) \geq t \biggr\}, \quad t \geq 0.
\end{aligned}
\end{equation}
Since $\varphi$ is a strictly increasing subordinator, one can use either strict or non-strict inequality under the infimum sign. It is also a standard fact that for a Brownian motion, this symmetric local time equals the right-sided local time, which is defined as 
\[
\lim_{\varepsilon \to 0+} \frac{1}{\varepsilon} \int_0^t \mathbb{I}_{\{B(z) \in [0, \varepsilon]\}} \, \mathrm{d} z.
\]

Recall that we have assumed $m_+(\mathrm{d} x) = \mathbb{I}_{\{x \geq 0\}} \, \mathrm{d} x$. Thus, $A_+(t) = \int_0^t \mathbb{I}_{\{B(u) \geq  0\}} \, \mathrm{d} u$. Denote by $N_+$ and $N_-$ the restrictions of the point measure $N$ to the sets $[0,\infty) \times \cC_+$ and $[0,\infty) \times \cC_-$, respectively.

Note that the process $B_+(t) := B(A_+^{-1}(t))$, which is the Brownian motion reflected at $0$, can be obtained via Itô's construction directly from the restriction $N_+$. Moreover, it follows from Remark \ref{remk:local_time} that, on a set of probability $1$, the local time $l_{B_+}(t)$ satisfies:
\[
\begin{aligned}
\lim_{\varepsilon \to 0+} \frac{1}{\varepsilon}
\int_0^t \mathbb{I}_{\{B_+(z) \in [0,\varepsilon]\}} \, \mathrm{d} z
&= l_{B_+}(t) \\
&= l_B(A_+^{-1}(t)) \\
&= \lim_{\varepsilon \to 0+} \frac{1}{\varepsilon}
\int_0^{A_+^{-1}(t)} \mathbb{I}_{\{B(s) \in  [0,\varepsilon]\}} \, \mathrm{d} s.
\end{aligned}
\]

Assume that $B_+(t)=B(A^{-1}_+(t)) > 0$ and that the underlying Brownian motion $B$ is inside the excursion $\omega_K$ at the time instant $A^{-1}_+(t)$. 
Then
\[
\sum_{z_k < z_K} \sigma(\omega_k) < A_+^{-1}(t) < \sum_{z_k \leq z_K} \sigma(\omega_k).
\]
On the other hand, $B_+$ is obtained from Itô's construction from $N_+$ and  $t = \int_0^t \mathbb{I}_{\{B_+(u) > 0\}} \, \mathrm{d} u$. So,
\[
\sum_{z_k < z_K} \sigma(\omega_k) \mathbb{I}_{\{\omega_k \in \cC_+\}} < t < \sum_{z_k \leq z_K} \sigma(\omega_k) \mathbb{I}_{\{\omega_k \in \cC_+\}}.
\]

Since $l_B = \varphi^{-1}$, it follows that
\[
l_B(A_+^{-1}(t)) = z_K = \inf\biggl\{z \geq 0 : \sum_{z_k \leq z} \sigma(\omega_k) \mathbb{I}_{\{\omega_k \in \cC_+\}} \geq t\biggr\}.
\]
The set of points $t \geq 0$ for which $B_+(t) = B(A_+^{-1}(t)) > 0$ is dense in $[0,\infty)$. The left-hand side of the last equality equals $l_{B_+}(t)$ and hence is continuous in $t.$ The right-hand side is also continuous in $t$ because it is the inverse of a strictly increasing subordinator.  Therefore
\[
l_{B_+}(t) = l_B(A_+^{-1}(t)) = \inf\biggl\{z \geq 0 : \sum_{z_k \leq z} \sigma(\omega_k) \mathbb{I}_{\{\omega_k \in \cC_+\}} \geq t\biggr\}
\]
holds for all $t \geq 0$.

Let $k^*$ be the index of the first excursion of $B$ that hits $-a$, where $a:=-s_-(0)>0$, i.e.,
\[
z_{k^*} = \inf \{z_k : \exists u > 0 \ \omega_k(u) = -a\}.
\]
Since the restrictions $N_+$ and $N_-$ of the Poisson point measure to the disjoint sets $[0,\infty)\times\cC_+$ and $[0,\infty)\times\cC_-$ are independent, the random variable $z_{k^*}$ is independent of $B_+$. Moreover, $z_{k^*}$ is exponentially distributed, as it represents the first instance $z_k$ where the second coordinate $\omega_k$ of the Poisson point process belongs to a specific deterministic subset of excursions. By construction, see \eqref{eq:local time excursion}, $z_{k^*}$ is the local time that $B$ accumulates at $0$ before its first hitting time of $-a$. It is a well-known fact, see, e.g., \cite[p.~199, formula (2.3.1)]{borodin2002handbook}, that $z_{k^*} \sim \text{Exp}\big(\frac{1}{2|a|}\big)$.

Let $t^* := \inf\{t \geq 0 : B(t) = -a\}$. Then  $z_{k^*} = \vf^{-1}(t^*)= l_B(t^*)$. Let $t \geq 0$ be such that $B(A_+^{-1}(t)) > 0$. Then $A_+^{-1}(t) \neq t^*$, and $A_+^{-1}$ is strictly increasing in a neighborhood of $t$, and we have the following chain of equivalent events:
\begin{align*}
\{A_+^{-1}(t) \leq t^*\}
&= \{A_+^{-1}(t) < t^*\} \\
&= \{t < A_+(t^*)\} \\
&= \biggl\{t < \sum_{z_k \leq z_{k^*}} \sigma(\omega_k) \mathbb{I}_{\{\omega_k \in \cC_+\}}\biggr\} \\
&= \biggl\{t < \sum_{z_k < z_{k^*}} \sigma(\omega_k) \mathbb{I}_{\{\omega_k \in \cC_+\}}\biggr\}.
\end{align*}
Consequently, this implies that $\{l_{B_+}(t) < z_{k^*}\} \subset \{A_+^{-1}(t) \leq t^*\} \subset \{l_{B_+}(t) \leq z_{k^*}\}$ whenever $B_+(t) > 0$.

There are at most countably many levels on which $l_{B_+}$ is constant on a nondegenerate interval, meaning there is at most a countable set of points $(x_k)$ such that for any $k$, the equation $l_{B_+}(t) = x_k$ has multiple solutions in $t \geq 0$. Since the distribution of $z_{k^*}$ is continuous and independent of $B_+$, it follows that with probability $1$ there exists a unique point $\hat{t} \geq 0$ such that $l_{B_+}(\hat{t}) = z_{k^*}$. This $\hat{t}$ must be a growth point of $l_{B_+}$, which implies $B_+(\hat{t}) = 0$. 

It is easy to see that if $t < \hat{t}$, then $l_{B_+}(t) < z_{k^*}$ and $A_+^{-1}(t) \leq t^*$. Conversely, if $t > \hat{t}$, then $l_{B_+}(t) > z_{k^*}$ and $A_+^{-1}(t) \geq t^*$. Combining these observations, we conclude that with probability $1$, the following sample-path identity holds:
\[
\big( B_+(t) \mathbb{I}_{\{A_+^{-1}(t) \leq t^*\}} \big)_{t \geq 0} = \big( B_+(t) \mathbb{I}_{\{l_{B_+}(t) \leq z_{k^*}\}} \big)_{t \geq 0}.
\]
The latter process has the same distribution as $\big( B_+(t) \mathbb{I}_{\{l_{B_+}(t) \leq \xi\}} \big)_{t \geq 0}$, where $\xi \sim \text{Exp}\big(\frac{1}{2|a|}\big)$, which completes the proof in the reflected Brownian motion case, i.e., when  $s_+(x)=x,\quad m_+(\di x) = \1_{x\geq 0}\di x.$ 

We now pass to the general case. By the subordination and chain rules for one-dimensional diffusions, see \cite[Section 5.5]{ItoMcKean65},  the diffusion \(X_+\) can be obtained from a reflected Brownian motion \(B_+\) by a space-time transformation  determined by the corresponding scale function and  speed measure, see \eqref{eq:reg_diff_refl} and \eqref{eq:reg_diff_CAF_refl}.  
The time change modifies the time parametrization of the excursions, but preserves their order and their parametrization by the local time at the reflecting boundary. 
In view of the relation between the corresponding local times in \eqref{eq:int_loc_time_X0}, this time change does not alter the local-time parameter of the excursions. Hence, the random local-time level which, in the Brownian construction above, corresponds to the first excursion hitting the level \(-a=s_-(0)\) has the same distribution in the general case. Consequently, it has the same exponential distribution
$$ \textrm{Exp}\left(\frac1{2a}\right) = \textrm{Exp}\left(\frac1{2(s_+(0)-s_-(0))}\right) $$
as in the reflected Brownian motion case. Applying the inverse scale transformation $s_+^{-1}$ then yields the required assertion for $X_+$. The argument for $X_-$ is analogous.
\end{proof}
Similarly to the proof of Theorem \ref{thm:regular point},   using the representation \eqref{eq:representation X_n}--\eqref{eq:representation A_n} and the approximation provided by Theorem \ref{thm:limit_SNOdiffosion}, we obtain the following representation of a snapping-out diffusion.
\begin{thm}
    \label{thm:SNO_representation}
    Let $s_\pm, m_\pm$ be from Theorem \ref{thm:limit_SNOdiffosion}. Set
    \[
    \begin{aligned}
    s(x)&:=s_-(x)\1_{x<0}+s_+(x)\1_{x\geq 0},\\
    \wt m&:=\wt m_-+\wt m_+=m_-\circ s_-^{-1} + m_+\circ s_+^{-1}.
    \end{aligned}
    \]
    Then for any Brownian motion $B$ the process $s^{-1}(B(A^{-1}(t)))$, where $A(t)=\int_\mbR l_B(t,y)\wt m(\di y)$ is a snapping-out diffusion with  the interface located at $0$, scale functions $s_\pm$, speed measures $m_\pm$, and equal transmission intensities $ \frac{1}{2(s_+(0) - s_-(0))}.$   
\end{thm}
\begin{remk}
    If $B(0)=s(x)\notin [s_-(0), s_+(0)]$, then the snapping-out diffusion starts from $x.$ Otherwise, if $B(0)=y\in [s_-(0), s_+(0) ]$, then the snapping-out diffusion starts from   $0+$ with probability $\frac{y - s_-(0)}{s_+(0) - s_-(0)}$ and from $0-$ with probability $\frac{s_+(0) - y}{s_+(0) - s_-(0)}$. Indeed, the Brownian time spent inside the interval $ (s_-(0),s_+(0))$  is collapsed by the time change, and the initial side is determined by the endpoint through which $B$ first exits this interval.
\end{remk}
\begin{remk}\label{remk:SNOB construction}  
    Let $X$ be a snapping-out Brownian motion with  transmission intensities $\lambda_\pm>0$, i.e., assume that  in the definition of a snapping-out diffusion we used standard Brownian motions $B_\pm$ reflecting at 0  into the positive and negative half-lines, respectively,  and their one-sided local times at $0\pm$. 
       We apply Theorem \ref{thm:SNO_representation} and construct a process that has the same distribution as $X$,  which is a function of a Brownian motion. 
     It follows from Remark \ref{remk:SNOD-different representation} and representation of the local time from Remark \ref{remk:local_time} that we have to choose affine scale functions   $s_\pm$ and speed measures $m_\pm$ such that 
    \[
  \frac{1}{2\big(s_+(0) - s_-(0)\big)}= \frac{\lambda_-}{s_-'(0-)}= \frac{\lambda_+}{s_+'(0+)},\]
     \[
      m_-(\di x)= (s_-'(0-) )^{-1} \1_{x<0} \di x,\quad m_+(\di x)= (s_+'(0+) )^{-1} \1_{x>0} \di x.
    \]
    In particular, we may select them  such that 
    \[
    \begin{aligned}
    s_-(x) &=\frac{\lambda_-}{\lambda_+}x  -\frac{1}{2\lambda_+}, \quad x<0, \\
    s_+(x)&=  x, \quad x>0, \\
    m_-(\di x)&= \frac{\lambda_+}{\lambda_-} \1_{x<0} \di x, \\
    m_+(\di x)&=   \1_{x>0} \di x.
    \end{aligned}
    \]
    So,
    \begin{equation}
        \label{eq:SNOB_s,m}
    \begin{aligned}
    s^{-1}(x)
    &=\frac{\lambda_+}{\lambda_-}(x+\frac{1}{2\lambda_+})\1_{x<-\frac{1}{2\lambda_+}}+  x\1_{x>0}, \\
    \wt m(\di x)
    &= (\frac{\lambda_+^2}{\lambda_-^2} \1_{x<-\frac{1}{2\lambda_+}}+ \1_{x>0}) \di x.
    \end{aligned}
    \end{equation}
    
\end{remk}
\section{Applications}\label{sec:expl of application}
In this section we give examples of applications of the results from \S \ref{sec:Main results} to several models. We consider sequences of multi-skewed Brownian motions in which several  membranes merge into a single membrane in the limit. 
Depending on the scaling of the membrane permeabilities and the distances
between them, the limiting process may be a skew Brownian motion,
a skew-sticky Brownian motion, or a snapping-out Brownian motion.  Recall also that convergence in distribution of Markov processes implies, in appropriate sense, convergence of solutions to the backward Kolmogorov equations; see Remark \ref{remk: relation Kolmogorov and weak convergence}.

\subsection{Approximation of delta-valued drift by integrated ones}
 \label{expl:Portenko}  
    Let $a \colon \mathbb{R} \to \mathbb{R}$ be a continuous, bounded, and integrable function, and set $a_\varepsilon(x) := \frac{1}{\varepsilon} a\big(\frac{x}{\varepsilon}\big)$. Denote $\alpha := \int_\mathbb{R} a(x) \, \mathrm{d} x = \int_\mathbb{R} a_\varepsilon(x) \, \mathrm{d} x$.

    Assume that $X_\varepsilon$ satisfies the SDE
    \[
    \mathrm{d} X_\varepsilon(t) = \mathrm{d} W(t) + a_\varepsilon(X_\varepsilon(t)) \, \mathrm{d} t, \quad t \geq 0.
    \]
    It is well known that this SDE has a unique (strong) solution, which is a regular diffusion on $\mathbb{R}$. Recall that its scale function $s_\varepsilon$ and speed measure $m_\varepsilon$ are given by
    \[
    s_\varepsilon(x) := \int_0^x \exp\left(-2\int_0^y a_\varepsilon(z) \, \mathrm{d} z\right) \mathrm{d} y, \quad x \in \mathbb{R},
    \]
    \[
    \mathrm{d} m_\varepsilon(x) = \exp\left( 2\int_0^x a_\varepsilon(z) \, \mathrm{d} z\right) \mathrm{d} x.
    \]

    Observe that for any fixed $y > 0$, a simple change of variables yields
    \[
    \begin{aligned}
    \lim_{\varepsilon \to 0+} \int_0^y a_\varepsilon(z) \, \mathrm{d} z
    &= \lim_{\varepsilon \to 0+} \int_0^{y/\varepsilon} a(\tilde{z}) \, \mathrm{d} \tilde{z} \\
    &= \int_0^\infty a(\tilde{z}) \, \mathrm{d} \tilde{z},
    \end{aligned}
    \]
    and similarly, for $y < 0$,
    \[
    \begin{aligned}
    \lim_{\varepsilon \to 0+} \int_0^y a_\varepsilon(z) \, \mathrm{d} z
    &= \lim_{\varepsilon \to 0+} \int_0^{y/\varepsilon} a(\tilde{z}) \, \mathrm{d} \tilde{z} \\
    &= \int_0^{-\infty} a(\tilde{z}) \, \mathrm{d} \tilde{z} \\
    &= -\int_{-\infty}^0 a(\tilde{z}) \, \mathrm{d} \tilde{z}.
    \end{aligned}
    \]
    Hence, the assumptions of Theorem \ref{thm:open-infinite-intervals}  are satisfied with $I = \mathbb{R}$, leading to the limiting scale function
    \[
    s_0(x) = \begin{cases}
        x \exp\left( -2\int_0^\infty a(z) \, \mathrm{d} z \right), & x \geq 0, \\
        x \exp\left( 2\int_{-\infty}^0 a(z) \, \mathrm{d} z \right), & x < 0,
    \end{cases}
    \]
    and the limiting speed measure
    \[
    \begin{aligned}
    \mathrm{d} m_0(x)
    = \Big(&\exp\left( 2\int_0^\infty a(z) \, \mathrm{d} z \right) \mathbb{I}_{\{x \geq 0\}} \\
    &+ \exp\left( -2\int_{-\infty}^0 a(z) \, \mathrm{d} z \right) \mathbb{I}_{\{x < 0\}} \Big) \mathrm{d} x.
    \end{aligned}
    \]
    Consequently, if $X_\ve(0)\Rightarrow x,$ then  the limit diffusion $X_0$ is a skew Brownian motion with the parameter $\gamma = \tanh(\alpha)$ and $X_0(0)=x$. The first result of this type appeared in the pioneering paper by N.~Portenko \cite{portenko1976generalized}; see also a generalization of the corresponding result for perturbations of a skew Brownian motion and Walsh Brownian motion in \cite{bobrowski2026portenko}.

\subsection{Penalization and reflected Brownian motion}

Consider a sequence of SDEs 
\[
\di X_n(t) = a_n(X_n(t)) \di t + \di W(t), t\geq 0,
\]
with the same initial condition $X_n(0)=x_0\geq 0.$ Assume that $a_n$ are nonnegative continuous functions that vanish on the
positive half-line and become increasingly large near $0$ on the negative
half-line, thereby pushing the process towards the positive half-line.  It follows from Theorem \ref{thm:regular point} that if   $\limn\int_{-\ve}^{0}a_n(y) \di y=+\infty,$ for any $\ve>0,$ then  the sequence $(X_n)$ converge in distribution to a reflected Brownian motion started from $x_0$; see  \S \ref{subsec:SDEScale} for the corresponding  formulas for the scale function and   speed measure. 

This condition is satisfied , for example,  if   $a_n(x)=n^{1+\delta} a(nx)$, where $\delta>0$ is any positive number and  $a$ is a continuous nonnegative function such that  $a(x)=0, x>0,$ which is non identically to 0. In contrast, the case $\delta=0$ may be insufficient to produce the reflection when  $a$ is integrable. Indeed, in this case the limiting process is a skew Brownian motion; see \S \ref{expl:Portenko}.

\subsection{Accumulation of semipermeable membranes at one point}
 \label{expl:5.1}
    Let $X_n$ be a solution to the SDE
    \[
    \mathrm{d} X_n(t) = \mathrm{d} W(t) + \sum_{k=1}^d \gamma_k \, \mathrm{d} L_{X_n}^{\varepsilon_{n,k}}(t),
    \]
    where $\gamma_k \in (-1,1)$ are the permeability parameters, $W$ is a standard Brownian motion, and $L_{X_n}^{\varepsilon_{n,k}}$ is the symmetric local time of $X_n$ at the membrane position $\varepsilon_{n,k}$. This process describes a Brownian motion with $d$ semipermeable membranes. We assume that for each fixed $n \geq 1$, the membranes are located at distinct points $\varepsilon_{n,1}, \dots, \varepsilon_{n,d}$, and that they all converge to the origin, i.e., $\lim_{n \to \infty} \varepsilon_{n,k} = 0$ for all $1 \leq k \leq d$.

    Setting $\gamma_k = p_k - q_k$ with $p_k + q_k = 1$ (where $p_k, q_k > 0$), the scale function $s_n$ and the speed measure $m_n$ of $X_n$ are determined from the following transmission and continuity conditions (see, e.g., \S \ref{expl:many membranes}):
    \begin{itemize}
        \item The continuous function $s_n$ is linear on each subinterval of $\mathbb{R} \setminus \{\varepsilon_{n,1}, \dots, \varepsilon_{n,d}\}$;
        \item At each membrane $\varepsilon_{n,k}$, the derivatives of the scale function satisfy the condition
        \[
        p_k s_n'(\varepsilon_{n,k}+) = q_k s_n'(\varepsilon_{n,k}-);
        \]
        \item The speed measure $m_n$ is given by $\mathrm{d} m_n(x) = \big(s_n'(x)\big)^{-1} \, \mathrm{d} x$.
    \end{itemize}

    Without loss of generality, we normalize scale functions such that $s_n(0)=0$ and  $s_n'(x)=1$ for all $x$ to the right of all membranes.

    As $n \to \infty$, the sequence of scale functions   $(s_n) $  and speed measures  $(m_n)$ converge to the scale $s_0$ and  the speed measure $m_0$ of a regular diffusion process $X_0$, where $s_0$ is a linear function on $(-\infty,0]$    and $[0,\infty)$, satisfying $p s_0'(0+) = q s_0'(0-)$ where 
    $    \frac{p}{q} = \prod_{k=1}^d \frac{p_k}{q_k}$, and $\mathrm{d} m_0(x) = \big(s_0'(x)\big)^{-1} \, \mathrm{d} x$. Hence Theorem \ref{thm:open-infinite-intervals} implies the following result.
    \begin{thm}\label{thm:5.1}
        Let $X_n$ be the Brownian motion with $d$ semipermeable membranes described above. Assume that $X_n(0) \Rightarrow \xi$ as $n \to \infty$. Then, we have the weak convergence
        \[
        X_n \Rightarrow X_0, \quad n \to \infty,
        \]
        in the space of continuous functions, where $X_0$ is a skew Brownian motion starting from $X_0(0) \overset{d}{=} \xi$, and its skewness parameter $\gamma = p - q \in (-1,1),$ where $p,q>0, p+q=1,$ is determined by the relation
        \[
        \frac{p}{q} = \prod_{k=1}^d \frac{p_k}{q_k} = \frac{p_1 \cdot \dots \cdot p_d}{q_1 \cdot \dots \cdot q_d}.
        \]
        Here $p_k=\frac{1+\gamma_k}{2}$ and $q_k=\frac{1-\gamma_k}{2}$.
    \end{thm}

    \begin{remk}
        By recalling the identity $e^{2x} = \frac{1 + \tanh(x)}{1 - \tanh(x)}$, if we parameterize individual permeabilities as $\gamma_k = \tanh(\alpha_k)$ for $1 \leq k \leq d$, then the limiting parameter $\gamma$ takes the elegant additive form:
        \[
        \gamma = \tanh(\alpha_1 + \dots + \alpha_d).
        \]
         The corresponding result for two membranes was obtained  in \cite{zaitseva2002wiener}; the general case was established in \cite{mazzonetto2016exact, krykun2017convergence}.
         Note also that the passage from the interface conditions \eqref{eq:multi-skewBM-interface conditions} for \(X_\ve\) to the limiting interface condition \eqref{eq:skewBM-interface conditions} for \(X_0\) is considerably more transparent than the formula involving the hyperbolic tangent. 
    \end{remk}
 
\subsection{Two membranes with small permeability. The limiting skew Brownian motion}
 \label{expl: two membranes skew BM}
    Consider a family of stochastic processes $X_\varepsilon$ satisfying the SDE
    \begin{equation}
    \label{eq:SDE_two_membranes}
    \begin{aligned}
    \mathrm{d} X_\varepsilon(t)
    &= \mathrm{d} W(t) + \gamma^\varepsilon_- \, \mathrm{d} L^{X_\varepsilon}_{-\varepsilon/2}(t) \\
    &\quad + \gamma^\varepsilon_+ \, \mathrm{d} L^{X_\varepsilon}_{\varepsilon/2}(t),
    \end{aligned}
    \end{equation}
    where $W$ is a standard Brownian motion, $L^{X_\varepsilon}_{\pm\varepsilon/2}$ is the symmetric local time of $X_\varepsilon$ at $\pm\varepsilon/2$, and the permeability parameters are given by $\gamma^\varepsilon_\pm = p^\varepsilon_\pm - q^\varepsilon_\pm \in (-1,1)$ with $p^\varepsilon_\pm + q^\varepsilon_\pm = 1$. At each membrane $\pm\varepsilon/2$, the derivatives of the scale function $s_\varepsilon$ satisfy the transmission conditions
    \[
    p_{\pm}^\varepsilon s_\varepsilon'((\pm \varepsilon/2)+) = q_{\pm}^\varepsilon s_\varepsilon'((\pm \varepsilon/2)-).
    \]
    This is a particular case of the SDE from \S\ref{expl:5.1}, and the scale functions $s_\varepsilon$ and the speed measures $m_\varepsilon$ can be computed explicitly in the same manner. 
    Assume that $ p^\varepsilon_+ \sim \lambda(\varepsilon)\alpha_+ \quad \text{and} \quad q^\varepsilon_- \sim \lambda(\varepsilon)\alpha_-$, where $\lambda(\varepsilon) > 0$ is such that $\lim_{\varepsilon \to 0+} \lambda(\varepsilon) = 0$. This means that the membranes push the diffusion toward 0 with high probability. The limit transmission condition is $\alpha_+s_0'(0+)=\alpha_- s'_0(0-).$

    If we assume that $\lim_{\varepsilon \to 0+} \frac{\varepsilon}{\lambda(\varepsilon)} = 0$, then it can be easily seen that the limit  speed measure does not accumulate mass at the origin  and the limit speed measure  satisfies relation $\mathrm{d} m_0(x) = \big(s_0'(x)\big)^{-1} \, \mathrm{d} x$. 

A similar situation occurs  if the  membranes push  the diffusion toward 0 with small probabilities proportional to $\lambda(\ve).$ In both cases the limiting process is a skew Brownian motion. We formulate the result as a theorem.    
    \begin{thm}\label{thm:lim-2 membranes-skew BM}
        Assume that a function $\lambda(\varepsilon) > 0, \ve>0$ is such that $\lim_{\varepsilon \to 0+} \lambda(\varepsilon) = 0$, $\lim_{\varepsilon \to 0+} \frac{\varepsilon}{\lambda(\varepsilon)} = 0$, and the parameters $p_\pm^\ve, q_\pm^\ve=1- p_\pm^\ve$ satisfy either of the following conditions
        \[
        p^\varepsilon_+ \sim \lambda(\varepsilon)\alpha_+ \quad \text{and} \quad q^\varepsilon_- \sim \lambda(\varepsilon)\alpha_-,
        \]
        or 
        \[
        q^\varepsilon_+ \sim  \frac{\lambda(\varepsilon)}{\alpha_+} \quad \text{and} \quad p^\varepsilon_- \sim \frac{\lambda(\varepsilon)}{\alpha_-},
        \]
        as $\ve \to0+. $ Here  $\alpha_\pm > 0$ are constants. If $X_\varepsilon(0) \Rightarrow \xi$ as $\varepsilon \to 0+$, then we have the weak convergence
        \[
        X_\varepsilon \Rightarrow X_0, \quad \varepsilon \to 0+,
        \]
        in the space of continuous functions, where $X_0$ is a skew Brownian motion starting from $X_0(0) \overset{d}{=} \xi$. Its skewness parameter $\gamma = p - q \in (-1,1), $ where $p,q>0, p+q=1,$ is equal to
           \[
        \gamma = \frac{\alpha_+ - \alpha_-}{\alpha_+ + \alpha_-}.
        \]
            \end{thm}

        The formula for the limiting parameter $\gamma$ can be interpreted as a limiting case of the product formula from Theorem \ref{thm:5.1}. Indeed, in both regimes   define $p_\varepsilon,q_\varepsilon>0 ,  p_\varepsilon+q_\varepsilon=1,$   by
        \[
        \begin{aligned}
       \frac{p_\ve}{q_\ve}
       &:=\frac{p^\varepsilon_-}{q^\varepsilon_-} \cdot \frac{p^\varepsilon_+}{q^\varepsilon_+} \\
       &\sim   \frac{\alpha_+}{\alpha_-}, \quad \varepsilon \to 0+.
        \end{aligned}
        \]
         Then 
         $$\gamma_\varepsilon := p_\varepsilon - q_\varepsilon = \frac{p_\varepsilon/q_\varepsilon - 1}{p_\varepsilon/q_\varepsilon + 1}
         \to \frac{\frac{\alpha_+}{\alpha_-} - 1}{\frac{\alpha_+}{\alpha_-} + 1} = \frac{\alpha_+ - \alpha_-}{\alpha_+ + \alpha_-},$$
                which matches the assertion of the theorem.

\subsection{Two membranes with small permeability. Limiting  absorption or reflection } \label{expl:unpermeable limit}
    Consider a family of stochastic processes $X_\varepsilon$ satisfying SDE \eqref{eq:SDE_two_membranes}. In contrast to the previous example, we now investigate the regime where the permeability of the membranes vanishes much faster than the distance between them, which corresponds to the condition $\lim_{\varepsilon \to 0+} \frac{\lambda(\varepsilon)}{\varepsilon} = 0$. This configuration leads to a complete decoupling of the state space, resulting in either absorption or reflection at the origin depending on the specific scaling behavior of the parameters. 
    Recall that  $p^\varepsilon_\pm=(1+\gamma^\varepsilon_\pm)/2$  and  $q^\varepsilon_\pm=(1-\gamma^\varepsilon_\pm)/2$.

    \begin{thm}
        Let  $\lambda(\varepsilon) > 0, \ve>0$ be such that $\lim_{\varepsilon \to 0+} \lambda(\varepsilon) = 0$, $\lim_{\varepsilon \to 0+} \frac{\lambda(\varepsilon)}{\varepsilon} = 0$. Assume that  $X_\varepsilon(0) \Rightarrow \xi$ as $\varepsilon \to 0+$.
        \begin{enumerate}
            \item[\rm(i)] If the parameters satisfy $p^\varepsilon_+  \sim  \lambda(\varepsilon)\alpha_+$ and $q^\varepsilon_- \sim \lambda(\varepsilon)\alpha_-$ for some $\alpha_\pm > 0$, then 
            \[
            X_\varepsilon \Rightarrow X_0, \quad \varepsilon \to 0+,
            \]
            where $X_0$ is a Brownian motion absorbed at the origin, starting from $X_0(0) \overset{d}{=} \xi$.
            
            \item[\rm(ii)] If the parameters satisfy $q^\varepsilon_+  \sim  \frac{\lambda(\varepsilon)}{\alpha_+}$ and $p^\varepsilon_-  \sim  \frac{\lambda(\varepsilon)}{\alpha_-}$ for some $\alpha_\pm > 0$ and $\xi\neq 0$ a.s., then 
            \[
            X_\varepsilon \Rightarrow X_0, \quad \varepsilon \to 0+,
            \]
            where $X_0$ is a Brownian motion with reflection at the origin. 
            Conditionally on \(\{\xi>0\}\), \(X_0\) is a Brownian motion reflected on \([0,\infty)\); conditionally on \(\{\xi<0\}\), it is a Brownian motion reflected on \((-\infty,0]\). 
        \end{enumerate}
    \end{thm}

    The proof of this theorem relies on analyzing the asymptotic behavior of the scale functions $s_\varepsilon$ and speed measures $m_\varepsilon$. In case \rm(i), the speed measure of the central layer tends to infinity, i.e., $m_\varepsilon([-\varepsilon/2, \varepsilon/2]) \to \infty$ as $\varepsilon \to 0+$, which turns the origin into an absorbing boundary; the result then follows from Theorem \ref{thm:conv_absorbing}. In case \rm(ii), the scale distance across the central layer tends to infinity, effectively separating the state space into two disconnected  components $(-\infty, 0-]$ and $ [0+, \infty)$, each with a reflecting boundary at $0$, which allows the application of Theorem \ref{thm:regular point}. Note that while Theorems  \ref{thm:conv_absorbing} and \ref{thm:regular point}   were formulated for a half-line $I_0$, their extension to the entire real line $\mathbb{R}$ in this decoupled setting is straightforward and does not alter the convergence argument.

\subsection{Two membranes with small permeability. Skew-sticky limits}

    Consider a family of stochastic processes $X_\varepsilon$ satisfying SDE \eqref{eq:SDE_two_membranes}. We now investigate the critical regime where the permeability parameters vanish at the same rate as the distance between the membranes, i.e., $\lambda(\varepsilon) = \varepsilon$. As we will see, this configuration leads to the appearance of an atom at the origin in the limiting speed measure, which generates a sticky effect.

    Denote $p^\varepsilon_\pm := \frac{1+\gamma_\pm^\varepsilon}{2}$ and $q^\varepsilon_\pm := \frac{1-\gamma_\pm^\varepsilon}{2}$. We assume that for sufficiently small $\varepsilon$, the parameters scale as $p^\varepsilon_+ \sim \alpha_+ \varepsilon$ and $q^\varepsilon_- \sim \alpha_- \varepsilon$ as $\varepsilon\to0+$ with $\alpha_\pm > 0$.

    Let us construct the scale function $s_\varepsilon$ normalized such that $s_\varepsilon(x) = \varepsilon x$ for $|x| \leq \varepsilon/2$. Then, its derivative satisfies $s'_\varepsilon(x) =  \frac{\varepsilon q^\varepsilon_+}{p^\varepsilon_+} \sim \frac{1}{\alpha_+}$ for $x > \varepsilon/2$, and $s'_\varepsilon(x) = \frac{\varepsilon p^\varepsilon_-}{q^\varepsilon_-}\sim \frac{1}{\alpha_-}$ for $x < -\varepsilon/2$. The corresponding speed measure is given by
    \[
    \begin{aligned}
    \mathrm{d} m_\varepsilon(x)
    = \Big(&\frac{ q^\varepsilon_-}{\varepsilon p^\varepsilon_-}\mathbb{I}_{\{x < -\varepsilon/2\}}
      + \frac{1}{\varepsilon} \mathbb{I}_{\{|x| \leq \varepsilon/2\}} \\
    &+ \frac{ p^\varepsilon_+}{\varepsilon q^\varepsilon_+}\mathbb{I}_{\{x > \varepsilon/2\}} \Big) \mathrm{d} x.
    \end{aligned}
    \]

    As $\varepsilon \to 0+$, the limits of $s_\varepsilon$ and $m_\varepsilon$ yield the scale function and the speed measure of a regular diffusion $X_0$ on $\mathbb{R}$:
    \[
    \begin{aligned}
    s_0(x)
    &= \frac{x}{\alpha_-} \mathbb{I}_{\{x < 0\}} + \frac{x}{\alpha_+} \mathbb{I}_{\{x \geq 0\}}, \\
    \mathrm{d} m_0(x)
    &= \left( \alpha_- \mathbb{I}_{\{x < 0\}} + \alpha_+ \mathbb{I}_{\{x \geq 0\}} \right) \mathrm{d} x + \delta_0(\mathrm{d} x),
    \end{aligned}
    \]
    where $\delta_0$ is the unit Dirac mass concentrated at $0$.

    It follows from \eqref{eq:skewstickyscale} that $X_0$ is a   skew-sticky Brownian motion, where $p = \frac{\alpha_+}{\alpha_- + \alpha_+}$ and $q = \frac{\alpha_-}{\alpha_- + \alpha_+}$.   The boundary condition \eqref{eq:sticky_boundary_cond}  takes the form:
    \begin{equation}
        \label{eq:sticky-limit-boundary}
           \frac{\alpha_+}{\alpha_- + \alpha_+} u'(0+) -  \frac{\alpha_-}{\alpha_- + \alpha_+} u'(0-) =   \frac{2}{2(\alpha_- + \alpha_+)}  u''(0).
    \end{equation}
        Theorem \ref{thm:open-infinite-intervals} implies the following   result.

    \begin{thm}
        Let $X_\varepsilon$ be the solution to SDE \eqref{eq:SDE_two_membranes} under the scaling assumptions described above, and let $X_\varepsilon(0) \Rightarrow \xi$ as $\varepsilon \to 0+$. Then, we have the weak convergence
        \[
        X_\varepsilon \Rightarrow X_0, \quad \varepsilon \to 0+,
        \]
        in the space of continuous functions, where the limit process $X_0$ is a skew-sticky Brownian motion $X^{\text{sticky}(\lambda)}_{\text{skew}(\gamma)}$ with the skewness and stickiness parameters given respectively by
        \[
        \gamma =   \frac{\alpha_+ - \alpha_-}{\alpha_- + \alpha_+}, \quad \lambda = \frac{2}{\alpha_- + \alpha_+}.
        \]
    \end{thm}
    \begin{remk}
        If we formally pass to a limit in \eqref{eq:sticky-limit-boundary} as $\alpha_\pm\to \infty$ and $\alpha_+/\alpha_-=p/q=const$, then we obtain boundary conditions for a skew Brownian motion with permeability parameter $\gamma= \frac{\alpha_+-\alpha_-}{\alpha_++\alpha_-}$.This is consistent with Theorem \ref{thm:lim-2 membranes-skew BM}.
    \end{remk}

\subsection{Two membranes with small permeability. Limiting snapping-out Brownian motion}
 
    Consider a family of stochastic processes $X_\varepsilon$ satisfying SDE \eqref{eq:SDE_two_membranes}. We now investigate the critical regime where the membranes become ``tight'' or highly impermeable from the inside of the layer. Specifically, we assume that as $\varepsilon \to 0+$, the permeability parameters inside the layer $[-\varepsilon/2, \varepsilon/2]$ vanish at the rate
    \[
    q^\varepsilon_+ \sim \frac{\varepsilon}{\alpha_+},\quad \text{and} \quad p^\varepsilon_- \sim \frac{\varepsilon}{\alpha_-} \quad \text{as } \ve\to0+,
    \]
    where $\alpha_\pm > 0$. In this configuration, the probabilities of crossing the membranes toward the interior of the layer are small. Recall that if they are not sufficiently small, as seen in \S\ref{expl: two membranes skew BM}, the limit is a skew Brownian motion. Conversely, if they are too small, the limiting process cannot pass through the membranes and we obtain reflection, as shown in \S \ref{expl:unpermeable limit}. The situation in the current example represents an intermediate case between these two regimes.

    Let us construct the scale function $s_\varepsilon$ normalized such that $s'_\varepsilon(x) = \frac{1}{\varepsilon}$ for $|x| \leq \varepsilon/2$ and $s_\varepsilon(0)=0$. Then, for the regions outside the layer, we have the following behavior. If   $x > \varepsilon/2$, then its derivative satisfies $s'_\varepsilon(x) = \frac{1}{\varepsilon} \frac{q^\varepsilon_+}{p^\varepsilon_+} \sim \frac{1}{\alpha_+ }, \ve\to0+$, and  for $x < -\varepsilon/2$ we have  $s'_\varepsilon(x) = \frac{1}{\varepsilon} \frac{p^\varepsilon_-}{q^\varepsilon_-} \sim \frac{1}{\alpha_-  }, \ve\to0+$. Therefore, for any fixed $x \neq 0$: 
    \[
    \lim_{\ve\to0+} s'_\varepsilon(x) = \frac{1}{\alpha_+} \mathbb{I}_{\{x > 0\}} + \frac{1}{\alpha_-} \mathbb{I}_{\{x < 0\}},
    \]
    and  the functions $s_\pm$ from Theorem \ref{thm:limit_SNOdiffosion} are given by 
    \[
    s_-(x) = \frac{x}{\alpha_-} - \frac12, \quad x \leq 0, \quad \text{and} \quad s_+(x) = \frac{x}{\alpha_+} + \frac12, \quad x \geq 0.
    \]

    The corresponding limiting speed measures $m_\pm$ are absolutely continuous with densities satisfying
    \[
    \mathrm{d} m_-(x) = \alpha_- \mathrm{d} x, \quad x \leq 0, \quad \text{and} \quad \mathrm{d} m_+(x) = \alpha_+ \mathrm{d} x, \quad x \geq 0.
    \]

    Applying Theorem \ref{thm:limit_SNOdiffosion}, we obtain the weak convergence of the processes.

    \begin{thm}
        Let $X_\varepsilon$ be the solution to SDE \eqref{eq:SDE_two_membranes} under the scaling assumptions described above, and let $X_\varepsilon(0) \Rightarrow \xi$ as $\varepsilon \to 0+$ with $\xi \neq 0$ a.s. Then, we have the weak convergence
        \[
        X_\varepsilon \Rightarrow X_0, \quad \varepsilon \to 0+,
        \]
        in the space of continuous functions, where the limit process $X_0$ is a snapping-out diffusion with scale functions $s_\pm$, speed measures $m_\pm$, and symmetric transmission intensities
$\lambda_\pm=\frac12$.
    \end{thm}

   \begin{remk}
        The boundary conditions \eqref{eq:SNOD-boundary conditions} for the limiting process $X_0$ take the following form:
    \begin{equation}
    \label{eq:snob_boundary_cond}
    \alpha_+ u'(0+) = \alpha_- u'(0-) =  {u(0+) - u(0-)} .
    \end{equation}
 If we formally pass to a limit as  $\alpha_\pm\to\infty$ and $\alpha_+/\alpha_-=p/q=const,$ then the limiting boundary conditions will become boundary conditions for a skew Brownian motion with $\gamma=\frac{\alpha_+-\alpha_-}{\alpha_++\alpha_-}$. It was proved in \cite{bobrowski2025snapping} that the corresponding snapping-out Brownian motions indeed converge to the skew Brownian motion, see also \S \ref{subsection:from SNOB to skew} below.
       \end{remk}
       \begin{remk}
           It was mentioned in Remark \ref{remk:SNOD-different representation} that a snapping-out diffusion can be represented in various ways. Formulas from  Remark \ref{remk:SNOD-different representation} yield  that  the process $X_0$ is a snapping-out Brownian motion with transmission intensities $(2\alpha_\pm)^{-1}.$ This also agrees with  boundary conditions \eqref{eq:SNOD-boundary conditions} from Remark \ref{remk:boundary SNOD}.
       \end{remk}
       \begin{remk}
    In \cite{lejay2016snapping} the snapping-out Brownian motion was considered for $\lambda_-=\lambda_+$ and the symmetric local time rather than the local time associated with the scale function and the speed measure of the diffusion. Hence, our formulas differ from  \cite{lejay2016snapping} by a multiplicative factor of 2. Moreover, in \cite{lejay2016snapping}, when the local time reaches an
exponentially distributed random level, the snapping-out diffusion selects
its sign with probability $1/2$. This should be taken into account when comparing results. It can be seen that the    process $Y^\varepsilon$ from \cite[Proposition 4]{lejay2016snapping} is a  particular case  of SDE \eqref{eq:SDE_two_membranes}. The limiting process  and the  transmission conditions at 0 coincide with ours, up to the corresponding normalization. 
\end{remk}

\subsection{Approximation of   snapping-out Brownian motion by SDEs with drifts repelling from the origin} 
Consider a sequence of SDEs 
\[
\di X_n(t) = a_n(X_n(t)) \di t + \di W(t), t\geq 0,
\]
where $X_n(0)=x_0\neq 0.$

Assume that $a_n(x)=\alpha_n a (nx),$ where  $a$ is an odd piecewise continuous function  vanishing outside of $[-1,1]$ and $a(x)>0, x\in(0,1).$ 
Then the function $s_n(x)=-s_n(-x)$,
$$
s_n(x)=
\begin{cases}
    \int_0^x \exp{\Big(2\int_{y}^{1/n} a_n(z) \di z}\Big) \di y, \ x\in[0,1/n];\\
    \int_0^{1/n} \exp{\Big(2\int_{y}^{1/n} a_n(z) \di z}\Big) \di y +(x-1/n), \ x>1/n,
\end{cases}
$$
is the  scale function of $X_n.$  

Denote $A(x):=\int_x^1 a(v) \di v$. Then 
\[
s_n(1/n)= \int_0^{1/n} \exp{\Big(2\int_{y}^{1/n} a_n(z) \di z}\Big) \di y =
\int_0^{1/n} \exp{\Big(2\int_{y}^{1/n} \alpha_n a(nz) \di z}\Big) \di y =
\]
\[
\int_0^{1/n} \exp{\Big(2\int_{ny}^{1 } n^{-1}\alpha_n a(v) \di v}\Big) \di y =
\int_0^{1 } n^{-1} \exp{\Big(2\int_{u}^{1 } n^{-1}\alpha_n a(v) \di v}\Big) \di u 
\]
\[
=\int_0^{1 } n^{-1} \exp \Big(2 n^{-1}\alpha_n A(u) \Big) \di u.
\]
Select $(\alpha_n)$ such that $\limn \int_0^{1 } n^{-1} \exp \Big(2 n^{-1}\alpha_n A(u) \Big) \di u=h>0.$ Then functions $s_\pm$ and measures $m_\pm$ from Theorem \ref{thm:limit_SNOdiffosion} are $s_+(x)=x+h, x>0,\quad s_-(x)=x-h, x<0,$   $m_+(\di x)=\1_{x>0} \di x,$ and $ m_-(\di x)=\1_{x<0} \di x.$ It follows from Theorem \ref{thm:limit_SNOdiffosion} that   $(X_n)$ converges in distribution to the SNOB process starting from $x_0$ with transmission intensities $\lambda_\pm=\frac{1}{4h}$.   Note that the asymptotics of the Laplace integral $\int_0^1 \exp \Big(\beta A(u) \Big) \di u$ as $\beta\to+\infty$ is well known for many functions. For example, if $a(x)=1, x\in(0,1),$ then 
$$
\int_0^1 \exp \Big(\beta A(u) \Big) \di u= \int_0^1 \exp \Big(\beta v \Big) \di v = \frac{e^\beta-1}{\beta}\sim e^\beta/\beta, \quad \beta\to+\infty.
$$
Condition  $\limn \int_0^{1 } n^{-1} \exp \Big(2 n^{-1}\alpha_n A(u) \Big) \di u=h$ is satisfied if $\limn \frac{e^{\frac{2\alpha_n}{n}}}{2\alpha_n}=h. $
Sufficient condition is, for example, $\alpha_n=2^{-1} n \ln (h n \ln n);$ see also \cite{mandrekar2016brownian}.

\subsection{Approximation of a skew Brownian motion by a sequence of  SNOB processes}\label{subsection:from SNOB to skew}
 Let $(X_n)$ be  a sequence of snapping-out Brownian motions with transmission intensities $\lambda^n_{\pm}=c_n /\alpha_\pm$, where $\alpha_\pm>0$ and  $\limn c_n=+\infty.$ Assume that they have representation $X_n(t)=s_n^{-1}(B(A_n^{-1}(t)))$, see \eqref{eq:SNOB_s,m} in  Remark \ref{remk:SNOB construction},  with the same Brownian motion, where
\[
\begin{aligned}
 s_n^{-1}(x)
 &=\frac{\alpha_-}{\alpha_+}(x+\frac{\alpha_+}{2 c_n})\1_{x<-\frac{\alpha_+}{2c_n}}+  x\1_{x>0}, \\
 \wt m_n(\di x)
 &= (\frac{\alpha_-^2}{\alpha_+^2} \1_{x<-\frac{\alpha_+}{2c_n}}+ \1_{x>0}) \di x.
\end{aligned}
\]
 It can be seen that $\limn s_n^{-1}(x)=\frac{\alpha_-}{\alpha_+} x \1_{x<0}  +  x\1_{x>0}=:s^{-1}(x)$ locally uniformly, where $s(x)= \frac{\alpha_+}{\alpha_-} x \1_{x<0}  +  x\1_{x>0},$ and $(\wt m_n)$ converges vaguely to $\wt m(\di x)=(\frac{\alpha_-^2}{\alpha_+^2} \1_{x<0}+ \1_{x>0}) \di x$. Hence, by Corollary \ref{corl:integrals}, $\limn A_n(t)=A(t):=\int_\mbR l_B(t,y)\wt m (\di y)= \int_\mbR l_B(t,s(x))  m (\di x),$ locally uniformly,  where 
 $m (\di x)= (\frac{\alpha_- }{\alpha_+ } \1_{x<0}+ \1_{x>0}) \di x$. Since $A$ is strictly increasing,  Corollary \ref{corl:5.1} yields locally uniform  convergence of inverses $\limn A_n^{-1}(t)= A^{-1}(t)$. Therefore, $\limn X_n(t)=s^{-1}(B(A^{-1}(t)))$ locally uniformly, and the limiting process is a skew Brownian motion with  $\gamma=\frac{\alpha_+-\alpha_-}{\alpha_++\alpha_-}$. The corresponding result was obtained in
\cite{bobrowski2025approximation} by analytic methods.

\section{Appendix}\label{sec:Appendix}
Here we collect well-known analytical facts that are frequently used throughout the paper. 

\subsection{Convergence of functions}
\begin{lem}\cite[\S 12.5]{Whitt}\label{lem: conv monot cont}
Let $f_n:[a,b] \to \mbR$, $n \geq 0$, be a sequence of non-decreasing càdlàg functions. Then $\lim_{n \to \infty} f_n(t) = f_0(t)$ for all $t$ belonging to a dense subset of $[a,b]$ including $a$ and $b$ if and only if $f_n \to f_0$ as $n \to \infty$ in Skorokhod's $M_1$ topology.

Moreover, if $f_0$ is continuous, then the convergence is uniform, i.e., $\lim_{n \to \infty} \sup_{t \in [a,b]} |f_n(t) - f_0(t)| = 0$.
\end{lem}

\begin{remk}
For the precise definition of $M_1$-convergence, we refer the reader to \cite{Whitt}. In this paper, we do not directly employ the definition of $M_1$-convergence, 
but only use some of its standard properties in the lemmas below.
\end{remk}

Recall the definition of the generalized inverse function. If $f:[a,b] \to [-\infty, \infty]$ is a non-decreasing function, where $-\infty \leq a < b \leq \infty$, then its generalized inverse $f^{-1}:[-\infty, \infty] \to [a,b]$ is defined by 
\[
f^{-1}(t) := \inf\{s \in [a,b] : f(s) > t\}.
\]
We set $f^{-1}(t) := b$ if the set $\{s \in [a,b] : f(s) > t\}$ is empty.  The set $[-\infty,\infty]$ is considered as  a compactification of $\mbR$.     The essential properties of generalized inverses are summarized in the following lemma.

\begin{lem}\label{lem: inverses} Let $f:[a,b] \to \mbR$ be a  non-decreasing càdlàg function. 
\begin{enumerate}
    \item[(1)] The generalized inverse $f^{-1}$ is  a càdlàg function  with values in $[a,b]$.
    \item[(2)] If $f$ is continuous and strictly increasing on $[a,b]$, then $f^{-1}(x)$ for $x \in [f(a), f(b)]$ coincides with its usual inverse.
    \item[(3)] A non-decreasing function $f$ is continuous on $[a,b] \subset \mbR$ if and only if $f^{-1}$ is strictly increasing on $[f(a), f(b)]$. 
        \item[(4)] A non-decreasing function  $f$  is strictly increasing  on $[a,b] \subset \mbR$ if and only if $f^{-1}$ is continuous on $[f(a), f(b)]$. 
    \item[(5)] Let $f_n:[0,T] \to \mbR$, $n \geq 0$, be a sequence of non-decreasing càdlàg functions. Assume that $t_0\in (0,T) $ is a point of increase of $f_0$, by which we mean that $f_0(t_-) < f_0(t_0) < f_0(t_+)$ whenever $0\leq t_- < t_0 < t_+\leq T$. Suppose that $\lim_{n \to \infty} f_n(t) = f_0(t)$ for all $t$ from a dense set in a neighborhood of $t_0$. Then $\lim_{n \to \infty} f_n^{-1}(z_0) = f_0^{-1}(z_0)$, where $z_0 = f_0(t_0)$.  
    \item[(6)] Let $f_n:[0,\infty) \to [0,\infty)$, $n \geq 0$, be a sequence of non-decreasing and unbounded càdlàg functions. Then
\[
\begin{aligned}
&\lim_{n \to \infty} f_n = f_0
    &&\text{in the $M_1$-topology in $D((0,\infty))$} \\
&\Longleftrightarrow\quad
\lim_{n \to \infty} f_n^{-1} = f_0^{-1}
    &&\text{in the $M_1$-topology on $D((0,\infty))$} \\
&\Longleftrightarrow\quad
\limn f_n(x)=f_0(x)
    &&\text{for all $x$ in a dense subset of $(0,\infty)$} \\
&\Longleftrightarrow\quad
\limn f_n^{-1}(x)=f_0^{-1}(x)
    &&\text{for all $x$ in a dense subset of $(0,\infty)$.}
\end{aligned}
\]
\end{enumerate}
\end{lem}
All assertions except for the last one follow directly from the definition and simple analytic arguments. The last assertion follows from \cite[Theorem 13.6.3]{Whitt}.

\begin{remk} 
Even though we haven't defined convergence in $M_1$-topology, note that  by definition, the convergence $f_n\to f $ in  $D((0,\infty))$  means convergence $f_n\to f$ in  $D([a,b])$  in $M_1$-topology for   any $[a,b]\subset (0,\infty)$. 
In addition, a sequence $(f_n)$ is $M_1$-convergent to $f_0$ in $D([0,\infty))$ if and only if $(f_n)$ is $M_1$-convergent to $f_0$ in $D([0,T])$ for almost all $T > 0$ with respect to the Lebesgue measure, see \cite[p. 83]{Whitt}.  
\end{remk}
\begin{remk}
The possibility of including $\pm\infty$ into the domain of definition or the range of values does not alter the proofs from \cite{Whitt},   as there exists a strictly increasing homeomorphism  from $[-\infty, \infty]$ onto $[0,1]$. 
\end{remk}

 \begin{corl}\label{corl:5.1}
Assume that a sequence $(f_n):[0,T]\to[-\infty,\infty]$ of
nondecreasing càdlàg functions satisfies
\[
\lim_{n\to\infty} f_n(t)=f_0(t)
\]
for all $t$ in a dense subset of $[0,T]$.

\begin{enumerate}

\item
Assume that $0<a<b<T$, and that $a$ and $b$ are points of increase
of $f_0$. Set
\[
\alpha:=f_0(a), \qquad \beta:=f_0(b).
\]
Then $f_n^{-1}\to f_0^{-1}$ in the $M_1$-topology on
$D([\alpha,\beta])$.

\item
Assume that $f_0^{-1}$ is finite and continuous on an open interval
$J\subset\mbR$. Then we have uniform convergence of the inverses:
\[
\lim_{n\to\infty}
\sup_{t\in[\alpha,\beta]}
|f_n^{-1}(t)-f_0^{-1}(t)|=0
\]
for every compact interval $[\alpha,\beta]\subset J$.

In particular, if $f_0$ is strictly increasing on
$[a,b]\subset[0,T]$, then the above convergence holds for every
\[
[\alpha,\beta]\subset (f_0(a),f_0(b-)).
\]

\item
Assume that $f_0$ is strictly increasing on $[0,T]$ and
$f_n(0)=0$ for all $n\geq0$. Then we have uniform convergence of
the inverses:
\[
\lim_{n\to\infty}
\sup_{t\in[0,\beta]}
|f_n^{-1}(t)-f_0^{-1}(t)|=0
\]
for any $0<\beta<f_0(T-)$.

\end{enumerate}
\end{corl} 
\begin{proof}
For assertion (1), Lemma \ref{lem: inverses}(5), together with the
preceding remark, yields the convergence of the inverse functions at
the endpoints $\alpha$ and $\beta$. 
At every continuity point $y\in(\alpha,\beta)$ of $f_0^{-1}$,
the convergence
\[
f_n^{-1}(y)\to f_0^{-1}(y)
\]
follows directly from the definition of the generalized inverse and
the convergence of $f_n$ on a dense set. Since the continuity points
of the nondecreasing function $f_0^{-1}$ form a dense set, Lemma
\ref{lem: conv monot cont} yields the $M_1$-convergence in (1).

To prove assertion (2), fix $y\in J$ and set $x=f_0^{-1}(y)$. By the continuity
of $f_0^{-1}$ at $y$, one can choose points $x_-<x<x_+$ from the
dense set of convergence, arbitrarily close to $x$, such that
\[
f_0(x_-)<y<f_0(x_+),
\]
with the obvious one-sided modification if $x=0$ or $x=T$.
Hence, for all sufficiently large $n$,
\[
f_n(x_-)<y<f_n(x_+),
\]
and therefore
\[
x_-\leq f_n^{-1}(y)\leq x_+.
\]
Thus $f_n^{-1}(y)\to f_0^{-1}(y)$. Lemma
\ref{lem: conv monot cont} then yields uniform convergence on every
compact subinterval of $J$.

Finally, (3) follows from (2). Indeed, $f_n(0)=0$ and monotonicity
imply $f_0(0)\geq0$. Since $f_0$ is strictly increasing,
$f_0^{-1}$ is finite and continuous on an open interval containing
$[0,\beta]$ for every $\beta<f_0(T-)$.
\end{proof}

\begin{lem}\label{lem:conv compositions}
Let $(f_n)_{n \geq 0}:[c, d] \to \mbR$ and $(g_n)_{n \geq 0}:[a,b] \to [c,d]$ be sequences of functions such that $\lim_{n \to \infty} f_n = f_0$ and $\lim_{n \to \infty} g_n = g_0$  uniformly on the corresponding sets. Assume that  $f_0$ and $g_0$ are continuous. Then we have the uniform convergence of compositions $\lim_{n \to \infty} f_n \circ g_n = f_0 \circ g_0$.
\end{lem}

For any function $f$, denote   $\sigma_x(f) := \inf\{t \geq 0 : f(t) = x\} \in [0, \infty]$.

\begin{lem}\label{lem:hitting_limit}
Let $(f_n)_{n \geq 0}: [0,\infty) \to \mbR$ be a sequence of continuous functions converging locally uniformly to $f_0$, and let $(x_n)_{n \geq 0} \subset \mbR$ be a sequence of numbers such that $\lim_{n \to \infty} x_n = x_0$. Assume that $\sigma_{x_0}(f_0) = \infty$ or that 
\[
\begin{aligned}
\forall \ve > 0, \ \exists t_\pm \in
(\sigma_{x_0}(f_0), \sigma_{x_0}(f_0) + \ve)
\quad \text{such that } \\
f_0(t_+) > x_0 \text{ and } f_0(t_-) < x_0.
\end{aligned}
\]
Then $\lim_{n \to \infty} \sigma_{x_n}(f_n) = \sigma_{x_0}(f_0)$.
\end{lem}

\subsection{Vague convergence and its properties}

Let $m_n$, $n \geq 0$, be  locally finite measures on   open intervals $I_n = (a_n, b_n)$, where $-\infty \leq a_n < b_n \leq \infty$.    Assume that the interval $I = (a,b)$ with $-\infty \leq a < b \leq \infty$ is contained in $I_0$ and any compact $K\subset I$ is contained in $I_n$ for all sufficiently large $n$. We denote by $C_0(I)$ the space of continuous functions $f:I \to \mbR$ with compact support $\mathrm{supp}(f) \subset I$. 

\begin{defn}
The sequence of measures $(m_n)$ converges vaguely to a measure $m_0$ on $I$ if $\lim_{n \to \infty} \int_I f \, \di m_n = \int_I f \, \di m_0$ for any $f \in C_0(I)$.
\end{defn}

Let $x_0 \in I$ be an arbitrary point. Introduce the distribution-like functions:
\[
F_n(x) := \begin{cases}
    m_n((x_0, x]), & x \geq x_0, \ x \in I, \\
    -m_n((x, x_0]), & x < x_0, \ x \in I.
\end{cases}
\]
It is well known that $m_n \to m_0$ vaguely if and only if $\lim_{n \to \infty} F_n(x) = F_0(x)$ for all $x \in I$ at which $F_0$ is continuous.

\begin{remk}\label{remk: vague}
The vague convergence $m_n \to m_0$ holds if $\lim_{n \to \infty} F_n(x) = F_0(x)$ for all $x$ in a dense subset $J \subset I$ where $F_0$ is continuous. Indeed, let  $x \in I \setminus J$ be a point of continuity of $F_0. $ Choose sequences $x_m^- < x < x_m^+$ in $J$ such that $\lim_{m \to \infty} F_0(x_m^\pm) = F_0(x)$. Then 
\[
\begin{aligned}
F_0(x_m^-)
&= \lim_{n \to \infty} F_n(x_m^-)
\leq \liminf_{n \to \infty} F_n(x) \\
&\leq \limsup_{n \to \infty} F_n(x)
\leq \lim_{n \to \infty} F_n(x_m^+)
= F_0(x_m^+)
\end{aligned}
\]
for any $m \geq 1$. Taking the limit as $m \to \infty$ yields $\lim_{n \to \infty} F_n(x) = F_0(x)$.
\end{remk}

\begin{lem}\label{lem:conv of integrals}
Assume that a sequence of locally finite measures $(m_n)$ converges vaguely on $I$ to a measure $m_0$. Suppose that a sequence of continuous functions $(f_n)$ converges uniformly to $f_0$, and all functions $f_n, n\geq 0,$ vanish outside of a common compact set $K \subset I$. Then $\lim_{n \to \infty} \int_I f_n \, \di m_n = \int_I f_0 \, \di m_0$.
\end{lem}

\begin{proof}
It follows from vague convergence that $\sup_{n \geq 0} m_n(K) < \infty$. Thus,
\[
\begin{aligned}
\limsup_{n \to \infty} \left| \int_I f_n \, \di m_n - \int_I f_0 \, \di m_0 \right|
&\leq \lim_{n \to \infty} \left( \sup_{x \in K} |f_n(x) - f_0(x)| \sup_{n \geq 0} m_n(K) \right) \\
&\quad + \lim_{n \to \infty} \left| \int_I f_0 \, \di m_n - \int_I f_0 \, \di m_0 \right| \\
&= 0.
\end{aligned}
\]
\end{proof}

\begin{corl}
Assume that a sequence of locally finite measures $(m_n)$ converges vaguely on $I$ to a measure $m_0$. Suppose that a sequence $(s_n)$ of nondecreasing continuous functions converges pointwise to a strictly increasing continuous function $s_0$ on an interval $I \subset I_0$. Set $J = s_0(I)$. Then for every $f \in C_0(J)$, we have
\[
\lim_{n \to \infty} \int_I f(s_n) \, \di m_n = \int_I f(s_0) \, \di m_0,
\]
which implies that the push-forward measures $m_n \circ s_n^{-1}$ converge vaguely to $m_0 \circ s_0^{-1}$ on $J$.
\end{corl}

The proof follows directly from Lemmas \ref{lem: conv monot cont} and \ref{lem:conv of integrals} by noting that there exists a compact set $K \subset I$ such that $\mathrm{supp}(f \circ s_n) \subset K$ for all sufficiently large $n$.

\begin{corl}\label{corl:integrals} 
Let $(m_n)$ be a sequence of measures converging vaguely to $m_0$ on $\mbR$,
and let $(y_n)$ be a sequence of real numbers such that
$\lim_{n \to \infty} y_n = 0$.
Let $l=l(t,y):[0,\infty)\times\mbR\to\mbR$ be a non-negative continuous
function such that $t\mapsto l(t,y)$ is nondecreasing for every $y\in\mbR$.
Assume that for any $T>0$, there exists an interval $[a_T,b_T]$ such that
$l(t,y)=0$ if $t\in[0,T]$ and $y\notin[a_T,b_T]$.
Set
\[
A_n(t):=\int_{\mbR} l(t,x+y_n)\,m_n(\di x),\qquad t\geq0.
\]
Then $A_n$ are continuous functions and $A_n\to A_0$ locally uniformly.
\end{corl}

The proof follows from Lemmas \ref{lem: conv monot cont} and \ref{lem:conv of integrals}.

\section*{Acknowledgements}

The authors are  grateful to Adam Bobrowski for helpful discussions.
The authors were partially supported by the Swiss National Science
Foundation, grants No. IZRIZ0\_226875, No. 200020\_214819, No. 10009199,
No. 200020\_192129.


\end{document}